\documentclass[10pt,reqno]{amsart}
\include{formatAndDefs}
\usepackage{amsfonts,amssymb,amsmath}

\parskip        \baselineskip
\usepackage{aliascnt}
\usepackage[mathscr]{eucal}
\usepackage{graphicx}
\usepackage{latexsym}
\usepackage{psfrag}
\usepackage{epsfig}
\usepackage[utf8]{inputenc}
\usepackage[english]{babel}
\usepackage{tikz}
\usepackage[foot]{amsaddr}
\usetikzlibrary{decorations.pathreplacing}
\newcommand{\suchthat}{\;\ifnum\currentgrouptype=16 \middle\fi|\;}
\numberwithin{equation}{section}
\newtheorem{thm}{Theorem}
\numberwithin{thm}{section}

\newaliascnt{lemma}{thm}
\newtheorem{lem}[lemma]{Lemma}
\aliascntresetthe{lemma}

\newaliascnt{proposition}{thm}
\newtheorem{prop}[proposition]{Proposition}
\aliascntresetthe{proposition}

\newaliascnt{corollary}{thm}
\newtheorem{corollary}[corollary]{Corollary}
\aliascntresetthe{corollary}

\newaliascnt{definition}{thm}
\newtheorem{mydef}[definition]{Definition}
\aliascntresetthe{definition}

\newaliascnt{remark}{thm}
\newtheorem{remark}[remark]{Remark}
\aliascntresetthe{remark}

\usepackage[colorlinks=true]{hyperref}

\begin{document}
\title{Local existence of Strong solutions for a fluid-structure interaction model}

\date{\today}

\author{Sourav Mitra}
\thanks{{Acknowledgments}: The author wishes to thank the ANR project ANR-15-CE40-0010 IFSMACS as well as the Indo-French Centre for Applied Mathematics (IFCAM) for the funding provided during this work.}
\address{Sourav Mitra, Institut de Mathématiques de Toulouse; UMR5219; Université de Toulouse; CNRS; UPS IMT, F-31062 Toulouse Cedex 9, France.}
\email{Sourav.Mitra@math.univ-toulouse.fr}

	\begin{abstract}
		We are interested in studying a system coupling the compressible Navier-Stokes equations with an elastic structure located at the boundary of the fluid domain. Initially the fluid domain is rectangular and the beam is located on the upper side of the rectangle. The elastic structure is modeled by an Euler-Bernoulli damped beam equation. We prove the local in time existence of strong solutions for that coupled system. 
		\end{abstract}
\maketitle	
\section{Introduction}
\subsection{Statement of the problem}
 Our objective is to study a fluid structure interaction problem in a 2d channel. The fluid flow here is modeled by the compressible Navier-Stokes equations. Concerning the structure we will consider an Euler-Bernoulli damped beam located on a portion of the boundary. As remarked in \cite{gavalos1}, such dynamical models arise in the study of many engineering systems (e.g., aircraft, bridges etc). In the present article we establish a result on the local in time existence of strong solutions of such a fluid structure interaction problem. To the best of our knowledge, this is the first article dealing with the existence of local in time strong solutions for the complete non-linear model considered here.\\
  We consider data and solutions which are periodic in the \textquoteleft channel direction\textquoteright\, (with period $L,$ where $L>0$ is a constant). Here $L$-periodicity of a function $f$ (defined on $\mathbb{R}$) means that $f(x+L)=f(x)$ for all $x\in \mathbb{R}.$\\ 
We now define a few notations. Let $\Omega$ be the domain $\mathbb{T}_{L}\times (0,1)\subset {\mathbb{R}}^{2},$ where $\mathbb{T}_{L}$ is the one dimensional torus identified with $(0,L)$ with periodic conditions. The boundary of $\Omega$ is denoted by $\Gamma$. We set 
       \begin{equation}\nonumber
       \begin{array}{l}
       \Gamma_{s}=\mathbb{T}_{L}\times \{1\},
       \quad
       \Gamma_{\ell}=\mathbb{T}_{L}\times\{0\},\quad\Gamma=\Gamma_{s}\cup\Gamma_{\ell}.
       \end{array}
       \end{equation}     
Now for a given function $$\eta : \Gamma_{s}\times (0,\infty)\rightarrow (-1,\infty),$$ which will correspond to the displacement of the one dimensional beam, let us denote by $\Omega_{t}$ and $\Gamma_{s,t}$ the following sets
\begin{equation}\nonumber
\begin{array}{ll}
\Omega_{t}=\{(x,y) \suchthat x\in (0,L),\quad 0<y<1+\eta(x,t)\}&=\,\mbox{domain of the fluid at time $t$},
\vspace{1.mm}\\
\Gamma_{s,t}=\{(x,y)\suchthat x\in (0,L),\quad y=1+\eta(x,t)\}&=\,\mbox{the beam at time $t$}.
\end{array}
\end{equation}    
 The reference configuration of the beam is $\Gamma_{s},$ and we set
\begin{equation}\label{setnot}
	\begin{array}{ll}
	\Sigma_{T}=\Gamma\times (0,T),&\quad\Sigma^{s}_{T}=\Gamma_{s}\times (0,T),\\ \widetilde{\Sigma^{s}_{T}}=\cup_{t\in (0,T)}\Gamma_{s,t}\times \{t\},&\quad\Sigma^{\ell}_{T}=\Gamma_{\ell}\times (0,T),\\  Q_{T}=\Omega\times (0,T),&\quad \widetilde{Q_{T}}=\cup_{t\in(0,T)} \Omega_{t}\times \{t\}.
	\end{array}
    \end{equation}  
     \begin{figure}[h!]
     	\centering
     	\begin{tikzpicture}[scale=0.75]
     	\coordinate (O) at (0,5);
     	\coordinate (A) at (10,5);
     	\coordinate (B) at (6.5,5);
     	\draw (5,0)node [below] {$\Gamma_{\ell}$};
     	\draw (0,0)node [below] {$0$};
     	\draw (10,0)node [below right] {$L$};
     	\draw (0,5)node [above] {$1$};
     	%\draw (10,2.5)node [above right] {$\Gamma_d$};
     	%\draw (0,2.5) node [above left] {$\Gamma_g$};
     	%\draw (5,2.5)  node [above]      {$\Gamma=\Gamma_{s}\cup \Gamma_{l}$};
     	%\draw (5,2.5)  node [below]      {$\Gamma_{0}=\Gamma\setminus\Gamma$}; 
     	\draw (3.2,5.2) node [above right] {$\eta(x,t)$};
     	\draw (5,5) node [below] {$\Gamma_{s}$};
     	\draw (0,0) -- (10,0);
     	\draw (10,0) -- (10,5);
     	\draw[dashed] (10,5) -- (0,5);
     	\draw (0,5) -- (0,0);
     	\draw[color=blue, ultra thick] (O) to [bend left=30] (B);
     	\draw[color=blue, ultra thick] (B)  to [bend right=30] (A);
     	\draw[<->, color=red, thick] (3.2,5) -- (3.2,6);
     	\end{tikzpicture}
     	\caption{Domain $\Omega_{t}$.}
     \end{figure} \\ 
   We consider a fluid with density $\rho$ and velocity ${\bf u}.$ The fluid structure interaction system coupling the compressible Navier-Stokes and the Euler-Bernoulli damped beam equation is modeled by
   \begin{equation}\label{1.1}
   \left\{
   \begin{array}{ll}
   \rho_{t}+\mbox{div}(\rho {\bf u})=0\quad &\mbox{in} \quad \widetilde{Q_{T}},
   \vspace{1.mm}\\
   (\rho {\bf u}_{t}+\rho({\bf u}.\nabla){\bf u})-(2\mu \mbox{div} (D({\bf u}))+\mu{'}\nabla\mbox{div}{\bf u}) +\nabla p(\rho) =0\quad &\mbox{in} \quad \widetilde{Q_T},
   \vspace{1.mm}\\
   {\bf u}(\cdot,t)=(0,\eta_{t})\quad & \mbox{on}\quad \widetilde{\Sigma^{s}_{T}},
   \vspace{1.mm}\\
   {\bf u}(\cdot,t)=(0,0)\quad &\mbox{on}\quad \Sigma^{\ell}_{T},
   \vspace{1.mm}\\
   {\bf u}(\cdot,0)={\bf u}_{0}\quad& \mbox{in} \quad \Omega_{\eta(0)}=\Omega,
   \vspace{1.mm}\\
   \rho(\cdot,0)=\rho_{0}\quad &\mbox{in}\quad \Omega_{\eta(0)}=\Omega,
   \vspace{1.mm}\\
   \eta_{tt}-\beta \eta_{xx}-  \delta\eta_{txx}+\alpha\eta_{xxxx}=(T_{f})_{2} \quad& \mbox{on}\quad \Sigma^{s}_{T},
   \vspace{1.mm}\\
   \eta(\cdot,0)=0\quad \mbox{and}\quad \eta_{t}(\cdot,0)=\eta_{1}\quad &\mbox{in}\quad \Gamma_{s}.
   \end{array} \right.
   \end{equation} 
 The initial condition for the density is assumed to be positive and bounded. We fix the positive constants $m$ and $M$ such that %$$\rho_{0}=\rho_{0}(x,y)>0.$$ 
 \begin{equation}\label{cor0}
 \begin{array}{l}
 0<m=\underset{\overline{\Omega}}{\min} \rho_{0}(x,y),\quad M=\underset{\overline{\Omega}}{\max} \rho_{0}(x,y).
 \end{array}
 \end{equation}
 In our model the fluid adheres to the plate and is viscous. This implies that the velocities corresponding to the fluid and the structure coincide at the interface and hence the condition \eqref{1.1}$_{3}$ holds.
    In the system \eqref{1.1}, $D({\bf u})=\frac{1}{2}(\nabla {\bf u}+\nabla^{T} {\bf u})$ is the symmetric part of the gradient and the real constants $\mu,$ $\mu{'}$ are the Lam\'{e} coefficients which are supposed to satisfy
    $$\mu>0,\quad {\mu}{'}\geqslant 0.$$ 
    In our case the fluid is isentropic i.e. the pressure $p(\rho)$ is only a function of the fluid density $\rho$ and is given by
    $$p(\rho)=a\rho^{\gamma},$$
    where $a>0$ and $\gamma>1$ are positive constants.\\ 
    We assume that there exists a constant external force ${p_{ext}}>0$ which acts on the beam. The external force ${p_{ext}}$ can be written as follows
    $${p_{ext}}=a\overline{\rho}^{\gamma},$$
    for some positive constant $\overline{\rho}.$\\
    To incorporate this external forcing term ${p_{ext}}$ into the system of equations \eqref{1.1}, we introduce the following 
    \begin{equation}\label{1.2}
    P(\rho)=p(\rho)-{p_{ext}}=a\rho^{\gamma}-a\overline{\rho}^{\gamma}.
    \end{equation}
    Since $\nabla p(\rho)=\nabla P(\rho),$ from now onwards we will use $\nabla P(\rho)$ instead of $\nabla p(\rho)$ in the equation \eqref{1.1}$_{2}.$\\
    In the beam equation the constants, $\alpha>0,$ $\beta\geqslant0$ and $\delta>0$ are respectively the adimensional rigidity, stretching and friction coefficients of the beam. The non-homogeneous source term of the beam equation $(T_{f})_{2}$ is the net surface force on the structure which is the resultant of force exerted by the fluid on the structure and the external force ${p_{ext}}$ and it is assumed to be of the following form
\begin{equation}\label{1.3}
(T_{f})_{2}=([-2\mu D({\bf u})-\mu{'}(\mbox{div}{\bf u}){\bf I}_{d}]\cdot {{\bf n}_{t}}+P{{\bf n}_{t}})\mid_{\Gamma_{s,t}}\sqrt{1+\eta^{2}_{x}}\cdot \vec{e}_{2}\quad\mbox{on}\quad \Sigma^{s}_{T},
\end{equation}
where ${\bf I}_{d}$ is the identity matrix, ${\bf n}_{t}$ is the outward unit normal to $\Gamma_{s,t}$ given by
$${\bf n}_{t}=-\frac{\eta_{x}}{\sqrt{1+\eta^{2}_{x}}}\vec{e}_{1}+\frac{1}{\sqrt{1+\eta^{2}_{x}}}\vec{e}_{2}$$
($\vec{e}_{1}=(1,0)$ and $\vec{e}_{2}=(0,1)$).\\
Observe that $(\rho,{\bf u},\eta)=(\overline{\rho},0,0)$ is a stationary solution to \eqref{1.1}-\eqref{1.2}-\eqref{1.3}.
\begin{remark}
	Now we can formally derive a priori estimates for the system \eqref{1.1}-\eqref{1.2}-\eqref{1.3} and show the following energy equality 
	\begin{equation}\label{1.15l}
	\begin{split}
	&\frac{1}{2}\frac{d}{dt}\left(\int\limits_{\Omega_{t}}\rho|{\bf u}|^{2}\,dx\right)+\frac{d}{dt}\left(\int\limits_{\Omega_{t}}\frac{a}{(\gamma-1)}\rho^{\gamma}\,dx \right) +\frac{1}{2}\frac{d}{dt}\left(\int\limits_{0}^{L}|\eta_{t}|^{2}\,dx\right)
	+\frac{\beta}{2}\frac{d}{dt}\left(\int\limits_{0}^{L}|\eta_{x}|^{2}\,dx\right)\\[1.mm]
	& +\frac{\alpha}{2}\frac{d}{dt}\left(\int\limits_{0}^{L}|\eta_{xx}|^{2}\,dx\right)
	+2\mu\int\limits_{\Omega_{t}}| D{\bf u}|^{2}\,dx+\mu'\int\limits_{\Omega_{t}}|\mathrm{div}{\bf u}|^{2}\,dx+\delta\int\limits_{0}^{L}|\eta_{tx}|^{2}\,dx =-{p_{ext}}\int\limits_{\Gamma_{s}}\eta_{t}.
	\end{split}
	\end{equation}
The equality \eqref{1.15l} underlines the physical interpretation of each coefficient and in particular of the viscosity coefficients, $\mu,$ $\mu'$ and $\delta$.
\end{remark}
   \begin{remark}\label{eta0}
   	Observe that in \eqref{1.1} we have considered the initial displacement $\eta(0)$ of the beam to be zero. This is because we prove the local existence of strong solution of the system \eqref{1.1} with the beam displacement $\eta$ close to the steady state zero. There are several examples in the literature where the authors consider the initial displacement of the structure (in a fluid-structure interaction problem) to be equal to zero. For instance the readers can look into the articles \cite{kukavica} and \cite{boukir}. We also refer to the article \cite{veiga} where the initial displacement of the structure is non zero but is considered to be suitably small. The issues involving the existence of strong solution for the model \eqref{1.1} but with a non zero initial displacement $\eta(0)$ of the beam is open. The case of a system coupling the incompressible Navier-Stokes equations and an Euler-Bernoulli damped beam with a non zero initial beam displacement is addressed in \cite{casanova}.  
   \end{remark}
 Our interest is to prove the local in time existence of a strong solution to system \eqref{1.1}-\eqref{1.2}-\eqref{1.3} i.e we prove that given a prescribed initial datum $(\rho_{0},{\bf u}_{0},\eta_{1}),$ there exists a solution of system \eqref{1.1}-\eqref{1.2}-\eqref{1.3} with a certain Sobolev regularity in some time interval $(0,T),$ provided that the time $T$ is small enough.\\ 
 We study the system \eqref{1.1}-\eqref{1.2}-\eqref{1.3} by transforming it into the reference cylindrical domain $Q_{T}.$ This is done by defining a diffeomorphism from $\Omega_{t}$ onto $\Omega.$ We adapt the diffeomorphism used in \cite{veiga} in the study of an incompressible fluid-structure interaction model. The reader can also look at \cite{raymondbeam}, \cite{grand} where the authors use a similar map in the context of a coupled fluid-structure model comprising an incompressible fluid. 
 \subsection{Transformation of the problem to a fixed domain}\label{transfixdm}
 To transform the system \eqref{1.1}-\eqref{1.2}-\eqref{1.3} in the reference configuration, for $\eta$ satisfying
 $
 1+\eta(x,t)>0$ for all $(x,t)\in\Sigma^{s}_{T},
$
  we introduce the following change of variables
 \begin{equation}\label{1.14}
 \begin{array}{l}
 \displaystyle
 {\Phi}_{\eta(t)}:\Omega_{t}\longrightarrow \Omega\quad\mbox{defined by}\quad {\Phi}_{\eta(t)}(x,y)=(x,z)=\left(x,\frac{y}{1+\eta(x,t)}\right),\\
 \displaystyle
 {\Phi}_{\eta}:\widetilde{{Q}_{T}}\longrightarrow Q_{T}\quad\mbox{defined by}\quad {\Phi}_{\eta}(x,y,t)=(x,z,t)=\left(x,\frac{y}{1+\eta(x,t)},t\right).
 \end{array} 
 \end{equation}
 \begin{remark}
 	It is easy to prove that for each $t\in[0,T),$ the map ${\Phi}_{\eta(t)}$ is a $C^{1}-$ diffeomophism from $\Omega_{t}$ onto $\Omega$ provided that $(1+\eta(x,t))>0$ for all $x\in \mathbb{T}_L$ and that $\eta(\cdot,t)\in C^{1}(\Gamma_{s}).$
 \end{remark}
 Notice that since $\eta(\cdot,0)=0,$ ${\Phi}_{\eta(0)}$ is just the identity map. We set the following notations
 \begin{equation}\label{1.15}
 \begin{array}{l}
 \widehat{\rho}(x,z,t)=\rho(\Phi^{-1}_{\eta}(x,z,t)),\,\,\widehat{{\bf u}}(x,z,t)=(\widehat{u}_{1},\widehat{u}_{2})={\bf u}(\Phi^{-1}_{\eta}(x,z,t)).\\
 %\mbox{and}\,\, \widehat{p}(\widehat{\rho}(x,z,t))=p(\rho(\Phi^{-1}(x,z,t))).
 \end{array}
 \end{equation}
 After transformation and using the fact that $\widehat{{ u}}_{1,x}=0$ on $\Sigma^{s}_{T}$ (since $\widehat {\bf u}=\eta_{t}\vec{e_{2}}$ on $\Sigma^{s}_{T}$) the nonlinear system  \eqref{1.1}-\eqref{1.2}-\eqref{1.3} is rewritten in the following form
 \begin{equation}\label{1.16}
 \left\{
 \begin{array}{ll}
 \widehat{\rho_{t}}+\begin{bmatrix}
 \widehat {u}_{1}\\
 \frac{1}{(1+\eta)}(\widehat {u}_{2}-\eta_{t}z-\widehat{u}_{1}z\eta_{x})
 \end{bmatrix}
 \cdot \nabla\widehat{\rho}+\widehat{\rho}\mbox{div}{\widehat{\bf u}}={F}_{1}(\widehat{\rho},\widehat{\bf u},\eta)\quad& \mbox{in}\quad Q_{T},
 \vspace{1.mm}\\
 \widehat\rho \widehat {\bf u}_{t}-\mu\Delta \widehat {\bf u}-(\mu'+\mu)\nabla(\mbox{div}\widehat {\bf u})+ \nabla  P(\widehat{\rho}) ={F}_{2}(\widehat \rho,\widehat {\bf u},\eta)\quad &\mbox{in} \quad {Q}_{T},
 \vspace{1.mm}\\
 \widehat{\bf u}=\eta_{t}\vec{e_{2}}\quad& \mbox{on}\quad \Sigma^{s}_{T},\\[1.mm]
 \widehat {\bf u}(\cdot,t)=0\quad& \mbox{on}\quad \Sigma^{\ell}_{T},
 \vspace{1.mm}\\
 \widehat {\bf u}(\cdot,0)={\bf u}_{0}\quad& \mbox{in} \quad \Omega,
 \vspace{1.mm}\\
 \widehat{\rho}(\cdot,0)={\rho_{0}}\quad& \mbox{in}\quad \Omega,
 \vspace{1.mm}\\
 \eta_{tt}-\beta \eta_{xx}-  \delta\eta_{txx}+\alpha\eta_{xxxx}=F_{3}(\widehat{\rho},\widehat{\bf u},{\eta}) \quad& \mbox{on}\quad \Sigma^{s}_{T},
 \vspace{1.mm}\\
 \eta(0)=0\quad \mbox{and}\quad \eta_{t}(0)=\eta_{1}\quad &\mbox{in}\quad \Gamma_{s},
 \end{array} \right.
 \end{equation}
 %where, $\widehat {\bf u}^{0}(x,z)={\bf u}^{0}(x,y)={\bf u}^{0}(x,z(1+\eta(x,0)))={\bf u}^{0}o\mathcal{T}^{-1}_{\eta^{0}_{1}}(x,z),$ $\widehat{\rho}^{0}(x,z)=\rho^{0}(x,y)$ and
 where 
 \begin{equation}\label{F123}
 \begin{aligned}
 \displaystyle
 {F}_{1}(\widehat {\rho},\widehat {\bf u},\eta)=  &  \frac{1}{(1+\eta)}(\widehat { u}_{1,z}z\eta_{x}\widehat\rho+\eta\widehat{\rho}\widehat{u}_{2,z}),\\
 {F}_{2}(\widehat \rho,\widehat {\bf u},\eta)=  & -\eta\widehat{\rho}\widehat {\bf u}_{t}+z\widehat\rho\widehat {\bf u}_{z}\eta_{t}-\eta\widehat \rho\widehat {u}_{1}\widehat {\bf u}_{x}+\widehat { u}_{1}\widehat {\bf u}_{z}\eta_{x}\widehat \rho z+\mu \big(\eta\widehat {\bf u}_{xx}-\frac{\eta \widehat {\bf u}_{zz}}{(1+\eta)}-2\eta_{x}z\widehat {\bf u}_{zx}+\frac{\widehat {\bf u}_{zz}z^{2}\eta^{2}_{x}}{(1+\eta)}\\
 &  -\widehat {\bf u}_{z}\big( \frac{(1+\eta)z\eta_{xx}-2\eta_{x}^{2}z}{(1+\eta)}\big)\big)-\widehat\rho(\widehat {\bf u}.\nabla)\widehat {\bf u}+(\mu+\mu')\\
% & % +\frac{\mu'}{(1+\eta)}\begin{bmatrix}
 %\eta \frac{\partial (\mbox{div}(\widehat {\bf u}))}{\partial x}-z\eta_{x}\frac{\partial (\mbox{div}(\widehat {\bf u}))}{\partial z}-\eta_{x}\mbox{div}(\widehat {\bf u})+(1+\eta)\eta_{x}\widehat { u}_{1,x}-\eta\eta_{x}\widehat { u}_{1,x}+\eta (1+\eta)\widehat { u}_{1,xx}\\[2.mm]
%-\eta \widehat {u}_{1,xz}z\eta_{x}-z\eta^{2}_{x}\widehat { u}_{1,z}-z\eta_{x}(1+\eta)\widehat {u}_{1,zx}+z^{2}\eta^{2}_{x}\widehat {u}_{1,zz}-z\widehat { u}_{1,z}(1+\eta)\eta_{xx}\\
 %+\eta\widehat {u}_{1,xz}-\eta_{x}\widehat{u}_{1,z}-\eta_{x}\widehat {u}_{1,zz}
 %\end{bmatrix}\\[1.mm]
 & \displaystyle \cdot\begin{bmatrix}
 \eta\widehat {u}_{1,xx}-\widehat {u}_{1,xz}z\eta_{x}-\eta_{x}z\big(\widehat { u}_{1,zx}-\frac{\widehat {u}_{1,zz}z\eta_{x}}{(1+\eta)}\big)+\widehat { u}_{1,z}\big(\frac{(1+\eta)z\eta_{xx}-2\eta^{2}_{x}z}{(1+\eta)}\big) -\frac{\eta_{x}\widehat {u}_{2,z}}{(1+\eta)}-\frac{\eta_{x}z\widehat { u}_{2,zz}}{(1+\eta)}\\[2.mm]
 -\frac{\eta_{x}\widehat {u}_{1,z}}{(1+\eta)}-\frac{\eta_{x}z\widehat { u}_{1,zz}}{(1+\eta)}-\frac{\eta\widehat {u}_{2,zz}}{(1+\eta)}
 \end{bmatrix}\\
 &  -(\eta P_{x}(\widehat{\rho})-P_{z}(\widehat{\rho})z\eta_{x})\vec{e_{1}},\\
 \pagebreak
 F_{3}(\widehat{\rho},\widehat {\bf u},\eta)=  &  -\mu\big(-{\widehat{ u}_{2,z}}+\eta_{x}\widehat {u}_{2,x}+\frac{\widehat { u}_{2,z}}{(1+\eta)}\eta^{2}_{x}z-\frac{2\eta\widehat { u}_{2,z}}{(1+\eta)}-\frac{\eta_{x}\widehat {u}_{1,z}}{(1+\eta)}\big)
 -\mu'\big(-2\widehat{u}_{2,z}+\frac{\widehat { u}_{1,z}}{(1+\eta)}\eta_{x}z\\
 &-\frac{\eta\widehat {u}_{2,z}}{(1+\eta)}\big)+P(\widehat{\rho}).&
 \end{aligned}
 \end{equation}
 The transport equation for density \eqref{1.16}$_{1}$-\eqref{1.16}$_{6}$ is of the form
 \begin{equation}\label{1.17}
 \left\{
 \begin{array}{ll}
 \widehat{\rho_{t}}+\begin{bmatrix}
 \widehat {u}_{1}\\
 \frac{1}{(1+\eta)}(\widehat {u}_{2}-\eta_{t}z-\widehat{u}_{1}z\eta_{x})
 \end{bmatrix}
 \cdot \nabla\widehat{\rho}+\widehat{\rho}\mbox{div}{\widehat{\bf u}}={F}_{1}\quad& \mbox{in}\quad Q_{T},
 \vspace{1.mm}\\
 \widehat{\rho}(\cdot,0)={\rho_{0}}\quad& \mbox{in}\quad \Omega.
 \end{array} \right.
 \end{equation}
 Due to the interface condition, $\widehat {\bf u}=\eta_{t}\vec{e_{2}}$ on $\Sigma^{s}_{T},$ we get that the velocity field $(
 \widehat {u}_{1},
 \frac{1}{(1+\eta)}(\widehat {u}_{2}-\eta_{t}z-\widehat{u}_{1}z\eta_{x})
 )$ satisfies
 $$\begin{bmatrix}
 \widehat {u}_{1}\\
 \frac{1}{(1+\eta)}(\widehat {u}_{2}-\eta_{t}z-\widehat{u}_{1}z\eta_{x})
 \end{bmatrix}\cdot {\bf n}=0\quad \mbox{on}\quad \Sigma^{s}_{T},$$
 where ${\bf n}$ is the unit outward normal to $\Omega.$ 
 Hence we shall not prescribe any boundary condition on the density for the system \eqref{1.17} to be well posed.\\
To avoid working in domains which deform when time evolves, the meaning of solutions for \eqref{1.1}-\eqref{1.2}-\eqref{1.3} will be understood as follows: The triplet $(\rho,{\bf u},\eta)$ solves \eqref{1.1}-\eqref{1.2}-\eqref{1.3} if and only if $(\widehat{\rho},\widehat{\bf u},\eta)$ solves \eqref{1.16}. This notion will be detailed in the next section.   
\subsection{Functional settings and the main result}
In the fixed domain $\Omega$ we have the following spaces of functions with values in $\mathbb{R}^{2},$
$${\bf H}^{s}(\Omega)=H^{s}(\Omega;\mathbb{R}^{2})\quad\mbox{for all}\quad s\geqslant 0.$$
We also introduce the following spaces of vector valued functions 
\begin{equation}\label{fntlsp}
\begin{array}{l}
{\bf H}^{1}_{0}(\Omega)=\{{\bf z}\in{\bf H}^{1}(\Omega)\suchthat {\bf z}=0\,\,\mbox{on}\,\, \Gamma \},\\[1.mm]
{\bf H}^{2,1}(Q_{T})=L^{2}(0,T;{\bf H}^{2}(\Omega))\cap H^{1}(0,T;{\bf L}^{2}(\Omega)),\\[1.mm]
{\bf H}^{2,1}_{\Sigma_{T}}(Q_{T})=\{{\bf z}\in {\bf H}^{2,1}(Q_{T})\suchthat {\bf z}=0\,\,\mbox{on}\,\,\Sigma_{T} \}.
\end{array}
\end{equation}
Similarly for $s\geqslant 0,$ we can define $H^{s}(\Omega),$ the Sobolev space for the scalar valued functions defined on $\Omega.$ Now for $\theta,\tau\geqslant 0,$ we introduce the following spaces which we use to analyze the beam equation
\begin{equation}\nonumber
\begin{array}{l}
 H^{\theta,\tau}(\Sigma^{s}_{T})=L^{2}(0,T;H^{\theta}(\Gamma_{s}))\cap H^{\tau}(0,T;L^{2}(\Gamma_{s})).
\end{array}
\end{equation} 
\begin{remark}
	Since $\Omega=\mathbb{T}_{L}\times (0,1)$ and $\Gamma_{s}=\mathbb{T}_{L}\times \{1\},$ the above definitions of the functional spaces implicitly assert that the functions are $L-$ periodic in the $x$ variable. 
\end{remark}
\begin{prop}\label{pr1}
Let $T>0.$ If $\eta$ is regular enough in the space variable, say $\eta(\cdot,t)\in H^{m}(\Gamma_{s})$ for $m\geqslant 2$ and the following holds
\begin{equation}\label{coet}
\begin{array}{l}
1+\eta(x,t)\geqslant\delta_{0}>0\quad\mbox{on}\quad \Sigma^{s}_{T},
\end{array}
\end{equation}
for some constant $\delta_{0},$ the map ${{g}}\mapsto \widehat{{g}}=g(\Phi^{-1}_{\eta(t)}(x,z))$ is a homeomorphism from ${ H}^{s}(\Omega_{t})$ to ${ H}^{s}(\Omega)$ for any $s\leqslant m.$
\end{prop}
The proposition stated above can be proved in the same spirit of \cite[Proposition 2, Section 3]{grand}.\\
Now in view of Proposition \ref{pr1}, we define the notion of strong solution of the system \eqref{1.1}-\eqref{1.2}-\eqref{1.3} in terms of the strong solution of the system \eqref{1.16}.
\begin{mydef}\label{doss}
	 The triplet $(\rho,{\bf u},\eta)$ is a strong solution of the system \eqref{1.1}-\eqref{1.2}-\eqref{1.3} if
		\begin{equation}\label{conoeta}
		\begin{array}{ll}
		\eta\in  C^{0}\big([0,T]; H^{9/2}(\Gamma_{s})\big),&
		\eta_{t}\in L^{2}\big(0,T;{ H}^{4}(\Gamma_{s})\big)\cap C^{0}\big([0,T];{ H}^{3}(\Gamma_{s}) \big),\\[1.mm]
		\eta_{tt}\in L^{2}\big(0,T;{ H}^{2}(\Gamma_{s})\big)\cap C^{0}\big([0,T];{ H}^{1}(\Gamma_{s}) \big),&
		\eta_{ttt}\in L^{2}\big(0,T;{ L}^{2}(\Gamma_{s})\big),
		\end{array}
		\end{equation}
	   \eqref{coet} holds for every $(x,t)\in \Sigma^{s}_{T}$ and the triplet $(\widehat{\rho},\widehat{\bf u},\eta)=({\rho}\circ\Phi_{\eta}^{-1},{\bf u}\circ\Phi_{\eta}^{-1},\eta )$ solves \eqref{1.16} in the following Sobolev spaces
	\begin{equation}\label{de2}
	\begin{array}{l}
	\widehat\rho\in C^{0}\big([0,T]; H^{2}(\Omega) \big),\,
	\widehat\rho_{t}\in C^{0}\big([0,T]; H^{1}(\Omega)\big),\\[1.mm]
	\widehat{\bf u}\in L^{2}\big(0,T;{\bf H}^{3}(\Omega)\big)\cap C^{0}\big([0,T]; {\bf H}^{5/2}(\Omega) \big),
	\vspace{1.mm}\\
	\widehat{\bf u}_{t}\in L^{2}\big(0,T;{\bf H}^{2}(\Omega)\big)\cap C^{0}\big([0,T];{\bf H}^{1}(\Omega) \big),
	\vspace{1.mm}\\
	\widehat{\bf u}_{tt}\in L^{2}\big(0,T;{\bf L}^{2}(\Omega)\big).
	\end{array}
	\end{equation}
	($\eta$ is in the space mentioned in \eqref{conoeta}). Note that $(\rho,{\bf u})$ can then be obtained from $(\widehat{\rho},\widehat{\bf u})$ by $(\rho,{\bf u})=(\widehat{\rho}\circ\Phi_{\eta},\widehat{\bf u}\circ\Phi_{\eta}).$
\end{mydef}
In relation with Definition \ref{doss}, we introduce the following functional spaces
\begin{equation}\label{dofYi}
\begin{split}
Y_{1}^{T}=&\{{\rho}\in C^{0}([0,T];H^{2}(\Omega))\suchthat \rho_{t}\in C^{0}([0,T]; H^{1}(\Omega))\},\\[1.mm]
Y_{2}^{T}=&\{{\bf u}\in L^{2}(0,T;{\bf H}^{3}(\Omega))\cap C^{0}([0,T];{\bf H}^{5/2}(\Omega))\suchthat {\bf u}_{t}\in L^{2}(0,T;{\bf H}^{2}(\Omega))
\cap C^{0}([0,T];{\bf H}^{1}(\Omega)),\\[1.mm]
&\qquad\qquad\qquad\qquad\qquad\qquad\qquad\qquad\qquad\quad{\bf u}_{tt}\in L^{2}(0,T;{\bf L}^{2}(\Omega)) \},\\[1.mm]
Y_{3}^{T}=&\{\eta\in C^{0}([0,T]; H^{9/2}(\Gamma_{s})),\,\eta(x,0)=0\suchthat \eta_{t}\in L^{2}(0,T;{ H}^{4}(\Gamma_{s}))\cap C^{0}([0,T];{ H}^{3}(\Gamma_{s})), \\[1.mm]
&\qquad\qquad\qquad\qquad\qquad\eta_{tt}\in L^{2}(0,T;{ H}^{2}(\Gamma_{s}))\cap C^{0}([0,T];{ H}^{1}(\Gamma_{s})),
\eta_{ttt}\in L^{2}(0,T;{ L}^{2}(\Gamma_{s}))\}.
\end{split}
\end{equation}
The spaces $Y_{1}^{T},$ $Y_{2}^{T}$ and $Y_{3}^{T}$ correspond to the spaces in which the unknowns $\widehat{\rho},$ $\widehat{\bf u}$ and $\eta$ respectively.\\
 Now we precisely state the main result of the article.
\begin{thm}\label{main}
	Assume that 
	\begin{equation}\label{cit}
	\left\{ \begin{array}{ll}
	(i)\,&(a)\,\,{Regularity\,of\,initial\,conditions:}\,\rho_{0}\in {H}^{2}(\Omega),\,\, \eta_{1}\in {H}^{3}(\Gamma_{s}),\,\, {\bf u}_{0}\in {\bf H}^{3}(\Omega).\\[2.mm]
    &(b)\,\,{Compatibility\,between\,initial\, and\,boundary\,conditions}:\\
	&\quad(b)_{1}\,\,\left( {\bf u}_{0}-\begin{bmatrix}
	0\\
	z\eta_{1}
	\end{bmatrix}\right)=0\,\mbox{on}\,\Gamma,\\[2.mm]
	 &\quad(b)_{2}\,\,-P'(\rho_{0})\nabla\rho_{0}-(\delta\eta_{1,xx}-(\mu+2\mu')({ u}_{0})_{2,z}+P(\rho_{0}))z\rho_{0}\vec{e}_{2}+z\rho_{0}({\bf u}_{0})_{z}\eta_{1}\\
	 &\quad\,\,\quad\quad-\rho_{0}({\bf u}_{0}\cdot\nabla){\bf u}_{0}-(-\mu \Delta -(\mu+\mu{'})\nabla\mathrm{div}){\bf u}_{0}=0\,\mbox{on}\,\Gamma,\\
	(ii)\,&\eqref{cor0}\,\mbox{holds},
	\end{array}\right.
	\end{equation}
	 where we use the notations $P'({\rho}_{0})=\nabla P({\rho}_{0}),$ $P(\rho_{0})=(a\rho_{0}^{\gamma}-a\overline{\rho}^{\gamma})$ and ${\bf u}_{0}=((u_{0})_{1},(u_{0})_{2}).$
	Then there exists $T>0$ such that the system \eqref{1.16} admits a solution $(\widehat{\rho},\widehat{\bf u},\eta)\in Y_{1}^{T}\times Y_{2}^{T}\times Y_{3}^{T}.$ 
	Consequently in the sense of Definition \ref{doss} the system \eqref{1.1}-\eqref{1.2}-\eqref{1.3} admits a  strong solution $({\rho},{\bf u},\eta).$ 
\end{thm} 
 \begin{remark}
 	Our analysis throughout the article can be suitably adapted to consider any pressure law $p(\cdot)\in C^{2}(\mathbb{R}^{+})$ (in this article we present the proofs with the pressure law given by $p(\rho)=a\rho^{\gamma},$ with $\gamma>1$) such that there exists a positive constant $\overline{\rho}$ satisfying $p(\overline{\rho})={p_{ext}},$ where ${p_{ext}}(>0)$ is the external force acting on the beam. The adaptation is possible since we only consider the case where the fluid density $\rho$ has a positive lower and upper bound. 
 \end{remark} 
Now let us sketch the strategy towards the proof of Theorem \ref{main}.
\subsection{Strategy}\label{Chap3strategy}
(i) $\mathit{Changing\,\,\eqref{1.16}\,\, to\,\, a\,\, homogeneous\,\, boundary\,\, value\,\, problem}$: Recall that (see Remark \ref{eta0}) we will prove the existence of local in time strong solution of the system \eqref{1.1}-\eqref{1.2}-\eqref{1.3} only when the beam displacement $\eta$ is close to zero. Again observe that $(\widehat{\rho}=\overline{\rho},\widehat{\bf u}=0,\widehat{\eta}=0)$ is a steady state solution of the system \eqref{1.1}-\eqref{1.2}-\eqref{1.3} and hence of the system \eqref{1.16}. So to work in a neighborhood of $\eta=0,$ we make the following change of unknowns in \eqref{1.16},
\begin{equation}\label{1.20}
\begin{array}{l}
\sigma=\widehat\rho-\overline{\rho},
\quad
{\bf v}=(v_{1},v_{2})=\widehat {\bf u}-0,
\quad
\eta=\eta-0.
\end{array}
\end{equation} 
In view of the change of unknowns \eqref{1.20} one obtains
\begin{equation}\label{1.21}
\left\{
\begin{array}{ll}
{\sigma_{t}}+\begin{bmatrix}
{v}_{1}\\
\frac{1}{(1+\eta)}({v}_{2}-\eta_{t}z-{v}_{1}z\eta_{x})
\end{bmatrix}
\cdot \nabla\sigma+(\sigma+\overline{\rho})\mbox{div}({\bf v})={F}_{1}(\sigma+\overline{\rho},{\bf v},\eta)\quad &\mbox{in}\quad Q_{T},
\vspace{1.mm}\\
(\sigma+\overline{\rho}){\bf v}_{t}-\mu \Delta {\bf v}-(\mu+\mu{'})\nabla\mathrm{div}{\bf v}=-{P}'(\sigma+\overline{\rho}) \nabla \sigma+{F}_{2}(\sigma+\overline{\rho},{\bf v},\eta)\quad &\mbox{in} \quad {Q}_{T},
\vspace{1.mm}\\
{\bf v}=\eta_{t}\vec{e_{2}}\quad& \mbox{on}\quad \Sigma^{s}_{T},\\[1.mm]
{\bf v}=0\quad &\mbox{on}\quad\Sigma^{\ell}_{T},
\vspace{1.mm}\\
{\bf v}(\cdot,0)={\bf u}_{0}\quad& \mbox{in} \quad \Omega,
\vspace{1.mm}\\
\sigma(\cdot,0)=\sigma_{0}={\rho_{0}}-{\overline\rho}\quad &\mbox{in}\quad \Omega,
\vspace{1.mm}\\
\eta_{tt}-\beta \eta_{xx}-  \delta\eta_{txx}+\alpha\eta_{xxxx}={F}_{3}(\sigma+\overline{\rho},{\bf v},\eta)\quad &\mbox{on}\quad \Sigma^{s}_{T},
\vspace{1.mm}\\
\eta(0)=0\quad \mbox{and}\quad \eta_{t}(0)=\eta_{1}\quad &\mbox{in}\quad \Gamma_{s}.
\end{array} \right.
\end{equation}
We transform the system \eqref{1.21} into a homogeneous Dirichlet boundary value problem by performing further the following change of unknown
\begin{equation}\label{cvhd}
\begin{array}{l}
{\bf w}=(w_{1},w_{2})={\bf v}-z\eta_{t}\vec{e_{2}}.
\end{array}
\end{equation}
Since ${\bf v}$ and ${\eta}_{t}$ both are $L$-periodic in the $x-$direction, the new unknown ${\bf w}$ is also $L$-periodic in the $x-$direction. With the new unknown ${\bf w},$ we write the transformed system in the following form
\begin{equation}\label{chdb}
\left\{
\begin{array}{ll}
{\sigma_{t}}+\begin{bmatrix}
{w}_{1}\\
\frac{1}{(1+\eta)}({w}_{2}-{w}_{1}z\eta_{x})
\end{bmatrix}
\cdot \nabla\sigma=G_{1}(\sigma,{\bf w},\eta)\quad& \mbox{in}\quad Q_{T}, 
\vspace{1.mm}\\
(\sigma+\overline{\rho}){\bf w}_{t}-\mu \Delta {\bf w} -(\mu+\mu{'})\nabla\mathrm{div}{\bf w}=G_{2}(\sigma,{\bf w},\eta)\quad& \mbox{in}\quad Q_{T},
\vspace{1.mm}\\
{\bf w}=0\quad &\mbox{on}\quad \Sigma_{T},
\vspace{1.mm}\\
{\bf w}(\cdot,0)={\bf w}_{0}={\bf u}_{0}-z\eta_{1}\vec{e}_{2}\quad &\mbox{in} \quad \Omega,
\vspace{1.mm}\\
\sigma(\cdot,0)=\sigma_{0}={\rho_{0}}-{\overline\rho}\quad &\mbox{in}\quad \Omega,
\vspace{1.mm}\\
\eta_{tt}-\beta \eta_{xx}-  \delta\eta_{txx}+\alpha\eta_{xxxx}=G_{3}(\sigma,{\bf w},\eta)\quad & \mbox{on}\quad \Sigma^{s}_{T},
\vspace{1.mm}\\
\eta(0)=0\quad \mbox{and}\quad \eta_{t}(0)=\eta_{1}\quad& \mbox{in}\quad \Gamma_{s},
\end{array} \right.
\end{equation}
where
\begin{equation}\label{1.22}
\begin{split}
& G_{1}(\sigma,{\bf w},\eta)=-(\sigma+\overline{\rho})\mbox{div}({\bf w}+z\eta_{t}\vec{e_{2}})+{F}_{1}(\sigma+\overline{\rho},{{\bf w}+z\eta_{t}\vec{e_{2}}},\eta),\\[1.mm]
& G_{2}(\sigma,{\bf w},\eta)=-{P}'(\sigma+\overline{\rho}) \nabla \sigma-	z\eta_{tt}(\sigma+\overline{\rho})\vec{e_{2}}-(-\mu \Delta -(\mu+\mu{'})\nabla\mathrm{div})(z\eta_{t}\vec{e_{2}})\\[1.mm]
& \qquad\qquad\qquad+{F}_{2}(\sigma+\overline{\rho},{\bf w}+z\eta_{t}\vec{e_{2}},\eta),
\vspace{1.mm}\\
& G_{3}(\sigma,{\bf w},\eta)={F}_{3}(\sigma+\overline{\rho},{\bf w}+\eta_{t}\vec{e_{2}},\eta).
\end{split}
\end{equation}
(ii) $\textit{Study\,\,of\,\,some\,\,decoupled\,\,linear\,\,problems}$: Observe that in the new system \eqref{chdb} the coupling between the velocity of the fluid and the elastic structure appears only as source terms. In order to solve the system \eqref{chdb} we first study some linear equations in Section \ref{lineqs}. In order to analyze the local in time existence of strong solution the difficulty is to track the dependence of the constants (appearing in the inequalities) with respect to the time parameter \textquoteleft T\textquoteright.\ \,In this direction we first obtain a priori estimates for the linear density and velocity equations with non homogeneous source terms in the spirit of \cite{vallizak}. Then we prove the existence of strong solutions for a linear beam equation. The proof strongly relies on the analyticity of the corresponding beam semigroup (see \cite{chen} for details). At this point we refer the readers to the articles \cite{denk} (maximal $L^{p}-L^{q}$ regularity of structurally damped beam equation), \cite{fanli} (analyticity and exponential stability of beam semigroup), \cite{raymondbeam} (study of beam equation in the context of an incompressible fluid structure interaction problem) and the references therein for the existence and regularity issues of the damped beam equation. In our case to obtain estimates with the constants independent of \textquoteleft T\textquoteright\ for the beam equation we first fix a constant $\overline{T}>0$ and restrict ourselves to work in the time interval $(0,T)$ where
\begin{equation}\label{Tbar}
\begin{array}{l}
T<\overline{T}.
\end{array}
\end{equation}
This technique is inspired from \cite{rayvan}.\\ 
(iii) $\mathit{Fixed\,\,point\,\,argument}$: In Section \ref{fixdpt} we prove the existence of a strong solution of \eqref{chdb} by using the Schauder's fixed point theorem based on \eqref{chdb}-\eqref{1.22}.
\begin{remark}
	Since $\eta(0)=0$ the regularity \eqref{conoeta} of $\eta$ guarantees that  
	\begin{equation}\label{smalleta}
	\begin{array}{l}
	\|\eta\|_{L^{\infty}(\Sigma^{s}_{T})}\leqslant CT\|\eta_{t}\|_{L^{\infty}(0,T;H^{3}(\Gamma_{s}))},
	\end{array}
	\end{equation}
	for a constant $C$ independent of $T.$ For small enough time $T,$ \eqref{smalleta} furnishes $\eta\approx 0$ and hence during small times, the beam stays close to the steady state zero.
\end{remark}
\subsection{Comments on initial and compatibility conditions}\label{incom}
(i) %Using the regularity assumptions \eqref{cit}$(i)$$(a)$ into \eqref{cit}$(i)$$(b)$ we observe that
%$$(-\mu \Delta -(\mu+\mu{'})\nabla\mathrm{div})\left({\bf u}_{0}-\begin{bmatrix}
%0\\
%z\eta_{1}
%\end{bmatrix}\right)\in L^{p}(\Omega),$$
%for some $1<p<2$ (we verify in particular that ${\bf u}_{0}\in {\bf H}^{1}(\Omega)\Rightarrow ({\bf u}_{0}\cdot\nabla){\bf u}_{0}\in L^{p}(\Omega)$ for $1<p<2$). Also from \eqref{cit}$(i)$$(a)$ we recall that $\left({\bf u}_{0}-(
%0,
%z\eta_{1}\right)
%%\begin{equation}\nonumber
%\left\{ \begin{array}{l}
%(-\mu \Delta -(\mu+\mu{'})\nabla\mathrm{div})\left({\bf u}_{0}-\begin{bmatrix}
%%z\eta_{1}
%\end{bmatrix}\right)\in L^{p}(\Omega),\\
%%0\\
%z\eta_{1}
%\end{bmatrix}\right)=0\quad\mbox{on}\quad\Gamma,
%\end{array}\right.
%\end{equation}
%one obtains ${\bf u}_{0}\in W^{2,p}(\Omega).$   
%Now using Sobolev embedding theorem and the regularity of $\rho_{0}$ (from \eqref{cit}$(i)$$(a)$) one verifies that $\rho_{0}({\bf u}_{0}\cdot\nabla){\bf u}_{0}\in L^{2}(\Omega).$ Similarly we can check the regularities of other terms appearing in \eqref{cit}$(i)$$(b)$ to obtain
%$$(-\mu \Delta -(\mu+\mu{'})\nabla\mathrm{div})\left({\bf u}_{0}-\begin{bmatrix}
%%z\eta_{1}
%\end{bmatrix}\right)\in L^{2}(\Omega).$$
%Once again using elliptic regularity result and a bootstrap argument furnish that
%\begin{equation}\label{regu0}
%\begin{array}{l}
%{\bf u}_{0}\in {\bf H}^{3}(\Omega).
%\end{array}
%\end{equation}
Recall from \eqref{cit}$(i)$$(a)$ that we assume ${\bf u}_{0}\in {\bf H}^{3}(\Omega).$ Also observe that in our solution (see \eqref{de2}) the vector field  $\widehat{\bf u}\in C^{0}([0,T];{\bf H}^{5/2}(\Omega))$ i.e for the velocity field there is a loss of $\frac{1}{2}$ space regularity as the time evolves. One can find such instances of a loss of space regularity in many other articles in the literature, for instance we refer the readers to \cite{bougue3}, \cite{kukavica} (for the coupling of fluid-elastic structure comprising a compressible fluid) and \cite{cout1}, \cite{cout2}, \cite{rayvan} (for incompressible fluid structure interaction models).\\[2.mm] 
(ii) We use \eqref{1.22}$_{3}$ to obtain the following expression of $G_{3}\mid_{t=0}$ (the value of $G_{3}(\sigma,{\bf w},\eta)$ at time $t=0$)
\begin{equation}\label{G30}
\begin{array}{l}
G_{3}\mid_{t=0}=-(\mu+2\mu')({u}_{0})_{2,z}+P(\rho_{0}).
\end{array}
\end{equation} 
Using $\rho_{0}\in H^{2}(\Omega),$ ${\bf u}_{0}\in {\bf H}^{3}(\Omega)$ (see \eqref{cit}$(i)$$(a)$) and standard trace theorems one easily checks that 
\begin{equation}\label{regG3t0}
\begin{array}{l}
G_{3}\mid_{t=0}\in H^{3/2}(\Gamma_{s}).
\end{array}
\end{equation}
We will use the regularity of $G_{3}\mid_{t=0}$ (in fact we will only use $G_{3}\mid_{t=0}\in H^{1}(\Gamma_{s})$) to prove the regularity of $\eta.$ This will be detailed in Theorem \ref{t23}.\\[2.mm]
 (iii) We use \eqref{G30} and the equation \eqref{chdb}$_{6}$ to check that
 $$\eta_{tt}(\cdot,0)=\delta\eta_{1,xx}-(\mu+2\mu')(u_{0})_{2,z}+P(\rho_{0}).$$
  Hence using \eqref{1.22}$_{2}$ one obtains the following expression of $G_{2}\mid_{t=0}$ (the value of $G_{2}(\sigma,{\bf w},\eta)$ at time $t=0$)
\begin{equation}\label{iG0}
\begin{split}
G_{2}\mid_{t=0}&=-P'(\rho_{0})\nabla\rho_{0}-(\delta\eta_{1,xx}-(\mu+2\mu')({ u}_{0})_{2,z}
+P(\rho_{0}))z\rho_{0}\vec{e}_{2}
+z\rho_{0}({\bf u}_{0})_{z}\eta_{1}\\
&-\rho_{0}({\bf u}_{0}\cdot\nabla){\bf u}_{0}-(-\mu\Delta-(\mu+\mu{'})\nabla\mathrm{div})(z\eta_{1}\vec{e_{2}}).
\end{split}
\end{equation}
 This gives
\begin{equation}\label{compat}
\begin{split}
G_{2}\mid_{t=0}&-(-\mu \Delta -(\mu+\mu{'})\nabla\mathrm{div})\left( {\bf u}_{0}-\begin{bmatrix}
0\\
z\eta_{1}
\end{bmatrix}\right)
=-P'(\rho_{0})\nabla\rho_{0}-(\delta\eta_{1,xx}-(\mu+2\mu')({ u}_{0})_{2,z}\\
&+P(\rho_{0}))z\rho_{0}\vec{e}_{2}
+z\rho_{0}({\bf u}_{0})_{z}\eta_{1}-\rho_{0}({\bf u}_{0}\cdot\nabla){\bf u}_{0}-(-\mu \Delta -(\mu+\mu{'})\nabla\mathrm{div}){\bf u}_{0}.
\end{split}
\end{equation}
The regularity assumptions \eqref{cit}$(i)$$(a)$ and \eqref{compat} furnish the following
\begin{equation}\label{regcom}
\begin{array}{l}
G_{2}\mid_{t=0}-(-\mu \Delta -(\mu+\mu{'})\nabla\mathrm{div})\left( {\bf u}_{0}-\begin{bmatrix}
0\\
z\eta_{1}
\end{bmatrix}\right)\in {\bf H}^{1}(\Omega).
\end{array}
\end{equation}
Hence one obtains (recalling that ${\bf w}_{0}={\bf u}_{0}-z\eta_{1}\vec{e}_{2}$)
\begin{equation}\label{asm2}
\begin{array}{ll}
&\mbox{the assumption}\, \eqref{cit}(i)(a)\,\mbox{and}\, \eqref{cit}(i)(b)_{2}\\
&\Longrightarrow G_{2}\mid_{t=0}-(-\mu \Delta{\bf w}_{0} -(\mu+\mu{'})\nabla\mathrm{div}{\bf w}_{0})\in {\bf H}^{1}_{0}(\Omega).
\end{array}
\end{equation}
We need this to prove some regularity of ${\bf w}$ and hence of ${\widehat{\bf u}}.$ This will be detailed in Theorem \ref{t21}.\\[2.mm] 
\subsection{Bibliographical comments}\label{bibcomlclext} Here we mainly focus on the existing literature devoted to the study of fluid structure interaction problems.\\ 
To begin with we quote a few articles dedicated to the mathematical study of compressible Navier-Stokes equations. The existence of local in time classical solutions for the compressible Navier-Stokes equations in a time independent domain was first proved in \cite{nash} and the uniqueness was established in \cite{serrin}. The global existence of strong solutions for a small perturbation of a stable constant state was established in the celebrated work \cite{matnis}. In the article \cite{vallizak} the authors established the local in time existence of strong solutions in the presence of inflow and outflow of the fluid through the boundary. In the same article they also present the proof of global in time existence for small data in the absence of the inflow. P.-L. Lions proved (in \cite{lions}) the global existence of renormalized weak solution with bounded energy for an isentropic fluid (i.e $p(\rho)=\rho^{\gamma}$) with the adiabatic constant $\gamma>3d/(d+2),$ where $d$ is the space dimension. E. Feireisl $\mathit{et\, al.}$ generalized the approach to cover the range $\gamma>3/2$ in dimension $3$ and $\gamma>1,$ in dimension $2$ in \cite{feireisl}. The well-posedness issues of the compressible Navier-Stokes equations for critical regularity data can be found in \cite{danchin}, \cite{danchin1}. For further references and a very detailed development of the mathematical theory of compressible flow we refer the reader into the books \cite{novotny1} and \cite{bresch1}.\\
In the last decades the fluid-structure interaction problems have been an area of active research. There is a rich literature concerning the motion of a structure inside or at the boundary of a domain containing a viscous incompressible Newtonian fluid, whose behavior is described by Navier-Stokes equations. For instance local existence and uniqueness of strong solutions of incompressible fluid-structure models with the structure immersed inside the fluid are studied in \cite{cout1} (the elastic structure is modeled by linear Kirchhoff equations) and \cite{cout2} (the elastic structure is governed by quasilinear elastodynamics). There also exist articles dealing with incompressible fluid-structure interaction problems where the structure appears on the fluid boundary and is modeled by Euler-Bernoulli damped beam equations \eqref{1.1}$_{7}$-\eqref{1.1}$_{8}.$ For example we refer the readers to \cite{veiga} (local in time existence of strong solutions), \cite{esteban} (existence of weak solutions), \cite{raymondbeam} (feedback stabilization), \cite{grand} (global in time existence) and the references therein for a very detailed discussion of such problems.\\
Despite of the growing literature on incompressible fluids the number of articles addressing the compressible fluid-structure interaction problems is relatively limited and the literature has been rather recently developed. One of the fundamental differences between the incompressible and compressible Navier-Stokes equations is that the pressure of the fluid in incompressible Navier-Stokes equations is interpreted as the Lagrange multiplier whereas in the case of compressible Navier-Stokes equations the pressure is given as a function of density with the density modeled by a transport equation of hyperbolic nature. The strong coupling between the parabolic and hyperbolic dynamics is one of the intricacies in dealing with the compressible Navier-Stokes equations and this results in the regularity incompatibilities between the fluid and the solid structure. However in the past few years there have been works exploring the fluid-structure interaction problems comprising the compressible Navier-Stokes equations with an elastic body immersed in the fluid domain. For instance in the article \cite{bougue2} the authors prove the existence and uniqueness of strong solutions of a fluid structure interaction problem for a compressible fluid and a rigid structure immersed in a regular bounded domain in dimension 3. The result is proved in any time interval $(0,T),$ where $T>0$ and for a small perturbation of a stable constant state provided there is no collision between the rigid body and the boundary $\partial\Omega$ of the fluid domain. In \cite{boulakia1} the existence of weak solution is obtained in three dimension for an elastic structure immersed in a compressible fluid. The structure equation considered in \cite{boulakia1} is strongly regularized in order to obtain suitable estimates on the elastic deformations. A result concerning the local in time existence and uniqueness of strong solutions for a problem coupling compressible fluid and an elastic structure (immersed inside the fluid) can be found in \cite{bougue3}. In the article \cite{bougue3} the equation of the structure does not contain any extra regularizing term. The flow corresponding to a Lagrangian velocity is used in \cite{bougue3} in order to transform the fluid structure interaction problem in a reference fluid domain $\Omega_{F}(0),$ whereas in the present article we use the non physical change of variables \eqref{1.14} for the similar purpose of writing the entire system in a reference configuration. A similar Navier-Stokes-Lam\'{e} system as that of \cite{bougue3} is analyzed in \cite{kukavica} to prove the existence of local in time strong solutions but in a different Sobolev regularity framework. In the article \cite{kukavica} the authors deal with less regular initial data.  We also quote a very recent work \cite{boukir} where the authors prove the local in time existence of a unique strong solution of a compressible fluid structure interaction model where the structure immersed inside the fluid is governed by the Saint-Venant Kirchhoff equations.\\
 On the other hand there is a very limited number of works on the compressible fluid-structure interaction problems with the structure appearing on the boundary of the fluid domain. The article \cite{flori2} deals with a 1-D structure governed by plate equations coupled with a bi-dimensional compressible fluid where the structure is located at a part of the boundary. Here the authors consider the velocity field as a potential and in their case the non linearity occurs only in the equation modeling the density. Instead of writing the system in a reference configuration in \cite{flori2} the authors proved the existence and uniqueness of solution in Sobolev-like spaces defined on time dependent domains. The existence of weak solution for a different compressible fluid structure interaction model (with the structure appearing on the boundary) is studied in dimension three by the same authors in \cite{flori1}. In the model considered in \cite{flori1}, the fluid velocity $v$ satisfies $\mbox{curl}v\wedge n=0$ on the entire fluid boundary and the plate is clamped everywhere on the structural boundary. In a recent article \cite{gavalos1} the authors prove the Hadamard well posedness of a linear compressible fluid structure interaction problem (three dimensional compressible fluid interacting with a bi-dimensional elastic structure) defined in a fixed domain  and considering the Navier-slip boundary condition at the interactive boundary. They write the coupled system in the form
 $$\frac{d}{dt}\begin{pmatrix}
 \rho\\
 {\bf u}\\
 \eta\\
 \eta_{t}
 \end{pmatrix}=A\begin{pmatrix}
 \rho\\
 {\bf u}\\
 \eta\\
 \eta_{t}
 \end{pmatrix}\,\,\mbox{in}\,\,(0,T),\quad \mbox{and}\quad\begin{pmatrix}
 \rho(0)\\
 {\bf u}(0)\\
 \eta(0)\\
 \eta_{t}(0)
 \end{pmatrix}=\begin{pmatrix}
 \rho_{0}\\
 {\bf u}_{0}\\
 \eta_{1}\\
 \eta_{2}
 \end{pmatrix},$$
  and prove the existence of mild solution $(\rho,{\bf u},\eta,\eta_{t})$ in the space $C^{0}([0,T];\mathcal{D}(A))$ where $\mathcal{D}(A)$ is the domain of the operator $A.$ Their approach is based on using the Lumer-Phillips theorem to prove that $A$ generates a strongly continuous semigroup. In yet another recent article \cite{breit} the authors consider a three dimensional compressible fluid structure interaction model where the structure located at the boundary is a shell of Koiter-type with some prescribed thickness. In the spirit of \cite{lions} and \cite{feireisl} the authors prove the existence of a weak solution for their model with the adiabatic constant restricted to $\gamma>\frac{12}{7}.$ They show that a weak solution exists until the structure touches the boundary of the fluid domain.\\
 To the best of our knowledge there is no existing work (neither in dimension 2 nor in 3) proving the existence of strong solutions for the non-linear compressible fluid-structure interaction problems (defined in a time dependent domain) considering the structure at the boundary of the fluid domain. In the present article we address this problem in the case of a fluid contained in a 2d channel and interacting with a 1d structure at the boundary. Our approach is different from that of \cite{gavalos1} and \cite{breit}. In \cite{gavalos1}, since the problem itself is linearized in a fixed domain, the authors can directly use a semigroup formulation to study the existence of mild solution, whereas \cite{breit} considers weak solutions and a $4$ level approximation process (using artificial pressure, artificial viscosity, regularization of the boundary and Galerkin approximation for the momentum equation). In the study of weak solutions (in \cite{lions}, \cite{feireisl}, \cite{breit}) one of the major difficulties is to pass to the limit in the non-linear pressure term which is handled by introducing a new unknown called the effective viscous flux. In our case of strong regularity framework we do not need to introduce the effective viscous flux and for small enough time $T,$ the term $\nabla P(\sigma+\overline{\rho})$ can be treated as a non homogeneous source term. Our approach is based on studying the regularity properties of a decoupled parabolic equation, continuity equation and a beam equation. This is done by obtaining some apriori estimates and exploiting the analyticity of the semigroup corresponding to the beam equation. Then the existence result for the non-linear coupled problem is proved by using the Schauder's fixed point argument. We prove the existence of the fixed point in a suitable convex set, which is constructed very carefully based on the estimates of the decoupled problems and the estimates of the non-homogeneous source terms. This led us to choose this convex set as a product of balls (in various functional spaces) of different radius. In the present article we prove a local in time existence result of strong solutions whose incompressible counterpart was proved in \cite{veiga}.\\  
  Let us also mention the very recent article \cite{shibata} where the global existence for the  compressible viscous fluids (without any structure on the boundary) in a bounded domain is proved in the maximal $L^{p}-L^{q}$ regularity class. In this article the authors consider a slip type boundary condition. More precisely the fluid velocity ${\bf u}$ satisfies the following on the boundary
  $$
D({\bf u}){\bf n}-\langle D({\bf u}){\bf n},{\bf n}\rangle {\bf n}=0,\quad\mbox{and}\quad {\bf u}\cdot {\bf n}=0\quad\mbox{on}\quad \partial\Omega\times (0,T).$$
 In a similar note one can consider a fluid structure interaction problem with slip type boundary condition. In that case the velocity field ${\bf u}$ solves the following
 \begin{equation}\label{fslip}
 \begin{array}{l}
 D({\bf u}){\bf n}-\langle D({\bf u}){\bf n},{\bf n}\rangle {\bf n}=0,\quad\mbox{and}\quad {\bf u}\cdot {\bf n}=\eta_{t}\quad\mbox{on}\quad \Gamma_{s}\times (0,T),
 \end{array}
\end{equation}
 where $\eta_{t}$ is the structural velocity at the interactive boundary $\Gamma_{s}\times(0,T).$ To the best of our knowledge for a compressible fluid structure interaction problem the condition \eqref{fslip} is treated only in \cite{gavalos1}, proving the existence of mild solution. Of course the boundary condition \eqref{fslip} is different from the one we consider in the present article since in our case we do not allow the fluid to slip tangentially through the fluid structure interface (i.e recall in our case ${ u}_{1}=0$ on $\Sigma^{s}_{T}$).\\
 A more generalized slip boundary condition is considered in \cite{muhacan} in the context of an incompressible fluid structure interaction problem. In the model examined in \cite{muhacan} the structural displacement has both tangential and normal components with respect to the reference configuration. At the interface the fluid and the structural velocities are coupled via a kinematic coupling condition and a dynamic coupling condition (stating that the structural dynamics is governed by the jump of the normal stress at the interface). The kinematic coupling conditions at the interface treated in \cite{muhacan} consists of continuity of the normal velocities and a second condition stating that the slip between the tangential components of the fluid and structural velocities is proportional to the fluid normal stress. The authors in \cite{muhacan} prove the existence of a weak solution for their model. 
 \subsection{Outline} Section \ref{lineqs} contains results involving the existence and uniqueness of some decoupled linear equations. We state the existence and uniqueness result for a parabolic equation in Section \ref{veleq}, continuity equation in Section \ref{deneq}, linear beam equation in Section \ref{beameq}. In Section \ref{fixdpt} we prove Theorem \ref{main} by using the Schauder fixed point theorem. 
 %In Section \ref{eneq} we include the proof of the energy identity stated in \eqref{1.15l}.
\section{Analysis of some linear equations}\label{lineqs}
We will prove the existence and uniqueness of strong solutions of a parabolic equation, a continuity equation and a damped beam equation with prescribed initial data and source terms in appropriate Sobolev spaces.\\
From now onwards all the constants appearing in the inequalities will be independent of the final time $T,$ unless specified. We also comment that we will denote many of the constants in the inequalities using the same notation although they might vary from line to line. 
\subsection{ Study of a parabolic equation}\label{veleq}
At first we consider the following linear problem
\begin{equation}\label{2.1.1}
\left\{ \begin{array}{ll}
\overline{\sigma}{\bf w}_{t}-\mu \Delta {\bf w} -(\mu+\mu{'})\nabla\mathrm{div}{\bf w}=G_{2}&\quad \mbox{in}\quad Q_{T},
\vspace{1.mm}\\
{\bf w}=0&\quad\mbox{on}\quad \Sigma_{T},\\[1.mm]
{\bf w}(0)={\bf w}_{0}&\quad \mbox{in}\quad \Omega,
\end{array}\right.
\end{equation}
where $\overline{\sigma},$ ${\bf w}_{0}$ and $G_{2}$ are known functions which are $L$-periodic in the $x$ direction.\\
Let $m$ and $M$ be positive constants such that $m<M.$ We are going to study \eqref{2.1.1} where $\overline{\sigma},$ ${\bf w}_{0}$ and $G_{2}$ satisfy the following
\begin{equation}\label{condvel}
\left\{ \begin{array}{l}
\overline{\sigma}\in L^\infty(Q_{T}),\, 0<m/2\leqslant\overline{\sigma}\leqslant 2M\, \mbox{in}\,\, Q_{T},\,
0<{m}\leqslant\overline{\sigma}(\cdot,0)\leqslant M\,\mbox{in}\,\, \Omega,\\ \nabla\overline{\sigma}\in L^{2}(0,T;{\bf L}^{3}(\Omega)),\, \overline{\sigma}_{t}\in L^{2}(0,T;{L}^{3}(\Omega)),\\
%\overline{\sigma}\in C^{0}([0,T];{\bf H}^{2}(\Omega)),\, 0<m/2\leqslant\overline{\sigma}\leqslant 2M\, \mbox{in}\,\, Q_{T},\,
%0<{m}\leqslant\overline{\sigma}(\cdot,0)\leqslant M\,\mbox{in}\,\, \Omega,\\ \nabla\overline{\sigma}\in L^{2}(0,T;{\bf L}^{3}(\Omega)),\, \overline{\sigma}_{t}\in L^{2}(0,T;{L}^{3}(\Omega)),\\
\end{array}\right.
\end{equation}
and
\begin{equation}\label{condvel2}
\left\{ \begin{array}{l}
G_{2}\in L^{2}(0,T;{\bf H}^{1}(\Omega)),\,G_{2,{t}}\in L^{2}(0,T;{\bf L}^{2}(\Omega)),\\ {\bf w}_{0}\in  {\bf H}^{1}_{0}(\Omega),\\
(G_{2}\mid_{t=0}-(-\mu \Delta {\bf w}_{0} -(\mu+\mu{'})\nabla\mathrm{div}{\bf w}_{0}))\in {\bf H}^{1}_{0}(\Omega).\\
\end{array}\right.
\end{equation}
 The following theorem corresponds to the existence and the regularity properties of the solution ${\bf w}$ of the system \eqref{2.1.1}.
\begin{thm}\label{t21} Let $m,$ $M$ be positive constants such that $m<M.$ Then for all $\overline{\sigma},$ $G_{2}$ and ${\bf w}_{0}$ satisfying \eqref{condvel} and \eqref{condvel2}, there exists a unique solution ${\bf w}$ of \eqref{2.1.1} which satisfies the following 
	\begin{equation}\label{spacvel}
	\begin{array}{ll}
	 &{\bf w}\in L^{2}(0,T;{\bf H}^{3}(\Omega))\cap C^{0}([0,T];{\bf H}^{5/2}(\Omega)),\,{\bf w}_{t}\in L^{2}(0,T;{\bf H}^{2}(\Omega))\cap C^{0}([0,T];{\bf H}^{1}(\Omega)),\\
	 &{\bf w}_{tt}\in L^{2}(0,T;{\bf L}^{2}(\Omega)).
	 \end{array}
	 \end{equation}
	  Besides, there exists a constant $c_{1}$ (depending on $m$ and $M$ but independent of $T,$ $\overline{\sigma},$ $G_{2}$ and ${\bf w}_{0}$) such that ${\bf w}$ satisfies the following inequality
	\begin{equation}\label{2.1.2}
	\begin{split}
	&\|{\bf w} \|_{L^{\infty}(0,T;{\bf H}^{2}(\Omega))}+\|{\bf w} \|_{L^{2}(0,T;{\bf H}^{3}(\Omega))}+\|{\bf w}_{t} \|_{L^{\infty}(0,T;{\bf H}^{1}(\Omega))}+\|{\bf w}_{t} \|_{L^{2}(0,T;{\bf H}^{2}(\Omega))}\\
	&+\|{\bf w}_{tt}\|_{L^{2}(0,T;{\bf L}^{2}(\Omega))}
	\leqslant c_{1}\{\|G_{2}\|_{L^{2}(0,T;{\bf H}^{1}(\Omega))} +\|G_{2} \|_{L^{\infty}(0,T;{\bf L}^{2}(\Omega))}\\
	&+\left(\|G_{2,t} \|_{L^{2}(0,T;{\bf L}^{2}(\Omega)}
	+\left\|\frac{ G_{2}\mid_{t=0}-(-\mu \Delta {\bf w}_{0} -(\mu+\mu{'})\nabla\mathrm{div}{\bf w}_{0})}{\overline{\sigma}(0)}\right\|_{{\bf H}^{1}(\Omega)}\right)\cdot (1+\|\overline{\sigma}_{t}\|_{L^{2}(0,T;L^{3}(\Omega))}\\
	&+\|\nabla\overline{\sigma}\|_{L^{2}(0,T;{\bf L}^{3}(\Omega))})
	\cdot\mathrm{exp}(c_{1}\|\overline{\sigma}_{t}\|^{2}_{L^{2}(0,T;L^{3}(\Omega))})\}.
	\end{split}
	\end{equation}
	%The constant $c_{1}$ depends on $m$ and $M$ but we do not write that explicitly since $m$ and $M$ are fixed (see \eqref{cor0}).
\end{thm}
\begin{remark}
	Observe from \eqref{spacvel} that ${\bf w}\in C^{0}([0,T];{\bf H}^{5/2}(\Omega))$ but in \eqref{2.1.2} we only include the estimate of $\|{\bf w}\|_{L^{\infty}(0,T;{\bf H}^{2}(\Omega))}$ and not of $\|{\bf w}\|_{L^{\infty}(0,T;{\bf H}^{5/2}(\Omega))}.$ Using interpolation one can recover an estimate of $\|{\bf w}\|_{L^{\infty}(0,T;{\bf H}^{5/2}(\Omega))}$ from the estimates of $\|{\bf w}\|_{L^{2}(0,T;{\bf H}^{3}(\Omega))}$ and $\|{\bf w}_{t}\|_{L^{2}(0,T;{\bf H}^{2}(\Omega))}$ where the constant of interpolation may depend on the final time $T.$ 
\end{remark}
\begin{remark}
	Using \eqref{condvel2} let us observe that $G_{2}\in L^{2}(0,T;{\bf H}^{1}(\Omega))\cap H^{1}(0,T;{\bf L}^{2}(\Omega))$ and hence by interpolation $G_{2}\mid_{t=0}\in {\bf H}^{1/2}(\Omega).$ Now from \eqref{condvel2}$_{3}$ one gets that 
	 $$(-\mu \Delta {\bf w}_{0} -(\mu+\mu{'})\nabla\mathrm{div}{\bf w}_{0})\in {\bf H}^{1/2}(\Omega).$$ 
	 The elliptic regularity result furnishes that ${\bf w}_{0}\in {\bf H}^{5/2}(\Omega).$ Since ${\bf w}\in C^{0}([0,T];{\bf H}^{5/2}(\Omega)),$ for the linear equation \eqref{2.1.1} we do not loose any regularity as time evolves. 
	 %But we have seen in Section \ref{incom} that there is a loss of $\frac{1}{2}$ space regularity for $\widehat{\bf u},$ (where $(\widehat{\sigma},\widehat{\bf u},\eta)$ is the solution of \eqref{1.16}).  
	 %As $(\rho_{0},{\bf u}_{0},\eta_{1})$ satisfy \eqref{cit}$(i)$ and $G_{2}(\sigma,{\bf w},\eta)\mid_{t=0}$ is given by \eqref{iG0}, the regularities of $\rho_{0},$ ${\bf u}_{0}$ and $\eta_{1},$ the compatibility condition \eqref{cit}$(i)$$(b)$ and bootstrap argument furnish $G_{2}\mid_{t=0}\in {\bf H}^{1}(\Omega).$ This is certainly not the case for the linear velocity equation \eqref{2.1.1} since \eqref{condvel2}$_{1}$ only provides $G_{2}\mid_{t=0}\in {\bf H}^{1/2}(\Omega)$.
\end{remark}
\begin{proof}[Proof of Theorem \ref{t21}]
	In the context of a smooth domain and with homogeneous Dirichlet boundary condition Theorem \ref{t21} is proved in the article \cite{vallizak}. There is no particular difficulty to adapt the same proof in $\Omega$ with $L$-periodic (in the $x$ direction) boundary condition. Hence we refer the readers to the proofs of \cite[Lemma 2.1]{vallizak}. For a related result we also refer the reader to \cite[Lemma 2.2]{valliperiodic}. 
%	\\
%	To remove the assumption $\overline\sigma \in C^0([0,T]; H^2(\Omega))$ and consider only $\overline\sigma$ satisfying \eqref{condvel}, we argue by density: if $\overline\sigma$ satisfies \eqref{condvel}, one can construct a sequence $\overline\sigma_n$ of functions in $C^0([0,T]; H^2(\Omega))$ such that $\overline\sigma_n$ strongly converges to $\overline \sigma$ in $L^2(0,T; {\bf W}^{1,3}(\Omega)) \cap H^1(0,T; L^3(\Omega))$ and satisfies, for each $n$, $m/2 \leqslant \overline \sigma_n \leqslant M$ and $m \leqslant \overline\sigma_n(0)\leqslant M$. Due to the bounds \eqref{2.1.2}, which hold uniformly with respect to $n$, this is enough to guarantee that the corresponding solutions ${\bf w}_n$ of \eqref{2.1.1} with $\overline\sigma_n$ in place of $\overline\sigma$ converge to the solution ${\bf w}$ of \eqref{2.1.1}, and that ${\bf w}$ satisfies the bound \eqref{2.1.2}.
	\end{proof}
		\subsection{Study of a continuity equation}\label{deneq}
		In this section we consider the following linear problem
		\begin{equation}\label{2.2.1}
		\left\{ \begin{array}{ll}
		\sigma_{t}+\overline {\bf w}\cdot \nabla\sigma=G_{1}\quad& \mbox{in}\quad Q_{T},
		\vspace{1.mm}\\
		\sigma(0)=\sigma_{0}\quad &\mbox{in}\quad \Omega,
		\end{array}\right.
		\end{equation}
		where the functions $\overline{\bf w},$ $G_{1}$ and $\sigma_{0}$ are $L$-periodic (in the $x$ direction) functions. The following theorem asserts the existence and regularity of the solution $\sigma$ of the density equation \eqref{2.2.1}.
		\begin{thm}\label{t22}
			Let $\overline {\bf w}\in L^{1}(0,T;{\bf H}^{3}(\Omega)),\,\overline {\bf w}\cdot {\bf n}= 0\,\mbox{on}\,\Sigma_{T},\,\sigma_{0}\in H^{2}(\Omega)\,\mbox{and}\,\, G_{1}\in L^{1}(0,T;H^{2}(\Omega)).$ Then there exists a unique solution $\sigma$ of \eqref{2.2.1} such that $\sigma\in C^{0}([0,T];H^{2}(\Omega))$ and 
			\begin{equation}\label{2.2.2}
			\begin{array}{l}
			\|\sigma\|_{L^{\infty}(0,T;H^{2}(\Omega))}\leqslant(\|\sigma_{0}\|_{H^{2}(\Omega)}+c_{2}\| G_{1}\|_{L^{1}(0,T;H^{2}(\Omega))})\mathrm{exp}(c_{2}\|\overline {\bf w}\|_{L^{1}(0,T;{\bf H}^{3}(\Omega))}).
			\end{array}
			\end{equation}
			If in addition $G_{1}\in L^{\infty}(0,T; H^{1}(\Omega))$ and $\overline{\bf w}\in L^{\infty}(0,T;{\bf H}^{2}(\Omega))$ then $\sigma_{t}\in L^{\infty}(0,T; H^{1}(\Omega))$ and
			\begin{equation}\label{2.2.3}
			\begin{split}
			\|\sigma_{t}\|_{L^{\infty}(0,T; H^{1}(\Omega))}\leqslant &c_{3}\|\overline {\bf w}\|_{L^{\infty}(0,T;{\bf H}^{2}(\Omega))}(\|\sigma_{0}\|_{H^{2}(\Omega)}+c_{2}\| G_{1}\|_{L^{1}(0,T;H^{2}(\Omega))})\\
			&\qquad\qquad\qquad\qquad\cdot\mathrm{exp}(c_{2}\|\overline {\bf w}\|_{L^{1}(0,T;{\bf H}^{3}(\Omega))})
			+\| G_{1}\|_{L^{\infty}(0,T; H^{1}(\Omega))}.
			\end{split}
			\end{equation}
			The constants $c_{2}$ and $c_{3}$ appearing respectively in \eqref{2.2.2} and \eqref{2.2.3} are independent of $T,$ $\overline{\bf w},$ $\sigma_{0}$ and $G_{1}.$
		\end{thm}
		\begin{proof}
			The theorem is proved in \cite[Lemma 2.4]{vallizak} with a particular expression of the function $G_{1}.$ In our case we adapt the same proof with minor changes.\\ 
			The existence of solution of \eqref{2.2.1} follows from the method of characteristics. The representation formula for the solution $\sigma$ is
			\begin{equation}\label{2.2.10}
			\begin{array}{l}
			\displaystyle
			\sigma(x,t)=\sigma_{0}(U(x,0,t))+\int\limits_{0}^{t}G_{1}(U(x,s,t),s)ds,
			\end{array}
			\end{equation}
			where $U(x,t,s)$ solves the following ODE
			\begin{equation}\label{2.2.11}
			\left\{ \begin{array}{ll}
			\displaystyle
			\frac{d}{dt}U(x,t,s)=\overline{\bf w}(U(x,t,s),t)\quad&\mbox{in}\quad Q_{T},
			\vspace{1.mm}\\
			\displaystyle
			U(x,s,s)=x\quad&\mbox{in}\quad \Omega.
			\end{array}\right.
			\end{equation}
			Observe 
			$$U(\cdot,\cdot,\cdot)\in C^{0}([0,T]\times[0,T];{\bf H}^{3}(\Omega))$$
			and consequently 
			$$\sigma(\cdot,\cdot)\in C^{0}([0,T];H^{2}(\Omega)).$$
			Now to prove the estimate \eqref{2.2.2}, we multiply \eqref{2.2.1}$_{1}$ by $\sigma$ and integrate in $\Omega.$ Integrating by parts the term $\displaystyle\int_{\Omega}\overline{\bf w}\cdot\nabla\sigma\sigma$ and using the fact that ${\overline{\bf w}}\cdot n=0$ we obtain
			\begin{equation}\nonumber
			\begin{array}{l}
			\displaystyle
			\frac{1}{2}\frac{d}{dt}\|\sigma\|^{2}_{L^{2}(\Omega)}\leqslant\frac{1}{2}\int_{\Omega}\mbox{div}{\overline{\bf w}}\sigma^{2}+\| G_{1}\|_{L^{2}(\Omega)}\|\sigma\|_{L^{2}(\Omega)}.
			\end{array}
			\end{equation}
			Due to the embedding ${\bf H}^{3}(\Omega)\hookrightarrow C^{1}(\overline{\Omega})$ one has
			\begin{equation}\label{2.2.4}
			\begin{array}{l}
			\displaystyle
			\frac{d}{dt}\|\sigma\|^{2}_{L^{2}(\Omega)}\leqslant c(\|{\overline{\bf w}}\|_{{\bf H}^{3}(\Omega)}\|\sigma\|_{L^{2}(\Omega)}^{2}+\| G_{1}\|_{L^{2}(\Omega)}\|\sigma\|_{L^{2}(\Omega)}).
			\end{array}
			\end{equation}
			Before going into the next estimate let us observe that
			\begin{equation}\label{2.2.5}
			\begin{array}{l}
			\displaystyle
			\int\limits_{\Omega}[(\overline{\bf w}\cdot\nabla)\nabla\sigma]\cdot\nabla\sigma=-\frac{1}{2}\int\limits_{\Omega}(\mbox{div}{\overline{\bf w}})|\nabla\sigma|^{2}.
			\end{array}
			\end{equation}
			Now take the gradient of \eqref{2.2.1}$_{1}$, multiply by $\nabla\sigma$ and integrate in $\Omega.$ Using \eqref{2.2.5} one obtains
			\begin{equation}\label{2.2.6}
			\begin{array}{l}
			\displaystyle
			\frac{d}{dt}\int_{\Omega}|\nabla\sigma|^{2}\leqslant c(\|{\overline{\bf w}}\|_{{\bf H}^{3}(\Omega)}\|\nabla\sigma\|_{{\bf L}^{2}(\Omega)}^{2}+\|\nabla G_{1}\|_{{\bf L}^{2}(\Omega)}\|\nabla\sigma\|_{{\bf L}^{2}(\Omega)}).
			\end{array}
			\end{equation}
			In a similar way for the second derivative we have
			\begin{equation}\label{2.2.7}
			\begin{array}{ll}
			\displaystyle
			\frac{d}{dt}\int\limits_{\Omega}|D^{2}\sigma|^{2}\leqslant c\big(\|{\overline{\bf w}}\|_{{\bf H}^{3}(\Omega)}\|D^{2}\sigma\|^{2}_{{\bf L}^{2}(\Omega)}+\int_{\Omega} | D^{2}{\overline{\bf w}}||\nabla\sigma||D^{2}\sigma|+\| D^{2}G_{1}\|_{{\bf L}^{2}(\Omega)}\|D^{2}\sigma\|_{{\bf L}^{2}(\Omega)}\big).
			\end{array}
			\end{equation}
			One has the following estimate
			\begin{equation}\label{2.2.8}
			\begin{split}
			\int_{\Omega} |D^{2}{\overline{\bf w}}||\nabla\sigma||D^{2}\sigma|&\leqslant \| D^{2}{\overline{\bf w}}\|_{{\bf L}^{3}(\Omega)}  \|\nabla\sigma\|_{{\bf L}^{6}(\Omega)}  \|D^{2}\sigma\|_{{\bf L}^{2}(\Omega)} 
			\vspace{1.mm}\\
			& \leqslant c\| D^{2}{\overline{\bf w}}\|_{{\bf L}^{3}(\Omega)}  \|\nabla\sigma\|_{ {\bf H}^{1}(\Omega)}\|D^{2}\sigma\|_{{\bf L}^{2}(\Omega)}.
			\end{split}
			\end{equation}
			The estimates \eqref{2.2.4} and \eqref{2.2.6}-\eqref{2.2.7}-\eqref{2.2.8} furnish the following
			\begin{equation}\label{2.2.9}
			\begin{array}{l}
			\displaystyle
			\frac{1}{2}\frac{d}{dt}\|\sigma\|^{2}_{H^{2}(\Omega)}\leqslant c(\|\overline{\bf w}\|_{{\bf H}^{3}(\Omega)}\|\sigma\|^{2}_{H^{2}(\Omega)}+\| G_{1}\|_{H^{2}(\Omega)}\|\sigma\|_{H^{2}(\Omega)}).
			\end{array}
			\end{equation}
			Now \eqref{2.2.2} is a consequence of \eqref{2.2.9} and Gronwall lemma.
			Finally the estimate \eqref{2.2.3} is a direct consequence of  \eqref{2.2.1}$_{1}$ and \eqref{2.2.2}.
		\end{proof}
		The following corollary directly follows from \eqref{2.2.1}$_{1}$ and the regularity $\sigma\in C^{0}([0,T];H^{2}(\Omega))$ which we have obtained in Theorem \ref{t22}.
		\begin{corollary}\label{dencor}
			In addition to the assumptions of Theorem \ref{t22} if $G_{1}\in C^{0}([0,T]; H^{1}(\Omega))$ and $\overline{\bf w}\in C^{0}([0,T];{\bf H}^{2}(\Omega))$ then $\sigma_{t}\in C^{0}([0,T]; H^{1}(\Omega)).$
		\end{corollary}
			\subsection{Study of a linear beam equation}\label{beameq}
			The linearized beam equation with a non homogeneous source term is the following
			\begin{equation}\label{2.3.1}
			\left\{ \begin{array}{ll}
			\eta_{tt}-\beta \eta_{xx}-  \delta\eta_{txx}+\alpha\eta_{xxxx}=G_{3} \quad& \mbox{in}\quad\Sigma^{s}_{T},
			\vspace{1.mm}\\
			\eta(0)=0\quad\mbox{and}\quad \eta_{t}(0)=\eta_{1}\quad&\mbox{in}\quad \Gamma_{s},
			\end{array}\right. 
			\end{equation}
			where $G_{2}$ and $\eta_{1}$ are known $L$-periodic (in the $x$ direction) functions.
			Let us denote
			\begin{equation}\label{2.3.2}
			\mathcal{A}=\begin{bmatrix}
			0 & & I\\[1.mm]
			-\alpha\Delta^{2}+\beta\Delta & & \delta\Delta
			\end{bmatrix}.
			\end{equation}
			The unbounded operator $(\mathcal{A},D(\mathcal{A}))$ is defined in
			\begin{equation}\label{doHs}
			\begin{array}{l}
			H_{s}=H^{2}(\Gamma_{s})\times L^{2}(\Gamma_{s}),
			\end{array}
			\end{equation}
			with domain
			$$\quad D(\mathcal{A})=H^{4}(\Gamma_{s})\times H^{2}(\Gamma_{s}).$$
			Hence with the notations
			\begin{equation}\label{vecnot}
			\begin{array}{l}
			$${\bf Y}(t)=\begin{bmatrix}
			\eta(t)\\[1.mm]
			\eta_{t}(t)\\
			\end{bmatrix},\quad {\bf Y}_{0}=\begin{bmatrix}
			0\\[1.mm]
			\eta_{1}\\
			\end{bmatrix}\quad\mbox{and}\quad \widetilde{G}_{3}=\begin{bmatrix}
			0\\[1.mm]
			G_{3}
			\end{bmatrix},
			\end{array}
			\end{equation}
			we can equivalently write \eqref{2.3.1} as
			\begin{equation}\label{2.3.3}
			\left\{ \begin{array}{ll}
			{\bf Y}_{t}(t)=\mathcal{A}{\bf Y}(t)+\widetilde{G}_{3}\,&\mbox{on}\quad (0,T),
			\vspace{1.mm}\\
			{\bf Y}(0)={\bf Y}_{0}.
			\end{array}\right.
			\end{equation}
			\begin{lem}\label{t23p}
				Let 
				\begin{equation}\label{gy0}
				\begin{array}{l}
				\widetilde{G}_{3}\in L^{2}(0,T;{H}^{2}(\Gamma_{s})\times L^{2}(\Gamma_{s}))\quad \mbox{and}\quad {\bf Y}_{0}\in H^{3}(\Gamma_{s})\times H^{1}(\Gamma_{s}).
				\end{array}
				\end{equation}
				 Then the equation \eqref{2.3.3} admits a unique solution ${\bf Y}$ which satisfies
				\begin{equation}\label{regY}
				\begin{array}{l}
				{\bf Y}\in L^{2}(0,T;H^{4}(\Gamma_{s})\times H^{2}(\Gamma_{s}))\cap H^{1}(0,T;H^{2}(\Gamma_{s})\times L^{2}(\Gamma_{s}))\cap C^{0}([0,T];H^{3}(\Gamma_{s})\times H^{1}(\Gamma_{s})).
				\end{array}
				\end{equation}
			In addition if
			\begin{equation}\label{gy01}
		    \begin{array}{l}
			 \widetilde{G}_{3,t}\in L^{2}(0,T;{H}^{2}(\Gamma_{s})\times L^{2}(\Gamma_{s}))\quad \mbox{and}\quad \mathcal{A}{\bf Y}_{0}+\widetilde{G}_{3}\mid_{t=0}\in  H^{3}(\Gamma_{s})\times  H^{1}(\Gamma_{s}),
			 \end{array}
			 \end{equation}
			the solution ${\bf Y}$ of the problem \eqref{2.3.3} has the following additional regularities
			\begin{equation}\label{regY2}
			\begin{array}{l}
			{\bf Y}_{t}\in L^{2}(0,T; H^{4}(\Gamma_{s})\times H^{2}(\Gamma_{s}))\cap C^{0}([0,T];H^{3}(\Gamma_{s})\times  H^{1}(\Gamma_{s})),
			\vspace{1.mm}\\
			{\bf Y}_{tt}\in L^{2}(0,T; H^{2}(\Gamma_{s})\times L^{2}(\Gamma_{s})).
			\end{array}
			\end{equation}
			\end{lem}
			\begin{proof}
				To prove this result we will use the maximal parabolic regularity results from \cite{ben}. Recall the definition of $H_{s}$ in \eqref{doHs}. The unbounded operator $(\mathcal{A},D(\mathcal{A}))$ is the infinitesimal generator of an analytic semigroup on $H_{s}$ (for the proof see \cite{chen}).
				    Hence using the isomorphism theorem \cite[Theorem 3.1, p. 143]{ben} 
					%\begin{equation}\nonumber
				%	\begin{matrix}
				%	L^{2}(0,T;D(\mathcal{A}))\cap H^{1}(0,T;H_{s}) & \rightarrow & L^{2}(0,T;H_{s})\times D(\mathcal{A}^{1/2}),\\[1.mm]
				%	{\bf Y} & \mapsto & ({\bf Y}_{t}-\mathcal{A}{\bf Y},{\bf Y}_{0}),
				%	\end{matrix}
				%	\end{equation}
					and the assumption \eqref{gy0}, which can be read as $\widetilde{G}_{3}\in L^{2}(0,T;H_{s})$ and ${\bf Y}_{0}\in {D}(\mathcal{A}^{1/2}),$ we get that the equation \eqref{2.3.3} admits a unique solution ${\bf Y}$ satisfying the following:
					$${\bf Y}\in L^{2}(0,T;H^{4}(\Gamma_{s})\times H^{2}(\Gamma_{s}))\cap H^{1}(0,T;H^{2}(\Gamma_{s})\times L^{2}(\Gamma_{S})).$$ 
					Using interpolation (see \cite{liomag2}) one also obtains that 
					$${\bf Y}\in C^{0}([0,T];H^{3}(\Gamma_{s})\times H^{1}(\Gamma_{s})).$$
					This proves \eqref{regY}.\\
				    Now we assume that \eqref{gy01} holds. In order to obtain the time regularity of ${\bf Y}$ let us differentiate \eqref{2.3.3} with respect to $t$ and write ${\bf Z}={\bf Y}_{t},$	\begin{equation}\label{2.3.7}
				    \left\{ \begin{array}{l}
				    {\bf Z}_{t}(t)=\mathcal{A}{\bf Z}(t)+\widetilde{G}_{3,t}\quad\mbox{on}\quad (0,T),
				    \vspace{1.mm}\\
				    {\bf Z}(0)={\bf Z}^{0}=\mathcal{A}{\bf Y}_{0}+\widetilde{G}_{3}\mid_{t=0}.
				    \end{array}\right.
				    \end{equation} 
				    Due to the assumptions \eqref{gy01}, $ \widetilde{G}_{3,t}\in L^{2}(0,T;H_{s})$ and $\mathcal{A}{\bf Y}_{0}+\widetilde{G}_{3}\mid_{t=0}\in {D}(\mathcal{A}^{1/2})$ ($=H^{3}(\Gamma_{s})\times H^{1}(\Gamma_{s})$). We can use the isomorphism theorem \cite[Theorem 3.1, p. 143]{ben} again to conclude
				    $${\bf Z}={\bf Y}_{t}\in L^{2}(0,T;H^{4}(\Gamma_{s})\times H^{2}(\Gamma_{s}))\cap H^{1}(0,T;H^{2}(\Gamma_{s})\times L^{2}(\Gamma_{s})).$$
				   Once again using interpolation we verify that
				   $${\bf Y}_{t}\in C^{0}([0,T];H^{3}(\Gamma_{s})\times H^{1}(\Gamma_{s})).$$
				    This completes the proof of Lemma \ref{t23p}.
				\end{proof}
				We are going to use the representation \eqref{2.3.3} of \eqref{2.3.1} to state the existence and regularity result for the problem \eqref{2.3.1}.
			\begin{thm}\label{t23}
				Assume that $T<\overline{T}$ (recall that $\overline{T}$ was fixed in \eqref{Tbar}), ${G}_{3}\in L^{\infty}(0,T;{H}^{1/2}(\Gamma_{s}))$ and ${G}_{3,t}\in L^{2}(0,T;{L}^{2}(\Gamma_{s})).$ Also suppose that
				$\eta_{1}\in H^{3}(\Gamma_{s})$ and
				${G}_{3}\mid_{t=0}\in H^{1}(\Gamma_{s}).$
				Then the equation \eqref{2.3.1} admits a unique solution $\eta$ which satisfies
				\begin{equation}\label{roe}
				\begin{array}{l}
				\eta\in L^{\infty}(0,T;H^{9/2}(\Gamma_{s})),
				\vspace{1.mm}\\
				\eta_{t}\in L^{2}(0,T;H^{4}(\Gamma_{s}))\cap C^{0}([0,T];H^{3}(\Gamma_{s})),
				\vspace{1.mm}\\
				\eta_{tt}\in L^{2}(0,T;H^{2}(\Gamma_{s}))\cap C^{0}([0,T]; H^{1}(\Gamma_{s})),
				\vspace{1.mm}\\
				\eta_{ttt}\in L^{2}(0,T;L^{2}(\Gamma_{s})),
				\end{array}
				\end{equation}
				and for some positive constant $c_{4}$ independent of $T,$ $G_{3}$ and $\eta_{1}$ we have the following estimate
				\begin{equation}\label{2.3.5}
				\begin{split}
				\|\eta\|_ {L^{\infty}(0,T;H^{9/2}(\Gamma_{s}))} &+\|\eta_{t}\|_{L^{2}(0,T;H^{4}(\Gamma_{s}))}+\|\eta_{t}\|_{L^{\infty}(0,T;H^{3}(\Gamma_{s}))}+\|\eta_{tt}\|_{ L^{2}(0,T;H^{2}(\Gamma_{s}))}\\
				&+\|\eta_{tt}\|_{L^{\infty}(0,T; H^{1}(\Gamma_{s}))}
				+\|\eta_{ttt}\|_{L^{2}(0,T;L^{2}(\Gamma_{s}))}\leqslant c_{4}\big(\|\eta_{1}\|_{H^{3}(\Gamma_{s})}\\
				&+\|G_{3}\mid_{t=0}\|_{ H^{1}(\Gamma_{s})}
				+\|G_{3}\|_{L^{\infty}(0,T;H^{1/2}(\Gamma_{s}))}+\|{G}_{3,t}\|_{ L^{2}(0,T;{L}^{2}(\Gamma_{s}))}\big).
				\end{split}
				\end{equation}
			\end{thm}
			\begin{proof}
				We first consider
				\begin{equation}\label{pras}
				\begin{array}{l}
				G_{3}\in L^{2}(0,T;L^{2}(\Gamma_{s}))\quad\mbox{and}\quad \eta_{1}\in H^{1}(\Gamma_{s}).
				\end{array}
				\end{equation}
				In view of the notations \eqref{vecnot}, \eqref{pras} corresponds to the case \eqref{gy0} of Lemma \ref{t23p}. Hence we can use \eqref{regY} to obtain
				\begin{equation}\label{2.3.6}
				\begin{split}
				\|\eta\|_{L^{2}(0,T;H^{4}(\Gamma_{s}))}+	\|\eta_{t}\|_{L^{2}(0,T;H^{2}(\Gamma_{s}))}+\|\eta_{tt}\|_{L^{2}(0,T;L^{2}(\Gamma_{s}))}\leqslant c\big(\|\eta_{1}\|_{ H^{1}(\Gamma_{s})}  
			   +\|G_{3}\|_{L^{2}(0,T;L^{2}(\Gamma_{s}))}\big),
				\end{split}
				\end{equation}
				 where the constant $c$ might depend on the final time $T.$ We want to show that there exists a constant $c$ independent of $T$ such that the inequality \eqref{2.3.6} is true. For that we extend ${G}_{3}$ by defining it zero in $(T,\overline{T})$ and denote the extended function also by ${G}_{3}.$ Observe that ${G}_{3}\in L^{2}(0,\overline{T};L^{2}(\Gamma_{s})).$ We can solve \eqref{2.3.1} in the time interval $(0,\overline{T})$ and consequently
				\begin{equation}\label{etaindT}
				\begin{split}
				&\|\eta\|_{L^{2}(0,{T};H^{4}(\Gamma_{s}))}+	\|\eta_{t}\|_{L^{2}(0,{T};H^{2}(\Gamma_{s}))}+\|\eta_{tt}\|_{L^{2}(0,{T};L^{2}(\Gamma_{s}))}\\
				&	\leqslant\|\eta\|_{L^{2}(0,\overline{T};H^{4}(\Gamma_{s}))}+	\|\eta_{t}\|_{L^{2}(0,\overline{T};H^{2}(\Gamma_{s}))}+\|\eta_{tt}\|_{L^{2}(0,\overline{T};L^{2}(\Gamma_{s}))}\\[1.mm]
				& \leqslant c(\overline{T})\big(\|\eta_{1}\|_{ H^{1}(\Gamma_{s})}
				+\|G_{3}\|_{L^{2}(0,\overline{T};L^{2}(\Gamma_{s}))}\big)\\[1.mm]
				& = c(\overline{T})\big(\|\eta_{1}\|_{ H^{1}(\Gamma_{s})}
				+\|G_{3}\|_{L^{2}(0,{T};L^{2}(\Gamma_{s}))}\big).
				\end{split}
				\end{equation}
				So we are able to get a constant $c(\overline{T})$ which is independent of $T.$\\ 
				To prove the regularity estimates of $\eta_{t},$ we will use
				\begin{equation}\label{tasG3}
				\begin{array}{l}
				G_{3,t}\in L^{2}(0,T;L^{2}(\Gamma_{s})),\quad\eta_{1}\in H^{3}(\Gamma_{s})\quad\mbox{and}\quad G_{3}\mid_{t=0}\in H^{1}(\Gamma_{s}).
				\end{array}
				\end{equation}
				Indeed, observe that \eqref{tasG3} implies $\delta\Delta\eta_{1}+G_{3}\mid_{t=0}\in H^{1}(\Gamma_{s}).$ Now differentiate the equation \eqref{2.3.1} with respect to $t,$ 
				\begin{equation}\label{zeta}
					\left\{ \begin{array}{ll}
					(\eta_{t})_{tt}-\beta (\eta_{t})_{xx}-  \delta(\eta_{t})_{txx}+\alpha(\eta_{t})_{xxxx}=G_{3,t} \quad& \mbox{on}\quad\Sigma^{s}_{T},
					\vspace{1.mm}\\
					\eta_{t}(0)=\eta_{1}\quad\mbox{and}\quad \eta_{tt}(0)=\delta\Delta\eta_{1}+G_{3}\mid_{t=0}\quad&\mbox{in}\quad \Gamma_{s}.
					\end{array}\right.
				\end{equation}
				In view of the notations \eqref{vecnot}, \eqref{tasG3} and \eqref{zeta} correspond respectively to \eqref{gy01} and \eqref{2.3.7} in Lemma \ref{t23p}. 
				Hence we can use \eqref{regY2} to furnish the following
				 \begin{equation}\label{2.3.8*}
				 \begin{split}
				 &\|\eta_{t}\|_{L^{2}(0,T;H^{4}(\Gamma_{s}))}+	\|\eta_{tt}\|_{L^{2}(0,T;H^{2}(\Gamma_{s}))}+\|\eta_{ttt}\|_{L^{2}(0,T;L^{2}(\Gamma_{s}))}\\[1.mm]
				 &\leqslant   c\big(\|\eta_{1}\|_{H^{3}(\Gamma_{s})}+\|G_{3}\mid_{t=0}\|_{ H^{1}(\Gamma_{s})}  
				 +\|{G}_{3,t}\|_{L^{2}(0,T;L^{2}(\Gamma_{s}))}\big),
				 \end{split}
				 \end{equation}
				where the constant $c$ might depend on the final time $T.$ Since we are interested in proving \eqref{2.3.8*} with a constant $c$ independent of $T,$ we extend the function ${G}_{3,t}$ by defining it zero in the interval $(T,\overline{T})$ and denote the extended function also by ${G}_{3,t}.$ In a similar spirit of the computation \eqref{etaindT} one can prove
			    \begin{equation}\label{2.3.8}
				\begin{split}
				&\|\eta_{t}\|_{L^{2}(0,T;H^{4}(\Gamma_{s}))}+	\|\eta_{tt}\|_{L^{2}(0,T;H^{2}(\Gamma_{s}))}+\|\eta_{ttt}\|_{L^{2}(0,T;L^{2}(\Gamma_{s}))}\\[1.mm]
				&\leqslant   c(\overline{T})\big(\|\eta_{1}\|_{H^{3}(\Gamma_{s})}+\|G_{3}\mid_{t=0}\|_{ H^{1}(\Gamma_{s})}  
				+\|{G}_{3,t}\|_{L^{2}(0,T;L^{2}(\Gamma_{s}))}\big)
				\end{split}
				\end{equation}
				for some constant $c(\overline{T})$ independent on $T.$
				%Directly by observing \eqref{2.3.7} and using \eqref{2.3.8} we get at once
				%$$	{\bf Y}_{tt}={\bf Z}_{t}\in L^{2}(0,T; H^{2}(\Gamma_{s})\times L^{2}(\Gamma_{s}))$$
				%and
				%\begin{equation}\label{2.3.9}
				%\begin{array}{l}
				%\|\eta_{ttt}\|_{L^{2}(0,T;L^{2}(\Gamma_{s}))}^{2}\leqslant c\big(\|\eta_{1}\|^{2}_{H^{3}}+\|{G}_{3,t}\|^{2}_{L^{2}(0,T;L^{2}(\Gamma_{s}))} \big).
				%\end{array}
				%\end{equation}
				%Using interpolation (see \cite{liomag2}) one obtains
				%$${\bf Y}\in C^{1}([0,T];H^{3}(\Gamma_{s})\times  H^{1}(\Gamma_{s})).$$ 
				In order to get explicit bounds on the $L^{\infty}(0,T)$ norms of $\eta,$ $\eta_{t}$ and $\eta_{tt}$ we first multiply \eqref{2.3.1}$_{1}$ by $\eta_{txx}$ and integrate over $\Gamma_{s}.$ We use the $L$-periodicity (in the $x$ direction) of $\eta$ and integrate the terms by parts to obtain 
				\begin{equation}\label{2.3.10}
				\begin{array}{ll}
				\displaystyle
				\frac{1}{2}\frac{d}{d t}\int\limits_{\Gamma_{s}}\eta^{2}_{tx}dx+\frac{\beta}{2}\frac{d}{d t}\int\limits_{\Gamma_{s}}\eta^{2}_{xx}dx+\delta\int\limits_{\Gamma_{s}}\eta^{2}_{txx}dx+\frac{\alpha}{2}\frac{d}{d t}\int\limits_{\Gamma_{s}}\eta^{2}_{xxx}dx\leqslant \frac{\delta}{8}\|\eta_{txx}(t)\|^{2}_{L^{2}(\Gamma_{s})}+\frac{2}{\delta}\|G_{3}\|^{2}_{L^{2}(\Gamma_{s})}.
				\end{array}
				\end{equation}
				Now integrating \eqref{2.3.10} with respect to $t,$
				\begin{equation}\label{2.3.11}
				\begin{split}
				\|\eta_{tx}\|^{2}_{L^{\infty}(0,T;L^{2}(\Gamma_{s}))}  +\|\eta_{xx}\|^{2}_{L^{\infty}(0,T;L^{2}(\Gamma_{s}))}&  +\|\eta_{txx}\|^{2}_{L^{2}(0,T;L^{2}(\Gamma_{s}))}+\|\eta_{xxx}\|^{2}_{L^{\infty}(0,T;L^{2}(\Gamma_{s}))}\\[1.mm]
				& \leqslant c\big(\|\eta_{tx}(0)\|^{2}_{L^{2}(\Gamma_{s})}+\|G_{3}\|^{2}_{L^{2}(0,T;L^{2}(\Gamma_{s}))} \big).
				\end{split}
				\end{equation}
				From \eqref{2.3.11} we get in particular
				\begin{equation}\label{2.3.12}
				\begin{array}{l}
				\|\eta\|^{2}_{L^{\infty}(0,T;H^{3}(\Gamma_{s}))}\leqslant c\big(\|\eta_{1}\|^{2}_{ H^{1}(\Gamma_{s})}+\|G_{3}\|^{2}_{L^{2}(0,T;L^{2}(\Gamma_{s}))}\big).
				\end{array}
				\end{equation}
				Now consider the equations \eqref{zeta}. One imitates the analysis used to obtain \eqref{2.3.11} to find
				\begin{equation}\label{2.3.13}
				\begin{split}
				\|\eta_{ttx}\|^{2}_{L^{\infty}(0,T;L^{2}(\Gamma_{s}))}  +\|\eta_{txx}\|^{2}&_{L^{\infty}(0,T;L^{2}(\Gamma_{s}))}  +\|\eta_{ttxx}\|^{2}_{L^{2}(0,T;L^{2}(\Gamma_{s}))}+\|\eta_{txxx}\|^{2}_{L^{\infty}(0,T;L^{2}(\Gamma_{s}))}\\[1.mm]
				& \leqslant c\big(\|\eta_{1}\|^{2}_{H^{3}(\Gamma_{s})}+\|G_{3}\mid_{t=0}\|^{2}_{ H^{1}(\Gamma_{s})}+\|{G}_{3,t}\|^{2}_{L^{2}(0,T;L^{2}(\Gamma_{s}))} \big).
				\end{split}
				\end{equation}
				Hence in particular
				\begin{equation}\label{2.3.14}
				\begin{split}
				\|\eta_{t}\|^{2}_{L^{\infty}(0,T;H^{3}(\Gamma_{s}))}+ \|\eta_{tt}\|^{2}_{L^{\infty}(0,T; H^{1}(\Gamma_{s}))}\leqslant c\big(\|\eta_{1}\|^{2}_{H^{3}(\Gamma_{s})}&+\|G_{3}\mid_{t=0}\|^{2}_{ H^{1}(\Gamma_{s})}\\
				&+\|{G}_{3,t}\|^{2}_{L^{2}(0,T;L^{2}(\Gamma_{s}))}\big).
				\end{split}
				\end{equation}
				Now we will use that 
				\begin{equation}\label{g3linf}
				\begin{array}{l}
				G_{3}\in L^{\infty}(0,T;H^{1/2}(\Gamma_{s})).
				\end{array}
				\end{equation}
				Write \eqref{2.3.1}$_{1}$ as
				\begin{equation}\label{2.3.15}
				\begin{array}{l}
				\eta_{xxxx}=\frac{1}{\alpha}\big(G_{3}+\delta\eta_{txx}+\beta\eta_{xx}-\eta_{tt} \big).
				\end{array}
				\end{equation}
				In view of \eqref{g3linf} one observes that all the terms appearing in the right hand side of \eqref{2.3.15} belongs to $L^{\infty}(0,T;H^{1/2}(\Gamma_{s})).$
				As the beam in our problem is one dimensional, $\eta\in L^{\infty}(0,T;H^{9/2}(\Gamma_{s}))$ and the estimates \eqref{2.3.12} and \eqref{2.3.14} furnish the following 
				\begin{equation}\label{2.3.16}
				\begin{split}
				\|\eta\|_{L^{\infty}(0,T;H^{9/2}(\Gamma_{s}))}\leqslant  c\big(\|\eta_{1}\|_{H^{3}(\Gamma_{s})}+\|G_{3}\mid_{t=0}\|_{ H^{1}(\Gamma_{s})}
				&+\|{G}_{3,t}\|_{L^{2}(0,T;L^{2}(\Gamma_{s}))}\\
				&+\|G_{3}\|_{L^{\infty}(0,T;H^{1/2}(\Gamma_{s}))}\big).
				\end{split}
				\end{equation}
				Hence combining all the above estimates we here conclude the proof of Theorem \ref{t23}.
			\end{proof}
			The following corollary follows directly by using the regularities \eqref{roe} and the expression \eqref{2.3.15} of $\eta_{xxxx}.$
			\begin{corollary}\label{timebeam}
				In addition to the assumptions of Theorem \ref{t23} if $G_{3}$ further satisfies the regularity assumption $G_{3}\in C^{0}([0,T];H^{1/2}(\Gamma_{s}))$ then $\eta\in C^{0}([0,T];H^{9/2}(\Gamma_{s})).$
			\end{corollary}
			\section{Local existence of the non linear coupled system}\label{fixdpt}
			From now on up to the end of this article, we fix the initial data $(\rho_{0},{\bf u}_{0},\eta_{1})$ such that they satisfy the assumptions stated in \eqref{cit}.  We also fix the constant 
			\begin{equation}\label{fixd0}
			\begin{array}{l}
			\delta_{0}\in (0,1).
			\end{array}
			\end{equation} 
			The constant $\delta_{0}$ will be used to keep a positive distance between the beam and the bottom $\Gamma_{\ell}$ of the domain $\Omega.$ Also recall that the positive constants $m$ and $M$ were fixed in \eqref{cor0} and $\overline{T}$ was fixed in \eqref{Tbar}.
				\begin{proof}[Proof of Theorem \ref{main}]
				This section is devoted to the study of the non linear system \eqref{chdb}. We will prove here that the system \eqref{chdb} admits a strong solution in a time interval $(0,T),$ for some $T>0$ small enough and hence we will conclude Theorem \ref{main}.\\
			    Now we sketch the steps towards the proof of Theorem \ref{main}:\\
				(i) First in Section \ref{dfp} we define a suitable map for which a fixed point gives a solution of the system \eqref{chdb}.\\
				(ii) Next we design a suitable convex set such that the map defined in step (i) maps this set into itself. This is done in Section \ref{Enht}. \\
				(iii) In Section \ref{comcon} we show that the convex set defined in step (ii) is compact in some appropriate topology. We further prove that the fixed point map from step (i), is continuous in that topology.\\
				(iv)  At the end in Section \ref{conc} we draw the final conclusion to prove Theorem \ref{main}.\\
				In what follows all the constants appearing in the inequalities may vary from line to line but will never depend on $T.$
				 \subsection{Definition of the fixed point map}\label{dfp}
				 For $(\widetilde{\sigma},\widetilde{\bf w},\widetilde{\eta})$ satisfying
				 \begin{equation}\label{regswe}
				 \left\{ \begin{array}{l}
				 \widetilde{\sigma}\in L^{\infty}(0,T;H^{2}(\Omega))\cap W^{1,\infty}(0,T;H^{1}(\Omega)),\\
				 \widetilde{\bf w}\in L^{\infty}(0,T;{\bf H}^{5/2}(\Omega))\cap L^{2}(0,T;{\bf H}^{3}(\Omega))\cap W^{1,\infty}(0,T;{\bf H}^{1}(\Omega))\cap H^{1}(0,T;{\bf H}^{2}(\Omega))\\
				 \qquad\cap H^{2}(0,T;{\bf L}^{2}(\Omega)),\\
				 \widetilde{\eta}\in L^{\infty}(0,T;H^{9/2}(\Gamma_{s}))\cap W^{1,\infty}(0,T;H^{3}(\Gamma_{s}))\cap H^{1}(0,T;H^{4}(\Gamma_{s}))\cap W^{2,\infty}(0,T;H^{1}(\Gamma_{s}))\\
				 \qquad \cap H^{2}(0,T;H^{2}(\Gamma_{s}))\cap H^{3}(0,T;L^{2}(\Gamma_{s})),
				 \end{array}\right.
				 \end{equation}
		we consider the following problem:
				 \begin{equation}\label{3.2}
				 \left\{
				 \begin{array}{lll}
				 &{\sigma_{t}}+\widetilde{W}(\widetilde{\bf w},\widetilde{\eta})
				 \cdot \nabla\sigma=G_{1}(\widetilde{\sigma},\widetilde{\bf w},\widetilde{\eta})&\quad \mbox{in}\quad Q_{T},
				 \vspace{1.mm}\\
				 &(\widetilde\sigma+\overline{\rho}){\bf w}_{t}-\mu \Delta {\bf w} -(\mu+\mu{'})\nabla\mathrm{div}{\bf w}=G_{2}(\widetilde{\sigma},\widetilde{\bf w},\widetilde{\eta})&\quad \mbox{in} \quad {Q}_{T},
				 \vspace{1.mm}\\
				 &{\bf w}=0&\quad \mbox{on}\quad \Sigma_{T},
				 \vspace{1.mm}\\
				 &{\bf w}(\cdot,0)={\bf w}_{0}={\bf u}_{0}-z\eta_{1}\vec{e}_{2}&\quad \mbox{in} \quad \Omega,
				 \vspace{1.mm}\\
				 &\sigma(\cdot,0)=\sigma_{0}={\rho_{0}}-{\overline\rho}&\quad\mbox{in}\quad \Omega,
				 \vspace{1.mm}\\
				 &\eta_{tt}-\beta \eta_{xx}-  \delta\eta_{txx}+\alpha\eta_{xxxx}=G_{3}(\widetilde{\sigma},\widetilde{\bf w},\widetilde{\eta}) &\quad\mbox{on}\quad \Sigma^{s}_{T},
				 \vspace{1.mm}\\
				 &\eta(0)=0\quad \mbox{and}\quad \eta_{t}(0)=\eta_{1}&\quad\mbox{in}\quad \Gamma_{s},
				 \end{array} \right.
				 \end{equation}
				 where $G_{1},$ $G_{2},$ $G_{3}$ are as defined in \eqref{1.22} and $\widetilde{W}(\widetilde{\bf w},\widetilde{\eta})$ is defined as follows
				 \begin{equation}\label{dotW}
				 \begin{array}{l}
				 \displaystyle
				 \widetilde{W}(\widetilde{\bf w},\widetilde{\eta})=\begin{bmatrix}
				 \widetilde{w}_{1}\\
				 \frac{1}{(1+\widetilde{\eta})}(\widetilde{w}_{2}-\widetilde{w}_{1}z\widetilde{\eta}_{x})
				 \end{bmatrix},\qquad \left(\widetilde{\bf w}=\begin{pmatrix}
				 \widetilde{w}_{1}\\
				 \widetilde{w}_{2}
				 \end{pmatrix}\right).
				 \end{array}
				 \end{equation}
				 
				 It turns out that it will be important for us to check that $G_{2}(\widetilde{\sigma},\widetilde{\bf w},\widetilde{\eta})$ and $ G_{3}(\widetilde{\sigma},\widetilde{\bf w},\widetilde{\eta})$ respectively coincide at time $t = 0$ with the values $G_2^0$ and $G_3^0$ computed in \eqref{iG0} and \eqref{G30}, and given as follows:
				 \begin{equation}\label{iG0-bis}
\begin{split}
G_{2}^0&=-P'(\rho_{0})\nabla\rho_{0}-(\delta\eta_{1,xx}-(\mu+2\mu')({ u}_{0})_{2,z}
+P(\rho_{0}))z\rho_{0}\vec{e}_{2}
+z\rho_{0}({\bf u}_{0})_{z}\eta_{1}\\
&-\rho_{0}({\bf u}_{0}\cdot\nabla){\bf u}_{0}-(-\mu\Delta-(\mu+\mu{'})\nabla\mathrm{div})(z\eta_{1}\vec{e_{2}}),
\end{split}
\end{equation}
\begin{equation}\label{G30-bis}
\begin{array}{l}
G_{3}^0=-(\mu+2\mu')({u}_{0})_{2,z}+P(\rho_{0}).
\end{array}
\end{equation} 
				 This will be imposed by assuming $(\widetilde{\sigma},\widetilde{\bf w},\widetilde{\eta},\widetilde{\eta}_{t})(\cdot,0)=(\sigma_{0},{\bf w}_{0},0,\eta_{1})$ and 
				 \begin{equation}\label{tetatt0}
				 \begin{array}{l}
				\displaystyle \widetilde{\eta}_{tt}(\cdot,0)=\delta\eta_{1,xx}-(\mu+2\mu')(u_{0})_{2,z}+P(\rho_{0})\,\,\mbox{in}\,\,\Omega, 
				 \\
	 		 	\displaystyle \widetilde{\bf w}_{t}(\cdot,0)=\frac{1}{\rho_{0}}\big(G_{2}^0-(-\mu \Delta -(\mu+\mu{'})\nabla\mathrm{div})({\bf u}_{0}-z\eta_{1}\vec{e_{2}})\big)\,\,\mbox{in}\,\,\Omega,
				 \end{array}
				 \end{equation}
				Indeed, under the above conditions, one can check from the expressions of $G_{2}(\widetilde{\sigma},\widetilde{\bf w},\widetilde{\eta})$ and $G_{3}(\widetilde{\sigma},\widetilde{\bf w},\widetilde{\eta})$ that  $G_{2}(\widetilde{\sigma},\widetilde{\bf w},\widetilde{\eta})\mid_{t=0} = G_2^0$ and $G_{3}(\widetilde{\sigma},\widetilde{\bf w},\widetilde{\eta})\mid_{t=0} = G_3^0$.
				 \begin{lem}\label{wellpos}
					Let the constant $\delta_{0}$ be fixed by \eqref{fixd0}.
				 	For $T < \overline{T}$, let us assume the following 
				 	\begin{equation}\label{tsweY}
				 	\begin{array}{l}
				 	(\widetilde{\sigma},\widetilde{\bf w},\widetilde{\eta})\,\,\mbox{satisfies}\,\,\eqref{regswe},
				 	\end{array}
				 	\end{equation}
				 	\begin{equation}\label{tweq0*}
				 	\begin{array}{l}
				 	\widetilde{\bf w}=0\,\,\mathrm{on}\,\,\Sigma_{T},
				 	\end{array}
				 	\end{equation}
				 	\begin{equation}\label{tswe0*}
				 	\begin{array}{l}
				 	(\widetilde{\sigma}(\cdot,0),\widetilde{\bf w}(\cdot,0),\widetilde{\eta}(\cdot,0),\widetilde{\eta}_{t}(\cdot,0))=(\rho_{0}-\overline{\rho},{\bf u}_{0}-z\eta_{1}\vec{e}_{2},0,\eta_{1})\,\,\mbox{in}\,\,\Omega,
				 	\end{array}
				 	\end{equation}
				 	\begin{equation}\label{etatt0*}
				 	\begin{array}{l}
				 	\eqref{tetatt0}\,\,\mbox{holds},
				 	\end{array}
				 	\end{equation}
%				 	\begin{equation}\label{twt0*}
%				 	\begin{array}{l}
%				 	\displaystyle
%				 	\widetilde{\bf w}_{t}(\cdot,0)=\frac{1}{\rho_{0}}\big(G_{2}^0-(-\mu \Delta -(\mu+\mu{'})\nabla\mathrm{div})({\bf u}_{0}-z\eta_{1}\vec{e_{2}})\big)\,\,\mbox{in}\,\,\Omega,
%				 	\end{array}
%				 	\end{equation}
				 	\begin{equation}\label{1etad0*}
				 	\begin{array}{l}
				 	1+\widetilde{\eta}(x,t)\geqslant\delta_{0}>0\,\,\mathrm{on}\,\,\Sigma^{s}_{T},
				 	\end{array}
				 	\end{equation}
				 	\begin{equation}\label{tsgmM*}
				 	\begin{array}{l}
				 	\displaystyle 0<\frac{m}{2}\leqslant\widetilde{\sigma}+\overline{\rho}\leqslant 2M\,\,\mathrm{in}\,\, Q_{T},
				 	\end{array}
				 	\end{equation}
				 	where $m$ and $M$ were fixed in \eqref{cor0}.\\
				 Then $G_{1}(\widetilde{\sigma},\widetilde{\bf w},\widetilde{\eta}),$ $G_{2}(\widetilde{\sigma},\widetilde{\bf w},\widetilde{\eta})$ and $G_{3}(\widetilde{\sigma},\widetilde{\bf w},\widetilde{\eta})$ satisfy the following
				 \begin{equation}\label{G123}
				 \begin{array}{ll}
				 & G_{1}(\widetilde{\sigma},\widetilde{\bf w},\widetilde{\eta})\in L^{1}(0,T;H^{2}(\Omega))\cap L^{\infty}(0,T;H^{1}(\Omega)),\\
				 & G_{2}(\widetilde{\sigma},\widetilde{\bf w},\widetilde{\eta})\in L^{2}(0,T;{\bf H}^{1}(\Omega))\cap H^{1}(0,T;{\bf L}^{2}(\Omega))\cap L^{\infty}(0,T;{\bf L}^{2}(\Omega)),\\
				 & G_{3}(\widetilde{\sigma},\widetilde{\bf w},\widetilde{\eta})\in L^{\infty}(0,T;H^{1/2}(\Gamma_{s}))\cap H^{1}(0,T;L^{2}(\Gamma_{s})),\\
				 & \widetilde{W}(\widetilde{\bf w},\widetilde{\eta})\in L^{1}(0,T;{\bf H}^{3}(\Omega))\cap L^{\infty}(0,T;{\bf H}^{2}(\Omega)), \\
				& G_{2}(\widetilde{\sigma},\widetilde{\bf w},\widetilde{\eta})\mid_{t=0} = G_2^0 \text{ and } G_{3}(\widetilde{\sigma},\widetilde{\bf w},\widetilde{\eta})\mid_{t=0} = G_3^0.
				 \end{array}
				 \end{equation}
				 \end{lem}
				 \begin{proof}
				 	The detailed computations to verify \eqref{G123} follows from Lemma \ref{eog1} (for estimates of $G_{1}$), Lemma \ref{eog2} (for estimates of $G_{2}$), Lemma \ref{eog3} (for estimates of $G_{3}$) and Lemma \ref{eoWs} (for estimates of $\widetilde{W}$) in the Section \ref{estimates}.
				 \end{proof}
				 Observe that the condition \eqref{tweq0*} implies that $\widetilde{W}(\widetilde{\bf w},\widetilde{\eta})\cdot {\bf n}=0$ (where $\widetilde{W}$ is as defined in \eqref{dotW}) on $\Sigma_{T}.$ Hence in view of Lemma \ref{wellpos}, for all $(\widetilde{\sigma},\widetilde{\bf w},\widetilde{\eta})$ satisfying the conditions \eqref{tsweY}-\eqref{tweq0*}-\eqref{tswe0*}-\eqref{etatt0*}-\eqref{1etad0*}
				 -\eqref{tsgmM*}, the system \eqref{3.2} admits a unique solution as a consequence of Theorem \ref{t21}, Theorem \ref{t22} and Theorem \ref{t23} in the space $Z^{T}_{1}\times Y^{T}_{2}\times Z^{T}_{3},$ where $Y^{T}_{2}$ is defined in \eqref{dofYi}, $Z^{T}_{1}$ and $Z^{T}_{3}$ are defined as follows
				 \begin{equation}\label{Z23}
				 \begin{array}{ll}
				 & Z^{T}_{1}=\{{\rho}\in C^{0}([0,T];H^{2}(\Omega))\suchthat \rho_{t}\in L^{\infty}(0,T; H^{1}(\Omega))\},\\
				 & Z^{T}_{3}=\{\eta\in L^{\infty}(0,T; H^{9/2}(\Gamma_{s})),\,\eta(x,0)=0\suchthat \eta_{t}\in L^{2}(0,T;{ H}^{4}(\Gamma_{s}))\cap C^{0}([0,T];{ H}^{3}(\Gamma_{s})), \\[1.mm]
				 &\qquad\qquad\qquad\qquad\qquad\eta_{tt}\in L^{2}(0,T;{ H}^{2}(\Gamma_{s}))\cap C^{0}([0,T];{ H}^{1}(\Gamma_{s})),
				 \eta_{ttt}\in L^{2}(0,T;{ L}^{2}(\Gamma_{s}))\}.
				 \end{array}
				 \end{equation}
				 Observe that the only difference between $Y^{T}_{1}$ (defined in \eqref{dofYi}) and $Z^{T}_{1}$ is that the elements of $Y^{T}_{1}$ belongs to $C^{1}([0,T];H^{1}(\Omega))$ while the elements of $Z^{T}_{1}$ are in $W^{1,\infty}(0,T;H^{1}(\Omega)).$ Also one observes that the elements of $Y^{T}_{3}$ (defined in \eqref{dofYi}) are in $C^{0}([0,T];H^{9/2}(\Gamma_{s}))$ while $Z^{T}_{3}$ is only a subset of $L^{\infty}(0,T;H^{9/2}(\Gamma_{s})).$\\
				  Before defining a suitable fixed point map (in order to solve the non-linear problem \eqref{chdb}), we will introduce a convex set $\mathscr{C}_{T}$ (where we will show the existence of a fixed point). The set $\mathscr{C}_{T}$ will be defined as a subset of $L^{2}(0,T;L^{2}(\Omega))\times L^{2}(0,T;{\bf L}^{2}(\Omega))\times L^{2}(0,T;L^{2}(\Gamma_{s}))$ such that the elements of $\mathscr{C}_{T}$ satisfy some norm bounds and some conditions at initial time $t=0.$\\
				 Let us make precise the assumptions which will be used to define the set $\mathscr{C}_{T}.$\\ 
				 $Regularity\,\,assumptions\,\,and\,\,norm\,\,bounds\,\,of\,\, (\widetilde{\sigma},\widetilde{\bf w},\widetilde{\eta})$:
				\begin{subequations}\label{normbnd}
				 	\begin{equation}\label{tsB12}
				 	\begin{array}{l}
				 	\|\widetilde{\sigma}\|_{L^{\infty}(0,T;H^{2}(\Omega))}\leqslant B_{1},\quad\|\widetilde\sigma_{t}\|_{L^{\infty}(0,T; H^{1}(\Omega))}\leqslant B_{2},
				 	\end{array}
				 	\end{equation}
				 	\begin{equation}\label{twB3}
				 	\begin{split}
				 	\|\widetilde{\bf w}\|_{L^{\infty}(0,T;{\bf H}^{2}(\Omega))}+\|\widetilde{\bf w}\|_{{L^{2}}(0,T;{\bf H}^{3}(\Omega))}
				 	+\|\widetilde{\bf w}_{t}\|_{L^{\infty}(0,T;{\bf H}^{1}(\Omega))}
				 	&+\|\widetilde{\bf w}_{t}\|_{L^{2}(0,T;{\bf H}^{2}(\Omega))}\\[1.mm]
				 	&+\|\widetilde{\bf w}_{tt}\|_{L^{2}(0,T;{\bf L}^{2}(\Omega))}\leqslant B_{3},
				 	\end{split}
				 	\end{equation}
				 	\begin{equation}\label{teB4}
				 	\begin{split}
				 	\|\widetilde\eta\|_{L^{\infty}(0,T;H^{9/2}(\Gamma_{s}))}+	\|\widetilde\eta_{t}\|_{L^{\infty}(0,T;H^{3}(\Gamma_{s}))}&+	\|\widetilde\eta_{t}\|_{L^{2}(0,T;H^{4}(\Gamma_{s}))}+	\|\widetilde\eta_{tt}\|_{L^{\infty}(0,T; H^{1}(\Gamma_{s}))}\\[1.mm]
				 	&+	\|\widetilde\eta_{tt}\|_{L^{2}(0,T;H^{2}(\Gamma_{s}))}
				 	+\|\widetilde\eta_{ttt}\|_{L^{2}(0,T;L^{2}(\Gamma_{s}))}\leqslant B_{4},
				 	\end{split}
				 	\end{equation} 
				 	\vspace{1.mm}\\
				 	\begin{equation}\label{1etad0}
				 	\begin{split}
				 	1+\widetilde{\eta}(x,t)\geqslant\delta_{0}>0\,\,\mathrm{on}\,\,\Sigma^{s}_{T},
				 	\end{split}
				 	\end{equation}
				 	\begin{equation}\label{tsgmM}
				 	\begin{split}
				 	0<\frac{m}{2}\leqslant\widetilde{\sigma}+\overline{\rho}\leqslant 2M\,\,\mathrm{in}\,\, Q_{T},
				 	\end{split}
				 	\end{equation}
				 \end{subequations} 
				 where $B_{i}$'s ($1\leqslant i\leqslant 4$) are positive constants and will be chosen in the sequel. The norm bound \eqref{twB3} implicitly asserts (by interpolation) that $\widetilde{\bf w}$ is in $C^{0}([0,T];{\bf H}^{5/2}(\Omega)).$\\
				 $Assumptions\,\,on\,\,initial\,\,and\,\,boundary\,\,conditions$:\\
				 \begin{subequations}\label{inbndc}
				 	\begin{equation}\label{tweq0}
				 	\begin{split}
				 	\widetilde{\bf w}=0\,\,\mbox{on}\,\,\Sigma_{T},
				 	\end{split}
				 	\end{equation}
				 	\begin{equation}\label{tswe0}
				 	\begin{split}
				 	(\widetilde{\sigma}(\cdot,0),\widetilde{\bf w}(\cdot,0),\widetilde{\eta}(\cdot,0),\widetilde{\eta}_{t}(\cdot,0))=(\rho_{0}-\overline{\rho},{\bf u}_{0}-z\eta_{1}\vec{e}_{2},0,\eta_{1})\,\,\mbox{in}\,\,\Omega,
				 	\end{split}
				 	\end{equation}
				 	\begin{equation}\label{etatt0}
				 	\begin{split}
				 	\widetilde{\eta}_{tt}(\cdot,0)=\delta\eta_{1,xx}-(\mu+2\mu')(u_{0})_{2,z}+P(\rho_{0})\,\,\mbox{in}\,\,\Omega,
				 	\end{split}
				 	\end{equation}
				 	\begin{equation}\label{twt0}
				 	\begin{split}
				 	\widetilde{\bf w}_{t}(\cdot,0)=\frac{1}{\rho_{0}}\big(G_{2}^0-(-\mu \Delta -(\mu+\mu{'})\nabla\mathrm{div})({\bf u}_{0}-z\eta_{1}\vec{e_{2}})\big)\,\,\mbox{in}\,\,\Omega.
				 	\end{split}
				 	\end{equation}
				 \end{subequations}
				 For $T<\overline{T},$ let us define the following set
				 \begin{equation}\label{defRT}
				 \begin{split}
				 \mathscr{C}_{T}(B_{1},B_{2},B_{3},B_{4})= \{(\widetilde{\sigma},\widetilde{\bf w},\widetilde{\eta})&\in L^{2}(0,T;L^{2}(\Omega))\times L^{2}(0,T;{\bf L}^{2}(\Omega))\times L^{2}(0,T;L^{2}(\Gamma_{s}))\suchthat\\
				 & \mbox{the relations}\, \eqref{normbnd}-\eqref{inbndc}\,\,\mbox{are true}\}.
				 \end{split}
				 \end{equation} 
				 Now for $(\widetilde{\sigma},\widetilde{\bf w},\widetilde{\eta})\in\mathscr{C}_{T}(B_{1},B_{2},B_{3},B_{4}),$ let $(\sigma,{\bf w},\eta)\in Z_{1}^{T}\times Y_{2}^{T}\times Z_{3}^{T}$ (recall the definition of $Y^{T}_{2},$ from \eqref{dofYi} and $Z^{T}_{1},$ $Z^{T}_{3}$ are defined in \eqref{Z23}) be the solution of the problem \eqref{3.2} corresponding to $(\widetilde{\sigma},\widetilde{\bf w},\widetilde{\eta}).$ This defines the map
				 \begin{equation}\label{welposL}
				 \begin{matrix}
				 L: \mathscr{C}_{T}(B_{1},B_{2},B_{3},B_{4})&\longrightarrow& Z_{1}^{T}\times Y_{2}^{T}\times Z_{3}^{T}\\
				  (\widetilde{\sigma},\widetilde{\bf w},\widetilde{\eta})&\mapsto& ({\sigma},{\bf w},{\eta}).
				  \end{matrix}
				 \end{equation}
				 Now observe that if the map $L$ admits a fixed point $(\sigma_{f},{\bf w}_{f},\eta_{f})$ on the set $\mathscr{C}_{T}(B_{1},B_{2},B_{3},B_{4}),$ then the triplet $(\sigma_{f},{\bf w}_{f},\eta_{f})$ is a solution to the system \eqref{chdb}. 
				 Thus, our goal from now is to prove the existence of a fixed point to the map $L.$ In that direction we first show that for suitable parameters $B_{i}$ ($1\leqslant i\leqslant 4$) and $T,$ the set $\mathscr{C}_{T}(B_{1},B_{2},B_{3},B_{4})$ is non-empty.
				 \begin{lem}\label{nonempty}
				 	There exists a constant $B_{0}^{*}>0$ such that for all $B_{i}\geqslant B^{*}_{0}$ ($1\leqslant i \leqslant 4$) there exists $T^{*}_{0}(B_{1},B_{2},B_{3},B_{4} )\in (0, \min\{1,\overline{T}\})$ such that for all $0<T\leqslant T^{*}_{0}(B_{1},B_{2},B_{3},B_{4})$ the set $\mathscr{C}_{T}(B_{1},B_{2},B_{3},B_{4})$ is non empty.
				 \end{lem}
				 \begin{proof}
				 	%In the following we will fix a constant $B^{*}_{0}$ such that for all $B_{i}\geqslant B^{*}_{0},$ ($1\leqslant i\leqslant 4$) there exists $0<T^{*}_{0}(B_{1},B_{2},B_{3},B_{4})(<\min\{1,\overline{T}\})$ such that for all $0<T\leqslant T^{*}_{0}(B_{1},B_{2},B_{3},B_{4})$ the set $\mathscr{C}_{T}(B_{1},B_{2},B_{3},B_{4})$ is non empty. 
				 	The choice of the constant $B^{*}_{0}$ will be done based on the calculations performed in the following steps.\\
				 	$Step\,1.$ In this step we will prove the existence of a function ${\bf w}^{*}$ which satisfies the norm bound \eqref{twB3} and the condition \eqref{twt0} at time $t=0.$ We begin by recalling that $(\rho_{0},{\bf u}_{0},\eta_{1})$ satisfies \eqref{cit} and hence one observes that $({\bf u}_{0}-z\eta_{1}\vec{e}_{2})\in {\bf H}^{3}(\Omega).$ As $G_{2}^0$ is given by the expression \eqref{iG0-bis}, using \eqref{regcom} one has $G_{2}^0\in H^{1}(\Omega).$ We can thus find a lifting ${\bf h}\in L^{2}(\mathbb{R}^{+};{\bf H}^{1}(\Omega))$ and ${\bf h}_{t}\in L^{2}(\mathbb{R}^{+};{\bf L}^{2}(\Omega))$ (see e.g. \cite[Theorem 3.2, p. 21]{liomag1}) such that ${\bf h}(0)=G_{2}^0$ in $\Omega$. (In fact, we only need $G_2^0 \in H^{1/2}(\Omega)$ in this step.)
	\\				
%					
%					We write $${\bf h}(0)=G_{2}\mid_{t=0}\in {\bf H}^{1}(\Omega).$$ We use in particular that ${\bf h}(0)\in {\bf H}^{1/2}(\Omega)$ to obtain a suitable lifting. We consider a lifting ${\bf h}$ of ${\bf h}(0)$ in $\mathbb{R}^{+}$ such that
%				 ${\bf h}\in L^{2}(\mathbb{R}^{+};{\bf H}^{1}(\Omega))$ and ${\bf h}_{t}\in L^{2}(\mathbb{R}^{+};{\bf L}^{2}(\Omega))$ (such a lifting can be obtained using \cite[Theorem 3.2, p. 21]{liomag1}).\\ 
				 	Let ${\bf w}^{*}$ be the solution of the following system
				 	\begin{equation}\label{eqv*}
				 	\left\{ \begin{array}{ll}
				 	\rho_{0}{\bf w}^{*}_{t}-\mu\Delta{\bf w}^{*}-(\mu+\mu')\nabla\mbox{div}{\bf w}^{*}={\bf h}& \quad\mbox{in}\quad Q_{\infty},\\[1.mm]
				 	{\bf w}^{*}=0& \quad\mbox{on}\quad \Sigma_{\infty},\\[1.mm]
				 	{\bf w}^{*}(0)={\bf w}_{0}=({\bf u}_{0}-z\eta_{1}\vec{e_{2}})&\quad\mbox{in}\quad\Omega.
				 	\end{array}\right.
				 	\end{equation}
				 	In view of \eqref{asm2} one can uniquely solve \eqref{eqv*} such that the function ${\bf w}^{*}$ satisfies the following estimate
				 	\begin{equation}\label{nfv}
				 	\begin{split}
				 	&\|{\bf w}^{*}\|_{L^{\infty}(0,\infty;{\bf H}^{2}(\Omega))}+\|{\bf w}^{*}\|_{{L^{2}}(0,\infty;{\bf H}^{3}(\Omega))}
				 	+\|{\bf w}^{*}_{t}\|_{L^{\infty}(0,\infty;{\bf H}^{1}(\Omega))}\\
				 	&+\|{\bf w}^{*}_{t}\|_{L^{2}(0,\infty;{\bf H}^{2}(\Omega))}
				 	+\|{\bf w}^{*}_{tt}\|_{L^{2}(0,\infty;{\bf L}^{2}(\Omega))}\\
				 	&\leqslant c(\|{\bf h}\|_{L^{2}(0,\infty;H^{1}(\Omega))}+\|{\bf h}_{t}\|_{L^{2}(0,\infty;L^{2}(\Omega))}+\|G_{2}\mid_{t=0}\|_{H^{1}(\Omega)}+\|{\bf u}_{0}-z\eta_{1}\vec{e_{2}}\|_{H^{3}(\Omega)})\\
				 	&\leqslant c_{5}(\|G_{2}\mid_{t=0}\|_{H^{1}(\Omega)}+\|{\bf u}_{0}-z\eta_{1}\vec{e_{2}}\|_{H^{3}(\Omega)}).
				 	\end{split}
				 	\end{equation}
				 	Using \eqref{eqv*} one also observes the following
				 	\begin{equation}\label{vt*}
				 	\begin{array}{l}
				 	\displaystyle
				 	{\bf w}^{*}_{t}(\cdot,0)=\frac{1}{\rho_{0}}\big(G_{2}^0-(-\mu \Delta -(\mu+\mu{'})\nabla\mathrm{div})({\bf u}_{0}-z\eta_{1}\vec{e_{2}})\big).
				 	\end{array}
				 	\end{equation}
				 	In view of \eqref{nfv} and \eqref{vt*} one observes that ${\bf w}^{*}$ satisfies \eqref{twB3} and \eqref{twt0} respectively.\\
				 	$Step \,2.$\,  In this step we will prove the existence of a function ${\eta}^{*}$ which satisfies the norm bound \eqref{teB4}, \eqref{1etad0} and the condition \eqref{etatt0} at time $t=0.$ In that direction first recall that $G_{3}^0\in H^{1}(\Gamma_{s}).$ We use in particular the regularity $G_{3}^0\in H^{1/2}(\Gamma_{s})$ to obtain a lifting $h_1$ of $G_{3}^0$ 
					%Consider a lifting $h_{1}$ of $G_{3}\mid_{t=0}$ in $\mathbb{R}^{+}$ 
					such that $h_{1}\in L^{2}(\mathbb{R}^{+};H^{1}(\Gamma_{s}))\cap H^{1}(\mathbb{R}^{+};L^{2}(\Gamma_{s}))\cap L^{\infty}(\mathbb{R}^{+};H^{1/2}(\Gamma_{s}))$ and $h_1(0) = G_3^0$ in $\Gamma_s$. Let $\eta^{*}$ be the solution of equation \eqref{2.3.1} with $G_{3}$ replaced by $h_{1}$. From Theorem \ref{t23} and inequality \eqref{2.3.5} one obtains
				 	\begin{equation}\label{nfe}
				 	\begin{split}
				 	&\|\eta^{*}\|_ {L^{\infty}(0,T;H^{9/2}(\Gamma_{s}))} +\|\eta_{t}^{*}\|_{L^{2}(0,T;H^{4}(\Gamma_{s}))}+\|\eta_{t}^{*}\|_{L^{\infty}(0,T;H^{3}(\Gamma_{s}))}+\|\eta_{tt}^{*}\|_{ L^{2}(0,T;H^{2}(\Gamma_{s}))}\\
				 	&+\|\eta_{tt}^{*}\|_{L^{\infty}([0,T]; H^{1}(\Gamma_{s}))}
				 	+\|\eta_{ttt}^{*}\|_{L^{2}(0,T;L^{2}(\Gamma_{s}))}\\
				 	&\leqslant c(\|(h_{1})_{t}\|_{L^{2}(0,\infty;L^{2}(\Gamma_{s}))}+\|h_{1}\|_{L^{\infty}(0,\infty;H^{1/2}(\Gamma_{s}))}+\|G_{3}^0\|_{H^{1}(\Gamma_{s})}+\|\eta_{1}\|_{H^{3}(\Gamma_{s})})\\
				 	& \leqslant c_{4}(\|G_{3}^0\|_{H^{1}(\Gamma_{s})}+\|\eta_{1}\|_{H^{3}(\Gamma_{s})})
				 	\end{split}
				 	\end{equation}  
				 	where the constant $c_{4}$ is independent of $T.$ One further uses \eqref{G30-bis} to check that
				 	\begin{equation}\label{etatt*}
				 	\begin{array}{l}
				 	\eta^{*}_{tt}(\cdot,0)=\delta\eta_{1,xx}-(\mu+2\mu')(u_{0})_{2,z}+P(\rho_{0}). 
				 	\end{array}
				 	\end{equation}
				 	In view of \eqref{nfe} and \eqref{etatt*} we get that $\eta^{*}$ satisfies \eqref{teB4} and \eqref{etatt0}.\\
				 	Since $\eta^{*}(.,0)=0,$ we observe the following by interpolation
				 	\begin{equation}\label{3.1}
				 	\begin{split}
				 	\|{\eta}^{*}\|_{C^{0}(\overline{\Sigma}_{T}^{s})}&\leqslant c\|{\eta}^{*}\|^{1/3}_{L^{\infty}(0,T; H^{1}(\Gamma_{s}))}\|{\eta}^{*}\|^{2/3}_{L^{\infty}(0,T;H^{2}(\Gamma_{s}))}\\[1.mm]
				 	& \leqslant cT^{1/3}\big(\|\eta_{t}^{*}\|_{L^{\infty}(0,T; H^{1}(\Gamma_{s}))}\big)^{1/3}\cdot \big(\|{\eta}^{*}\|_{{L^{\infty}(0,T;H^{2}(\Gamma_{s}))}}\big)^{2/3}.
				 	\end{split}
				 	\end{equation}	
				 	At this point we set
				 		\begin{equation}\label{coBi}
				 		\begin{split}
				 		B^{*}_{0}=\mbox{max}\{c_{5}(\|G_{2}^0\|_{H^{1}(\Omega)}+\|{\bf u}_{0}-z\eta_{1}\vec{e_{2}}\|_{H^{3}(\Omega)}),&\,\, \|\rho_{0}-\overline\rho\|_{H^{2}(\Omega)},\\
				 		& c_{4}(\|G_{3}^0\|_{H^{1}(\Gamma_{s})}+\|\eta_{1}\|_{H^{3}(\Gamma_{s})})\}
				 		\end{split}
				 		\end{equation}
				 		and  for all $1\leqslant i\leqslant 4,$ $B_{i}\geqslant B^{*}_{0}.$\\ 
				 		Hence in view of \eqref{3.1}, there exists $T^{*}_{0}(B_{1},B_{2},B_{3},B_{4})\in (0, \min\{1,\overline{T}\})$ such that for all $0<T\leqslant T^{*}_{0}(B_{1},B_{2},B_{3},B_{4})$ we verify that
				 	$$1+{\eta}^{*}\geqslant\delta_{0}>0\quad\mbox{on}\quad \Sigma^{s}_{T},$$
				 	%Similarly for small enough $T$ we can have the following
				 	%$$0<\frac{m}{2}\leqslant {\sigma}^{*}+\overline{\rho}\leqslant 2M.$$
				 	$i.e$ $\eta^{*}$ satisfies \eqref{1etad0}.\\
				 	$Step\, 3.$\, One easily checks that $\sigma^{*}={\rho}_{0}-\overline{\rho}$ verifies \eqref{tsB12} and \eqref{tsgmM}.\\
				 	We further observe that $(\sigma^{*},{\bf w}^{*},\eta^{*})$ satisfies \eqref{tweq0} and \eqref{tswe0} automatically by construction.\\ 
				 	 So we have shown that if we choose $B^{*}_{0}$ (and hence $B_{i}\geqslant B^{*}_{0},$ for all $1\leqslant i\leqslant 4$) as in \eqref{coBi} and $0<T\leqslant T^{*}_{0}(B_{1},B_{2},B_{3},B_{4})$ then $(\sigma^{*}={\rho}_{0}-\overline{\rho},{\bf w}^{*},\eta^{*})\in \mathscr{C}_{T}(B_{1},B_{2},B_{3},B_{4}),$ $i.e.$
				 	$$\mathscr{C}_{T}(B_{1},B_{2},B_{3},B_{4})\neq\varnothing.$$
				 \end{proof}
				 \begin{remark}
				 	Observe from the proof of Lemma \ref{nonempty}, the constant $T_{0}^{*}(B_{1},B_{2},B_{3},B_{4})$ depends on $\delta_{0}\in (0,1).$ Since $\delta_{0}$ is fixed (see \eqref{fixd0}) we do not write explicitly the dependence of $T_{0}^{*}(B_{1},B_{2},B_{3},B_{4})$ on $\delta_{0}.$ 
				 \end{remark}
				\subsection{For small enough $T,$ $L$ maps $\mathscr{C}_{T}(B_{1},B_{2},B_{3},B_{4})$ into itself}\label{Enht}
			To prove that the map $L$ admits a fixed point we first show that for $T$ small enough and a suitable choice of parameters $(B_1, B_2, B_3, B_4)$, the set $\mathscr{C}_{T}(B_{1},B_{2},B_{3},B_{4})$ is mapped into itself by $L.$\\
			 Provided $(\widetilde{\sigma},\widetilde{\bf w},\widetilde{\eta})\in \mathscr{C}_{T}(B_{1},B_{2},B_{3},B_{4}),$	we have to estimate the terms $G_{1}(\widetilde{\sigma},\widetilde{\bf w},\widetilde{\eta}),$ $G_{2}(\widetilde{\sigma},\widetilde{\bf w},\widetilde{\eta}),$ $G_{3}(\widetilde{\sigma},\widetilde{\bf w},\widetilde{\eta})$ and $\widetilde{W}(\widetilde{\bf w},\widetilde{\eta})$ (recall the definition of $G_{1},$ $G_{2},$ $G_{3}$ and $\widetilde{W}$ from \eqref{1.22} and \eqref{dotW} respectively). For this purpose we will require some results which we collect in the following section.
			\subsubsection{Useful lemmas}
			The following lemma concerning the Sobolev regularity of the product of two functions is standard in the literature.
				\begin{lem}\label{lfpss}
					Consider a bounded domain $\Omega_{0}$ in $\mathbb{R}^{d}$ (for $d=1,2$).
					Let $r>\frac{d}{2},$ $0\leqslant s\leqslant r.$ If $v\in H^{r}(\Omega_{0})$ and $w\in H^{s}(\Omega_{0})$ then $vw\in H^{s}(\Omega_{0})$ with
					$$\|vw\|_{H^{s}(\Omega_{0})}\leqslant K(\Omega_{0})\|v\|_{H^{r}(\Omega_{0})}\|w\|_{H^{s}(\Omega_{0})}.$$ 
					Similar estimates hold when $v$ and $w$ are vector valued functions i.e for ${\bf v}\in {\bf H}^{r}(\Omega_{0})$ and ${\bf w}\in {\bf H}^{s}(\Omega_{0}).$
					%(ii) If $v\in  H^{1}(\Omega_{0})$ and $w\in  H^{1}(\Omega_{0})$ then $vw\in L^{2}(\Omega_{0})$ with $$\|vw\|_{L^{2}(\Omega_{0})}\leqslant K(\Omega_{0})\|v\|_{ H^{1}(\Omega_{0})}\|w\|_{ H^{1}(\Omega_{0})}.$$
					%Similar estimate holds when $v$ and $w$ are vector valued functions i.e for ${\bf v}\in {\bf H}^{1}(\Omega_{0})$ and ${\bf w}\in {\bf H}^{1}(\Omega_{0}).$
				\end{lem}
				\begin{lem}\label{leoin}
					Let $T<\overline{T}$ (recall that we have fixed $\overline{T}$ in \eqref{Tbar}). We assume that ${\bf f}\in {\bf H}^{2,1}_{\Sigma_{T}}(Q_{T}).$ As usual we use the notation ${\bf f}_{z}$ to denote the directional derivative $\partial_{z}{\bf f}$ of ${\bf f}$ with respect to $z.$ Also suppose that $\Gamma_{s}$ is a smooth subset of $\Gamma.$ Then the trace ${\bf f}_{z}\mid_{\Sigma_{T}}$ on $\Gamma_{s}$ (i.e the normal derivative of ${\bf f}$ on $\Gamma_{s}$) belongs to $H^{1/6}(0,T;{\bf L}^{2}(\Gamma_{s})).$ In particular there exists a constant $K>0$ such that for all ${\bf f}\in {\bf H}^{2,1}_{\Sigma_{T}}(Q_{T})$ we have the following
					\begin{equation}\label{ibin}
						\begin{array}{l}
							\|{\bf f}_{z}\mid_{\Sigma_{T}}\|_{L^{2}(0,T;{\bf L}^{2}(\Gamma_{s}))}\leqslant T^{1/6}K(\|{\bf f}(0)\|_{{\bf H}^{1}_{0}(\Omega)}+\|{\bf f}\|_{{\bf H}^{2,1}_{\Sigma_{T}}(Q_{T})}),
						\end{array}
					\end{equation}
					where ${\bf f}(0)$ denotes the function ${\bf f}$ at time $t=0.$ We specify that in our case the space ${\bf H}^{2,1}_{\Sigma_{T}}(Q_{T})$ is endowed with the following norm
					$$\|{\bf f}\|_{{\bf H}^{2,1}_{\Sigma_{T}}(Q_{T})}=\|{\bf f}\|_{L^{2}(0,T;{\bf H}^{2}(\Omega))}+\|{\bf f}_{t}\|_{L^{2}(0,T;{\bf L}^{2}(\Omega))}.$$
				\end{lem}
				\begin{remark}
					The appearance of ${\bf f}(0)$ in the inequality \eqref{ibin} might seem redundant since for all, ${\bf f}\in {\bf H}^{2,1}_{\Sigma_{T}}(Q_{T})$
					$$\|{\bf f}(0)\|_{{\bf H}^{1}_{0}(\Omega)}\leqslant K_{T}\|{\bf f}\|_{{\bf H}^{2,1}_{\Sigma_{T}}(Q_{T})}.$$ 
					But the constant $K_{T}$ there may depend on $T$ while the constant $K$ in \eqref{ibin} is independent of $T.$ This is the reason why we prefer working with \eqref{ibin}.
				\end{remark}
				\begin{proof}[Proof of Lemma \ref{leoin}]
					We have to estimate $\|{\bf f}_{z}\mid_{\Sigma_{T}}\|_{L^{2}(0,T;L^{2}(\Gamma_
						{s}))}.$ Using H\"{o}lder's inequality we get the following
					\begin{equation}\label{il32}
					\begin{array}{l}
					\displaystyle
					\left(\int\limits_{0}^{T}\|{\bf f}_{z}\mid_{\Sigma_{T}}\|^{2}_{L^{2}(\Gamma_{s})}\right)^{1/2}\leqslant \left(\int\limits_{0}^{{T}}\|{\bf f}_{z}\mid_{\Sigma_{T}}\|^{3}_{L^{2}(\Gamma_{s})}\right)^{1/3}T^{1/6}\leqslant K(\Omega)\left(\int\limits_{0}^{{T}}\|{\bf f}\|^{3}_{{\bf H}^{5/3}{(\Omega)}}\right)^{1/3}T^{1/6}.
					\end{array}
					\end{equation}
					To prove \eqref{ibin}, in view of \eqref{il32} it is enough to show the following inequality
						\begin{equation}\label{triain3}
						\begin{array}{l}
						\|{\bf f}\|_{L^{3}(0,T;{\bf H}^{5/3}(\Omega))}\leqslant K(\Omega,\overline{T})(\|{\bf f}\|_{{\bf H}^{2,1}_{\Sigma_{T}}(Q_{T})}+\|{\bf f}(0)\|_{{\bf H}^{1}_{0}(\Omega)}).
						\end{array}
						\end{equation}
					In order to prove \eqref{triain3}, first let us consider the solution ${\bf{f}}^{*}$ of
					\begin{equation}\label{ellstf}
						\left\{ \begin{array}{lll}
							{\bf f}^{*}_{t}-\Delta{\bf f}^{*}=0&\quad\mbox{in}\quad& Q_{T},\\[1.mm]
							{\bf f}^{*}=0&\quad\mbox{on}\quad &\Sigma_{T},\\[1.mm]
							{\bf f}^{*}(.,0)={\bf f}(0)&\quad\mbox{in}\quad&\Omega.
						\end{array}\right.
					\end{equation}
					As ${\bf f}(0)\in {\bf H}^{1}_{0}(\Omega),$ ${\bf f}^{*}\in {\bf H}^{2,1}_{\Sigma_{T}}(Q_{T}).$ It is also well known that there exists a constant $K(\Omega)$ such that ${\bf f}^{*}$ satisfies the following inequalities 
					%Now to get time independent norm bounds over ${\bf f}^{*}$ we multiply \eqref{ellstf} by $({\bf f}^{*}_{t}-\Delta{\bf f}^{*})$ and integrate over $\Omega.$ We get
					%\begin{equation}\nonumber
						%\begin{array}{l}
						%\displaystyle
						%	\int\limits_{\Omega}|{\bf f}^{*}_{t}|^{2}-2\int\limits_{\Omega}{\bf f}^{*}_{t}\cdot\Delta{\bf f}^{*}+\int\limits_{\Omega}|\Delta{\bf f}^{*}|^{2}=0.
					%	\end{array}
					%\end{equation}
					%Integrating by parts the second term in the left hand side we obtain
					%\begin{equation}\nonumber
					%	\begin{array}{l}
						%\displaystyle
						%	\int\limits_{\Omega}|{\bf f}^{*}_{t}|^{2}+\frac{d}{dt}\int\limits_{\Omega}|\nabla{\bf f}^{*}|^{2}+\int\limits_{\Omega}|\Delta{\bf f}^{*}|^{2}=0.
					%	\end{array}
				%	\end{equation}
					%Now just by integrating both sides of the inequality above in the time interval $(0,T)$ one can at once conclude that there exists a constant $K(\Omega)$ such that
					\begin{equation}\label{inell}
						\begin{array}{ll}
							(i)&\|{\bf f}^{*}\|_{{\bf H}^{2,1}_{\Sigma_{T}}(Q_{T})}\leqslant K(\Omega)\|{\bf f}(0)\|_{{\bf H}^{1}_{0}(\Omega)},\\[1.mm]
							(ii)& \|{\bf f}^{*}\|_{L^{\infty}(0,T;{\bf H}^{1}_{0}(\Omega))}+\|{\bf f}^{*}\|_{L^{2}(0,T;{\bf H}^{2}(\Omega))}\leqslant K(\Omega)\|{\bf f}(0)\|_{{\bf H}^{1}_{0}(\Omega)}.
						\end{array}
					\end{equation}
					Now we will estimate the norm of ${\bf f}^{*}$ in $L^{3}(0,T;{\bf H}^{5/3}(\Omega)).$ %So let us extend ${\bf f}^{*}$ by zero in the interval $(T-\overline{T},0),$ for some $\overline{T}>T.$ We denote this extended function also by ${\bf f}^{*}.$ We can check that this extended function, $${\bf f}^{*}\in L^{\infty}(\overline{T}-T,T;{\bf H}^{1}_{0})\cap L^{2}(\overline{T}-T,T;{\bf H}^{2}_{0,\#}(\Omega)),$$
					%and the following estimate holds
					%\begin{equation}\label{esoinn}
					%\begin{array}{l}
					%\|{\bf f}^{*}\|^{2}_{L^{\infty}(\overline{T}-T,T;{\bf H}^{1}_{0}(\Omega))}+ \|{\bf f}^{*}\|^{2}_{L^{2}(\overline{T}-T,T;{\bf H}^{2}_{0,\#}(\Omega))}=\|{\bf f}^{*}\|^{2}_{L^{\infty}(0,T;{\bf H}^{1}_{0}(\Omega))}+ \|{\bf f}^{*}\|^{2}_{L^{2}(0,T;{\bf H}^{2}_{0,\#}(\Omega))}.
					%	\end{array}
					%\end{equation}
					Using interpolation we have for $a.e$ $t$
					$$\|{\bf f}^{*}(t)\|_{{\bf H}^{5/3}(\Omega)}\leqslant K(\Omega)\|{\bf f}^{*}(t)\|^{2/3}_{{\bf H}^{2}(\Omega)}\|{\bf f}^{*}(t)\|^{1/3}_{{\bf H}^{1}_{0}(\Omega)}.$$
					From the last inequality one obtains the following
					\begin{equation}\label{iafin}
						\begin{array}{l}
						\displaystyle
							\|{\bf f}^{*}\|_{L^{3}(0,T;{\bf H}^{5/3}(\Omega))}=\left(\int\limits_{0}^{T}\|{\bf f}^{*}(t)\|^{3}_{{\bf H}^{5/3}(\Omega)}\right)^{1/3}\leqslant K(\Omega)\|{\bf f}^{*}\|^{1/3}_{L^{\infty}(0,T;{\bf H}^{1}_{0}(\Omega))}\|{\bf f}^{*}\|^{2/3}_{L^{2}(0,T;{\bf H}^{2}(\Omega))}.
						\end{array}
					\end{equation}
					Hence using inequality $(ii)$ of \eqref{inell} in \eqref{iafin} we obtain
					\begin{equation}\label{iafin1}
						\begin{array}{l}
							\|{\bf f}^{*}\|_{L^{3}(0,T;{\bf H}^{5/3}(\Omega))}\leqslant K(\Omega)\|{\bf f}(0)\|_{{\bf H}^{1}_{0}(\Omega)}.
						\end{array}
					\end{equation}
					Now let us observe that $({\bf f}-{\bf f}^{*})(0)=0.$ Extend the function $({\bf f}-{\bf f}^{*})$ by defining it zero in the time interval $(T-\overline{T},0)$ (the extended function is also denoted by $({\bf f}-{\bf f}^{*})$). In what follows we will use the notation 
					\begin{equation}\nonumber
					\begin{array}{l}
					{Q_{T-\overline{T},T}}=\Omega\times(T-\overline{T},T).
					\end{array}
					\end{equation}
					We also introduce the space ${\bf H}^{2,1}_{\Sigma_{T}}({Q_{T-\overline{T},T}})$ which is defined as in \eqref{fntlsp} with $Q_{T}$ replaced by ${Q_{T-\overline{T},T}}.$\\
					One can check that the extended function $({\bf f}-{\bf f}^{*})\in {\bf H}^{2,1}_{\Sigma_{T}}({Q_{T-\overline{T},T}})$ and
					\begin{equation}\label{noeq}
						\begin{array}{l}
							\|({\bf f}-{\bf f}^{*})\|_{{\bf H}^{2,1}_{\Sigma_{T}}({Q_{T-\overline{T},T}})}=\|({\bf f}-{\bf f}^{*})\|_{{\bf H}^{2,1}_{\Sigma_{T}}(Q_{T})}.
						\end{array}
					\end{equation}
					Again due to the embedding ${\bf H}^{2,1}_{\Sigma_{T}}({Q_{T-\overline{T},T}})\hookrightarrow H^{1/6}(T-\overline{T},{T};{\bf H}^{5/3}(\Omega))$ we have the following 
					\begin{equation}\label{il321}
						\begin{array}{l}
							\|{\bf f}-{\bf f}^{*}\|_{{H}^{1/6}(T-\overline{T},T;{\bf H}^{5/3}(\Omega))}\leqslant K(\overline{T},\Omega)\|{\bf f}-{\bf f}^{*}\|_{{\bf H}^{2,1}_{\Sigma_{T}}({Q_{T-\overline{T},T}})}.
						\end{array}
					\end{equation}
					Since $H^{1/6}(T-\overline{T},T)$ is continuously embedded into $L^{3}(T-\overline{T},T)$, hence from \eqref{il321}
					\begin{equation}\label{il32ne}
						\begin{array}{l}
							\|{\bf f}-{\bf f}^{*}\|_{{L}^{3}(T-\overline{T},T;{\bf H}^{5/3}(\Omega))}\leqslant K(\overline{T},\Omega)\|{\bf f}-{\bf f}^{*}\|_{{\bf H}^{2,1}_{\Sigma_{T}}({Q_{T-\overline{T},T}})}.
						\end{array}
					\end{equation}
					Use of triangle inequality furnishes the following
					\begin{equation}\label{triain}
						\begin{array}{l}
							\|{\bf f}\|_{L^{3}(0,T;{\bf H}^{5/3}(\Omega))}\leqslant K(\|{\bf f}-{\bf f}^{*}\|_{L^{3}(0,T;{\bf H}^{5/3}(\Omega))}+\|{\bf f}^{*}\|_{L^{3}(0,T;{\bf H}^{5/3}(\Omega))}).
						\end{array}
					\end{equation}
					Incorporate inequalities \eqref{iafin1} and \eqref{il32ne} in \eqref{triain} in order to obtain
					\begin{equation}\label{triain1}
						\begin{array}{l}
							\|{\bf f}\|_{L^{3}(0,T;{\bf H}^{5/3}(\Omega))}\leqslant K(\Omega,\overline{T})(\|{\bf f}-{\bf f}^{*}\|_{{\bf H}^{2,1}_{\Sigma_{T}}({Q_{T-\overline{T},T}})}+\|{\bf f}(0)\|_{{\bf H}^{1}_{0}(\Omega)}).
						\end{array}
					\end{equation}
					In view of the equality \eqref{noeq} we can obtain the following from \eqref{triain1},
					\begin{equation}\label{triain2}
						\begin{array}{l}
							\|{\bf f}\|_{L^{3}(0,T;{\bf H}^{5/3}(\Omega))}\leqslant K(\Omega,\overline{T})(\|{\bf f}-{\bf f}^{*}\|_{{\bf H}^{2,1}_{\Sigma_{T}}(Q_{T})}+\|{\bf f}(0)\|_{{\bf H}^{1}_{0}(\Omega)}).
						\end{array}
					\end{equation}
					Once again use triangle inequality and \eqref{inell} $(i)$, in order to prove \eqref{triain3}.\\
					Finally use \eqref{triain3} in \eqref{il32} to show \eqref{ibin}. This completes the proof.
				\end{proof}
				%Lemma A.1, \cite{ray1}. Similar results can also be found in Lemma 2.6, \cite{celmad}. \\
				The following lemma is a simple consequence of the fundamental theorem of calculus, whose proof is left to the reader.
				\begin{lem}\label{futcl}
					Fix $i\geqslant 0$ and a domain $\Omega_{0}$ in $\mathbb{R}^{d}$ ($d$ is either $1$ or $2$). Then there exists a constant $K>0$ such that for all $\psi\in H^{1}(0,T;H^{i}(\Omega_{0})),$ the following holds 
					\begin{equation}\label{eoLit}
						\begin{array}{l}
						\displaystyle
							\|\psi\|_{L^{\infty}(0,T;H^{i}(\Omega_{0}))}\leqslant K(\|\psi(0)\|_{H^{i}(\Omega_{0})}+T^{1/2}\|\psi_{t}\|_{L^{2}(0,T;H^{i}(\Omega_{0}))}),
						\end{array}
					\end{equation} 
					where $\psi(0)$ denotes $\psi$ at time $t=0.$ The inequality \eqref{eoLit} is true even for a vector valued function ${\Psi}\in H^{1}(0,T;{\bf H}^{i}(\Omega_{0})).$
				\end{lem}
				\subsubsection{Estimates of $G_{1},$ $G_{2},$ $G_{3}$ and $\widetilde{W}$} \label{estimates}   
				\begin{lem}\label{eog1}
					Let $B^{*}_{0}$ and $T^{*}_{0}$ are as in Lemma \ref{nonempty} and $B_{i}\geqslant B^{*}_{0}$ ($\forall\, 1\leqslant i\leqslant 4$). Then there exist $K_{1}=K_{1}(B_{1},B_{2},B_{3},B_{4})>0$ and $K_{2}>0$ such that for all $0<T\leqslant T^{*}_{0}(B_{1},B_{2},B_{3},B_{4})$ and $(\widetilde{\sigma},\widetilde{\bf w},\widetilde{\eta})\in \mathscr{C}_{T}(B_{1},B_{2},B_{3},B_{4}),$ $G_{1}(\widetilde{\sigma},\widetilde{\bf w},\widetilde{\eta})$ (defined in \eqref{1.22}) satisfies the following estimates  
						\begin{equation}\label{eg1}
						\begin{array}{lll}
						&(i)&\|G_{1}(\widetilde{\sigma},\widetilde{\bf w},\widetilde{\eta})\|_{L^{1}(0,T;H^{2}(\Omega))}\leqslant K_{1}(B_{1},B_{2},B_{3},B_{4})T^{1/2},\\[1.mm]
						&(ii)& \|G_{1}(\widetilde{\sigma},\widetilde{\bf w},\widetilde{\eta})\|_{L^{\infty}(0,T; H^{1}(\Omega))}\leqslant K_{2}\|\rho_{0}\mathrm{div}({\bf u}_{0})\|_{ H^{1}(\Omega)}+K_{1}(B_{1},B_{2},B_{3},B_{4})T^{1/2}.
						\end{array}
						\end{equation}
						\end{lem}
						\begin{remark}
							In \eqref{eg1}, the constant $K_{2}$ does not depend on any of the $B_{i}$ ($1\leqslant i\leqslant 4$). 
						\end{remark}
			    \begin{proof}[Proof of Lemma \ref{eog1}]
				(i) We will first prove \eqref{eg1}$(i).$\\
				\smallskip
				{Estimate of $(\widetilde{\sigma}+\overline{\rho})\mbox{div}(\widetilde{\bf w}+z\widetilde{\eta}_{t}\vec{e}_{2})$ in $L^{1}(0,T;H^{2}(\Omega))$:}
				From  \eqref{tsweY}, we get that $(\widetilde{\sigma}+\overline{\rho})\in L^{\infty}(0,T;H^{2}(\Omega))$ and $(\widetilde{\bf w}+z\widetilde{\eta}_{t}\vec{e}_{2})\in L^{2}(0,T;{\bf H}^{2}(\Omega)).$ Hence we have the following inequality
				\begin{equation}\label{eog11}
				\begin{split}
				&\|(\widetilde{\sigma}+\overline{\rho})\mbox{div}(\widetilde{\bf w}+z\widetilde{\eta}_{t}\vec{e}_{2})\|_{L^{1}(0,T;H^{2}(\Omega))}\\[1.mm]
				&\leqslant K(\|(\widetilde{\sigma}+\overline{\rho})\|_{L^{\infty}(0,T;H^{2}(\Omega))}\|(\widetilde{\bf w}+z\widetilde{\eta}_{t}\vec{e}_{2})\|_{L^{1}(0,T;{\bf H}^{3}(\Omega))})\quad(\mbox{using Lemma}\,\ref{lfpss})\\[1.mm]
				& \leqslant KT^{1/2}(\|(\widetilde{\sigma}+\overline{\rho})\|_{L^{\infty}(0,T;H^{2}(\Omega))}\|(\widetilde{\bf w}+z\widetilde{\eta}_{t}\vec{e}_{2})\|_{L^{2}(0,T;{\bf H}^{3}(\Omega))})\quad(\mbox{using}\,\mbox{H\"{o}lder's inequality})\\[1.mm]
				& \leqslant K(B_{1},B_{3},B_{4})T^{1/2},\quad(\mbox{using }\eqref{tsB12},\eqref{twB3},\eqref{teB4}).
			    \end{split}
				\end{equation}
				{Estimate of $F_{1}(\widetilde{\sigma}+\overline{\rho},\widetilde{w}+z\widetilde{\eta}_{t}\vec{e}_{2},\widetilde{\eta})$} (defined in \eqref{F123}) in $L^{1}(0,T;H^{2}(\Omega))$: First observe that, as $\widetilde{\eta}\in L^{\infty}(0,T;H^{9/2}(\Gamma_{s}))$ and \eqref{1etad0} holds, one can verify the following
			 \begin{equation}\label{esp1be}
			 \begin{array}{l}
			 \displaystyle
			 \frac{1}{(1+\widetilde{\eta})}\in L^{\infty}(0,T;{H}^{9/2}(\Gamma_{s}))
			 \end{array}
			 \end{equation}
			 and 
			 \begin{equation}\label{eog1*}
			 \begin{array}{l}
			 \displaystyle
			 \left\|\frac{1}{(1+\widetilde{\eta})}\right\|_{L^{\infty}(0,T;{ H}^{9/2}(\Gamma_{s}))}\leqslant K\|\widetilde{\eta}\|_{L^{\infty}(0,T;{ H}^{9/2}(\Gamma_{s}))}\leqslant K(B_{4}),\quad(\mbox{using \eqref{teB4}}).
			 \end{array}
			 \end{equation}
			 Hence we get the following estimate of $\displaystyle\frac{z\widetilde{\eta}_{x}(\widetilde{\sigma}+\overline{\rho})(\widetilde{ w}_{1}+z\widetilde{\eta}_{t}\vec{e_{2}})_{z}}{(1+\widetilde{\eta})},$
			 \begin{equation}\label{eog12}
			 \begin{split}
			 \displaystyle
			 &\left\|\frac{\widetilde{\eta}_{x}(\widetilde{\sigma}+\overline{\rho})(\widetilde{ w}_{1}+z\widetilde{\eta}_{t}\vec{e_{2}})_{z}}{(1+\widetilde{\eta})}\right\|_{L^{1}(0,T;H^{2}(\Omega))}\\[1.mm]
			 &\displaystyle\leqslant K(\|\widetilde{\eta}_{x}\|_{L^{\infty}(0,T;H^{7/2}(\Gamma_{s}))}\left\|\frac{1}{(1+\widetilde{\eta})}\right\|_{L^{\infty}(0,T;H^{9/2}(\Gamma_{s}))}\|(\widetilde{\sigma}+\overline{\rho})\|_{L^{\infty}(0,T;H^{2}(\Omega))}\\
			 &\qquad\|(\widetilde{ w}_{1}+z\widetilde{\eta}_{t}\vec{e_{2}})_{z}\|_{L^{1}(0,T;H^{2}(\Omega))})
			\qquad\qquad\qquad (\mbox{using Lemma}\,\ref{lfpss})\\[1.mm]
			 & \displaystyle\leqslant K(B_{1},B_{3},B_{4})T^{1/2},\quad(\mbox{using H\"{o}lder's inequality and}\,\eqref{tsB12},\eqref{twB3},\eqref{teB4}\,\mbox{and}\,\eqref{eog1*}).
			 \end{split}
			 \end{equation}
			 Rest of the terms in the expression of $F_{1}(\widetilde{\sigma}+\overline{\rho},\widetilde{w}+z\widetilde{\eta}_{t}\vec{e}_{2},\widetilde{\eta})$ can be estimated in a similar way. Hence we can show the following 
			 \begin{equation}\label{eog13}
			 \begin{array}{l}
			 \|F_{1}(\widetilde{\sigma}+\overline{\rho},\widetilde{w}+z\widetilde{\eta}_{t}\vec{e}_{2},\widetilde{\eta})\|_{L^{1}(0,T;H^{2}(\Omega))}\leqslant K(B_{1},B_{3},B_{4})T^{1/2}.
			 \end{array}
			 \end{equation}
			 We combine \eqref{eog11} and \eqref{eog13} to prove \eqref{eg1}$(i).$\\[2.mm]
			 (ii) We will now prove \eqref{eg1}$(ii).$\\
			 {Estimate of $(\widetilde{\sigma}+\overline{\rho})\mbox{div}(\widetilde{\bf w}+z\widetilde{\eta}_{t}\vec{e}_{2})$ in $L^{\infty}(0,T; H^{1}(\Omega))$:}
               We observe the following
               \begin{equation}\label{eog14}
               \begin{split}
               \displaystyle
               &\|((\widetilde{\sigma}+\overline{\rho})\mbox{div}(\widetilde{\bf w}+z\widetilde{\eta}_{t}\vec{e}_{2}))_{t}\|_{L^{2}(0,T; H^{1}(\Omega))}\\[1.mm]
               &\displaystyle\leqslant K(\|\widetilde{\sigma}_{t}\mbox{div}(\widetilde{\bf w}+z\widetilde{\eta}_{t}\vec{e}_{2})\|_{L^{2}(0,T; H^{1}(\Omega))}+\|(\widetilde{\sigma}+\overline{\rho})\mbox{div}(\widetilde{\bf w}_{t}+z\widetilde{\eta}_{tt}\vec{e}_{2})\|_{L^{2}(0,T; H^{1}(\Omega))})\\[1.mm]
               &\displaystyle\leqslant K(\|\widetilde{\sigma}_{t}\|_{L^{\infty}(0,T; H^{1}(\Omega))}\|\mbox{div}(\widetilde{\bf w}+z\widetilde{\eta}_{t}\vec{e}_{2})\|_{L^{2}(0,T;H^{2}(\Omega))}\\[1.mm]
               &\displaystyle+\|\widetilde{\sigma}+\overline{\rho}\|_{L^{\infty}(0,T;H^{2}(\Omega))}\|\mbox{div}(\widetilde{\bf w}_{t}+z\widetilde{\eta}_{tt}\vec{e}_{2})\|_{L^{2}(0,T; H^{1}(\Omega))})\quad(\mbox{using Lemma}\,\ref{lfpss})\\[1.mm]
               & \displaystyle\leqslant K(B_{1},B_{2},B_{3},B_{4}),\quad(\mbox{using}\,\eqref{tsB12},\eqref{twB3},\eqref{teB4}).
               \end{split}
               \end{equation}
               Now apply the inequality \eqref{eoLit} with $\psi=(\widetilde{\sigma}+\overline{\rho})\mbox{div}(\widetilde{\bf w}+z\widetilde{\eta}_{t}\vec{e}_{2}).$ We obtain
               \begin{equation}\label{eog15}
               \begin{split}
               \displaystyle
               &\|(\widetilde{\sigma}+\overline{\rho})\mbox{div}(\widetilde{\bf w}+z\widetilde{\eta}_{t}\vec{e}_{2})\|_{L^{\infty}(0,T; H^{1}(\Omega))}\\[1.mm]
               &\leqslant K\|\rho_{0}\mbox{div}({\bf u}_{0})\|_{ H^{1}(\Omega)}+T^{1/2}K(B_{1},B_{2},B_{3},B_{4}),\quad(\mbox{using}\,\eqref{eog14}).
               \end{split}
               \end{equation}
              	{Estimate of $F_{1}(\widetilde{\sigma}+\overline{\rho},\widetilde{w}+z\widetilde{\eta}_{t}\vec{e}_{2},\widetilde{\eta})$ in $L^{\infty}(0,T; H^{1}(\Omega))$:}
                 We can have the following estimate
                 \begin{equation}\label{eog16}
                 \begin{split}
                 &\displaystyle\left\|\frac{\widetilde{\eta}_{x}(\widetilde{\sigma}+\overline{\rho})(\widetilde{ w}_{1}+z\widetilde{\eta}_{t}\vec{e_{2}})_{z}}{(1+\widetilde{\eta})}\right\|_{L^{\infty}(0,T; H^{1}(\Omega))}\\[1.mm]
                 &\displaystyle\leqslant K(\|\widetilde{\eta}_{x}\|_{L^{\infty}(0,T;H^{2}(\Gamma_{s}))}\left\|\frac{1}{(1+\widetilde{\eta})}\right\|_{L^{\infty}(0,T;H^{9/2}(\Gamma_{s}))}\|(\widetilde{\sigma}+\overline{\rho})\|_{L^{\infty}(0,T;H^{2}(\Omega))}\\
                 &\qquad\|(\widetilde{ w}_{1}+z\widetilde{\eta}_{t}\vec{e_{2}})_{z}\|_{L^{\infty}(0,T; H^{1}(\Omega))})
                 \qquad\qquad (\mbox{using Lemma}\,\ref{lfpss})\\[1.mm]
                 &\displaystyle\leqslant KT^{1/2}(\|\widetilde{\eta}_{xt}\|_{L^{2}(0,T;H^{2}(\Gamma_{s}))}\left\|\frac{1}{(1+\widetilde{\eta})}\right\|_{L^{\infty}(0,T;H^{9/2}(\Gamma_{s}))}\|(\widetilde{\sigma}+\overline{\rho})\|_{L^{\infty}(0,T;H^{2}(\Omega))}\\
                 &\displaystyle\qquad\|(\widetilde{ w}_{1}+z\widetilde{\eta}_{t}\vec{e_{2}})_{z}\|_{L^{\infty}(0,T; H^{1}(\Omega))})
                 \qquad(\mbox{using \eqref{eoLit} with $\psi=\widetilde{\eta}_{x}$ and the fact that $\widetilde{\eta}_{x}(0)=0$})\\[1.mm]
                 &\displaystyle \leqslant K(B_{1},B_{3},B_{4})T^{1/2},\quad(\mbox{using \eqref{tsB12},\eqref{twB3},\eqref{teB4} and \eqref{eog1*}}).
                 \end{split}
                 \end{equation}
                 A similar analysis can be applied to estimate other summands of $F_{1}(\widetilde{\sigma}+\overline{\rho},\widetilde{w}+z\widetilde{\eta}_{t}\vec{e}_{2},\widetilde{\eta}).$ Hence we can now show that
                 \begin{equation}\label{eog17}
                 \begin{array}{l}
                 \|F_{1}(\widetilde{\sigma}+\overline{\rho},\widetilde{w}+z\widetilde{\eta}_{t}\vec{e}_{2},\widetilde{\eta})\|_{L^{\infty}(0,T; H^{1}(\Omega))}\leqslant K(B_{1},B_{2},B_{4})T^{1/2}.
                 \end{array}
                 \end{equation}
                 Combine \eqref{eog15} with \eqref{eog17} to show \eqref{eg1}$(ii).$
				\end{proof}
				\begin{lem}\label{eog2}
				Let $B^{*}_{0}$ and $T^{*}_{0}$ are as in Lemma \ref{nonempty} and $B_{i}\geqslant B^{*}_{0}$ ($\forall\, 1\leqslant i\leqslant 4$). Then there exist $K_{3}=K_{3}(B_{1},B_{2},B_{3},B_{4})>0,$ $K_{4}=K_{4}(B_{1},B_{4})>0$ and $K_{5}>0$ such that for all $0<T\leqslant T^{*}_{0}(B_{1},B_{2},B_{3},B_{4})$ and $(\widetilde{\sigma},\widetilde{\bf w},\widetilde{\eta})\in \mathscr{C}_{T}(B_{1},B_{2},B_{3},B_{4}),$ $G_{2}(\widetilde{\sigma},\widetilde{\bf w},\widetilde{\eta})$ (defined in \eqref{1.22}) satisfies the following estimates  
					\begin{equation}\label{eg2}
					\begin{array}{lll}
					&(i)&\|{G}_{2}(\widetilde{\sigma},\widetilde{\bf w},\widetilde{\eta})\|_{L^{2}(0,T;{\bf H}^{1}(\Omega))}\leqslant K_{3}(B_{1},B_{2},B_{3},B_{4})T^{1/2},\\[1.mm]
					&(ii)&\|({G}_{2}(\widetilde{\sigma},\widetilde{\bf w},\widetilde{\eta}))_{t}\|_{L^{2}(0,T;{\bf L}^{2}(\Omega))}\leqslant K_{3}(B_{1},B_{2},B_{3},B_{4})T^{1/2}+K_{4}(B_{1},B_{4}),\\[1.mm]
					&(iii)&\|{G}_{2}(\widetilde{\sigma},\widetilde{\bf w},\widetilde{\eta})\|_{L^{\infty}(0,T;{\bf L}^{2}(\Omega))}\leqslant K_{5}\|G_{2}^0\|_{{\bf L}^{2}(\Omega)}+ K_{3}(B_{1},B_{2},B_{3},B_{4})T^{1/2}.\\
					\end{array}
					\end{equation}
					\begin{remark}
						The estimates in \eqref{eg2} are inspired from the results stated in \cite[p. 269]{vallizak} which is done in absence of the beam unknown $\eta$ but includes the evolution of the temperature of the fluid.\\
						We further emphasize that the constant $K_{4}$ does not depend on $(B_{2},B_{3})$ and $K_{5}$ does not depend on any of the $B_{i}$ ($1\leqslant i\leqslant 4$).  
					\end{remark}
					\begin{proof}[Proof of Lemma \ref{eog2}]
					    One can use \eqref{tsgmM} to show that for $\gamma> 1,$
						$$(\widetilde{\sigma}+\overline{\rho})^{\gamma-1}\in C^{0}([0,T];H^{2}(\Omega))\,\,\mbox{and}\,\,(\widetilde{\sigma}+\overline{\rho})^{\gamma-2}\in C^{0}([0,T];H^{2}(\Omega))$$  
						and
						\begin{equation}\label{eog2*}
						\left\{ \begin{split}
						\|(\widetilde{\sigma}+\overline{\rho})^{\gamma-1}\|_{L^{\infty}(0,T;H^{2}(\Omega))}\leqslant K\|\widetilde{\sigma}\|_{L^{\infty}(0,T;H^{2}(\Omega))}\leqslant K(B_{1}),\\
						\|(\widetilde{\sigma}+\overline{\rho})^{\gamma-2}\|_{L^{\infty}(0,T;H^{2}(\Omega))}\leqslant K\|\widetilde{\sigma}\|_{L^{\infty}(0,T;H^{2}(\Omega))}\leqslant K(B_{1}).
						\end{split}\right.
						\end{equation}
						%$$P((\widetilde{\sigma}+\overline{\rho}))\in C^{0}([0,T];H^{2}(\Omega)),\,\, \partial_{t}P((\widetilde{\sigma}+\overline{\rho}))=P_{t}((\widetilde{\sigma}+\overline{\rho}))\in C^{0}([0,T]; H^{1}(\Omega))$$
						%and 
						%\begin{equation}\label{eog2*}
						%\begin{array}{l}
						%\|P((\widetilde{\sigma}+\overline{\rho}))\|_{L^{\infty}(0,T;H^{2}(\Omega))}\leqslant K(B_{1}),\\[1.mm]
						%\|P_{t}((\widetilde{\sigma}+\overline{\rho}))\|_{L^{\infty}(0,T; H^{1}(\Omega))}\leqslant K(B_{1},B_{2}).
						%\end{array}
						%\end{equation}
						 (i)	We first estimate $G_{2}(\widetilde{\sigma},\widetilde{\bf w},\widetilde{\eta})$ in $L^{2}(0,T;{\bf H}^{1}(\Omega)).$\\ 
						{Estimate of $P'(\widetilde{\sigma}+\overline{\rho})\nabla\widetilde{\sigma}$ in $L^{2}(0,T;{\bf H}^{1}(\Omega))$}:\\
							\begin{equation}\label{eog21}
							\begin{split}
							&\|P'(\widetilde{\sigma}+\overline{\rho})\nabla\widetilde{\sigma}\|_{L^{2}(0,T;{\bf H}^{1}(\Omega))}\\[1.mm]
							& \leqslant T^{1/2}\|P'(\widetilde{\sigma}+\overline{\rho})\nabla\widetilde{\sigma}\|_{L^{\infty}(0,T;{\bf H}^{1}(\Omega))}\\[1.mm]
							&\leqslant T^{1/2}K(\|(\widetilde{\sigma}+\overline{\rho})^{\gamma-1}\|_{L^{\infty}(0,T;H^{2}(\Omega))}\|\nabla{\widetilde{\sigma}}\|_{L^{\infty}(0,T;{\bf H}^{1}(\Omega))})\\[1.mm]
							&\qquad\qquad\qquad\qquad\qquad\qquad\qquad\quad(\mbox{using the definition of}\,P\,\mbox{and}\,\mbox{Lemma}\,\ref{lfpss})\\[1.mm]
							& \leqslant K(B_{1})T^{1/2},\quad(\mbox{using}\,\eqref{tsB12}).
							\end{split}
							\end{equation}
							{Estimate of $z\widetilde{\eta}_{tt}(\widetilde{\sigma}+\overline{\rho})\vec{e}_{2}-(\mu \Delta +(\mu+\mu{'})\nabla\mathrm{div})(z\widetilde{\eta}_{t}\vec{e}_{2})$ in $L^{2}(0,T;{\bf H}^{1}(\Omega))$:}\\
							\begin{equation}\label{eog22}
							\begin{split}
							&\|z\widetilde{\eta}_{tt}(\widetilde{\sigma}+\overline{\rho})\vec{e}_{2}-(\mu \Delta +(\mu+\mu{'})\nabla\mathrm{div})(z\widetilde{\eta}_{t}\vec{e}_{2})\|_{L^{2}(0,T;{\bf H}^{1}(\Omega))}\\[1.mm]
							&\leqslant T^{1/2}\|z\widetilde{\eta}_{tt}(\widetilde{\sigma}+\overline{\rho})\vec{e}_{2}-(\mu \Delta +(\mu+\mu{'})\nabla\mathrm{div})(z\widetilde{\eta}_{t}\vec{e}_{2})\|_{L^{\infty}(0,T;{\bf H}^{1}(\Omega))}\\[1.mm]
							&\leqslant T^{1/2}K(\|\widetilde{\eta}_{tt}\|_{L^{\infty}(0,T; H^{1}(\Omega))}\|(\widetilde{\sigma}+\overline{\rho})\|_{L^{\infty}(0,T;H^{2}(\Omega))}+\|\widetilde{\eta}_{t}\|_{L^{\infty}(0,T; {H}^{3}(\Omega))})\\[1.mm]
							&\qquad\qquad\qquad\qquad\qquad\qquad\qquad\qquad\qquad\qquad\qquad\qquad(\mbox{using Lemma}\,\ref{lfpss})\\[1.mm]
							&\leqslant K(B_{1},B_{4})T^{1/2},\quad(\mbox{using}\,\eqref{tsB12},\eqref{teB4}).
							\end{split}
							\end{equation}
							%The linear terms $P'\nabla\widetilde{\sigma}$ and $A(z\widetilde{\eta}_{t}\vec{e_{2}})$ are already estimated in $(i),$ lemma \ref{lfenl}. Let us start by estimating the following non linear term.
							%\begin{equation}\label{3.3}
							%%\|z\widetilde{\eta}_{tt}(\widetilde{\sigma}+\overline{\rho})\vec{e_{2}}\|^{2}_{L^{2}(0,T;{\bf H}^{1}(\Omega))}&\leqslant K_{36}(\|\eta_{tt}\|^{2}_{L^{\infty}(0,T; H^{1}(\Gamma_{s}))}\|(\widetilde{\sigma}+\overline{\rho})\|^{2}_{L^{\infty}(0,T;H^{2}(\Omega))})T\\[1.mm]
							%& \leqslant K_{37}(B_{2},B_{3})T.\\[1.mm]
							%\end{split}
							%\end{equation}
							%The first inequality of \eqref{3.3} follows by an application of $(i),$ lemma \ref{lfpss}. The second inequality is a consequence of the norm bounds in the definition of $\mathscr{C}_{T}.$
							{Estimate of $F_{2}(\widetilde{\sigma}+\overline{\rho},\widetilde{\bf w}+z\widetilde{\eta}_{t}\vec{e}_{2},\widetilde{\eta})$ (defined in \eqref{F123}) in $L^{2}(0,T;{\bf H}^{1}(\Omega))$:}
						    We will only estimate the terms of $F_{2}(\widetilde{\sigma}+\overline{\rho},\widetilde{\bf w}+z\widetilde{\eta}_{t}\vec{e}_{2},\widetilde{\eta})$ which are the most intricate to deal with. The others are left to the reader.
							\begin{equation}\label{eog23}
							\begin{array}{ll}
							(a)&\quad\|\widetilde{\eta}(\widetilde{\sigma}+\overline{\rho})(\widetilde{\bf w}_{t}+z\widetilde{\eta}_{tt}\vec{e_{2}})\|_{L^{2}(0,T;{\bf H}^{1}(\Omega))}\\[1.mm]
							& \leqslant T^{1/2}\|\widetilde{\eta}(\widetilde{\sigma}+\overline{\rho})(\widetilde{\bf w}_{t}+z\widetilde{\eta}_{tt}\vec{e_{2}})\|_{L^{\infty}(0,T;{\bf H}^{1}(\Omega))}\\[1.mm]
							&\leqslant T^{1/2}K(\|\widetilde{\eta}\|_{L^{\infty}(0,T;{H}^{2}(\Gamma_{s}))}\|(\widetilde{\sigma}+\overline{\rho})\|_{L^{\infty}(0,T;{H}^{2}(\Omega))}
							\|(\widetilde{\bf w}_{t}+z\widetilde{\eta}_{tt}\vec{e_{2}})\|_{{L^{\infty}(0,T;{\bf H}^{1}(\Omega))}})\\[1.mm]
							&\qquad\qquad\qquad\qquad\qquad\qquad\qquad\qquad\qquad\qquad\qquad\qquad\qquad\qquad(\mbox{using Lemma}\,\ref{lfpss})\\[1.mm]
							&\leqslant K(B_{1},B_{3},B_{4})T^{1/2},\quad(\mbox{using}\,\eqref{tsB12},\eqref{twB3},\eqref{teB4}).
							\end{array}
							\end{equation}
							\begin{equation}\label{eog24}
							\begin{array}{ll}
							(b)&\quad\|z(\widetilde{\sigma}+\overline{\rho})(\widetilde{\bf w}_{z}+\widetilde{\eta}_{t}\vec{e_{2}})\widetilde{\eta}_{t}\|_{L^{2}(0,T;{\bf H}^{1}(\Omega))}\\[1.mm]
							& \leqslant T^{1/2}\|z(\widetilde{\sigma}+\overline{\rho})(\widetilde{\bf w}_{z}+\widetilde{\eta}_{t}\vec{e_{2}})\widetilde{\eta}_{t}\|_{L^{\infty}(0,T;{\bf H}^{1}(\Omega))}\\[1.mm]
	                        &\leqslant T^{1/2} K(\|(\widetilde{\sigma}+\overline{\rho})\|_{L^{\infty}(0,T;H^{2}(\Omega))}\|(\widetilde{\bf w}_{z}+\widetilde{\eta}_{t}\vec{e_{2}})\|_{L^{\infty}(0,T;{\bf H}^{1}(\Omega))}\|\widetilde{\eta}_{t}\|_{L^{\infty}(0,T;H^{2}(\Gamma_{s}))})\\[1.mm]
	                        &\qquad\qquad\qquad\qquad\qquad\qquad\qquad\qquad\qquad\qquad\qquad\qquad\qquad\qquad(\mbox{using Lemma}\,\ref{lfpss})\\[1.mm]
							&\leqslant  K(B_{1},B_{3},B_{4})T^{1/2},\quad(\mbox{using}\,\eqref{tsB12},\eqref{twB3},\eqref{teB4}).
							\end{array}
							\end{equation}
							\begin{equation}\label{eog25}
							\begin{array}{ll}
							(c) &\displaystyle\quad\left\|\frac{\widetilde{\bf w}_{zz}z^{2}\eta^{2}_{x}}{(1+\widetilde\eta)}\right\|_{L^{2}(0,T;{\bf H}^{1}(\Omega))}\\[1.mm]
							&\displaystyle\leqslant K(\|\widetilde{\bf w}_{zz}\|_{L^{2}(0,T;{\bf H}^{1}(\Omega))}\|\widetilde{\eta}^{2}_{x}\|_{L^{\infty}(0,T;H^{2}(\Gamma_{s}))}\left\|\frac{1}{(1+\widetilde{\eta})}\right\|_{L^{\infty}(0,T;H^{9/2}(\Gamma_{s}))})\\
							&\qquad\qquad\qquad\qquad\qquad\qquad\qquad\qquad\qquad\qquad\qquad\qquad\qquad\qquad\,\,(\mbox{using Lemma}\,\ref{lfpss})\\[1.mm]
							&\displaystyle\leqslant T^{1/2}K(\|\widetilde{\bf w}_{zz}\|_{L^{2}(0,T;{\bf H}^{1}(\Omega))}\|\widetilde{\eta}^{2}_{xt}\|_{L^{2}(0,T;H^{2}(\Gamma_{s}))}\left\|\frac{1}{(1+\widetilde{\eta})}\right\|_{L^{\infty}(0,T;H^{9/2}(\Gamma_{s}))})\\[1.mm]
							&\displaystyle\qquad\qquad\qquad\qquad\qquad\qquad(\mbox{using}\,\eqref{eoLit}\,\mbox{with}\,\psi=\widetilde{\eta}^{2}_{x}\,\mbox{and the fact}\,\widetilde{\eta}_{x}(,0)=0)\\[1.mm]
							&\displaystyle\leqslant K(B_{3},B_{4})T^{1/2},\quad(\mbox{using}\,\eqref{tsB12},\eqref{twB3},\eqref{teB4}\,\mbox{and}\,\eqref{eog1*}).
							\end{array}
							\end{equation}
							 \,\,\,\,\,\,\quad$(d)$\quad Using arguments similar to that in the computation \eqref{eog21} we show the following
							\begin{equation}\label{eog26}
							\begin{split}
							\|(\widetilde{\eta} P'\widetilde{\sigma}_{x}- P'\widetilde{\sigma}_{z}z\widetilde{\eta}_{x})\vec{e_{1}}\|_{L^{2}(0,T; {\bf H}^{1}(\Omega))}
							\leqslant K(B_{1},B_{4})T^{1/2}.
							\end{split}
							\end{equation}
							Now the reader can deal with the other terms using similar arguments in order to prove
							\begin{equation}\label{eog27}
							\begin{array}{l}
							\|{F}_{2}(\widetilde{\sigma}+\overline{\rho},(\widetilde{\bf{w}}+z\widetilde\eta_{t}\vec{e_{2}}),\widetilde{\eta})\|_{L^{2}(0,T;{\bf H}^{1}(\Omega))}\leqslant K(B_{1},B_{3},B_{4})T^{1/2}.
							\end{array}
							\end{equation}
							Combining the estimates \eqref{eog21}, \eqref{eog22} and \eqref{eog27} we conclude the proof of the inequality \eqref{eg2}$(i)$.\\[2.mm]
							(ii)
							We now estimate  $\|(G_{2}(\widetilde{\sigma},\widetilde{\bf w},\widetilde{\eta}))_{t}\|_{L^{2}(0,T;{\bf L}^{2}(\Omega))}.$\\
							{Estimate of $(P'(\widetilde{\sigma}+\overline{\rho})\nabla\widetilde{\sigma})_{t}$ in $L^{2}(0,T;{\bf L}^{2}(\Omega))$:}
							\begin{equation}\label{eog28}
							\begin{split}
						    &\|(P'(\widetilde{\sigma}+\overline{\rho})\nabla\widetilde{\sigma})_{t}\|_{L^{2}(0,T;{\bf L}^{2}(\Omega))}\\[1.mm]
						    &\leqslant T^{1/2}\|(P'(\widetilde{\sigma}+\overline{\rho})\nabla\widetilde{\sigma})_{t}\|_{L^{\infty}(0,T;{\bf L}^{2}(\Omega))}\\[1.mm]
						    &\leqslant T^{1/2}K(\|(\widetilde{\sigma}+\overline{\rho})^{(\gamma-2)}\widetilde{\sigma}_{t}\nabla{\widetilde\sigma}\|_{L^{\infty}(0,T;{\bf L}^{2}(\Omega))}+\|(\widetilde{\sigma}+\overline{\rho})^{\gamma-1}\nabla\widetilde{\sigma}_{t}\|_{L^{\infty}(0,T;{\bf L}^{2}(\Omega))})\\[1.mm]
						    &\leqslant T^{1/2}K(\|(\widetilde{\sigma}+\overline{\rho})^{\gamma-2}\|_{L^{\infty}(0,T;H^{2}(\Omega))}\|\widetilde{\sigma}_{t}\|_{L^{\infty}(0,T; H^{1}(\Omega))}\|\nabla\widetilde{\sigma}\|_{L^{\infty}(0,T; {\bf H}^{1}(\Omega))}\\[1.mm]
						    &\qquad+\|(\widetilde{\sigma}+\overline{\rho})^{\gamma-1}\|_{L^{\infty}(0,T;H^{2}(\Omega))}\|\nabla\widetilde{\sigma}_{t}\|_{L^{\infty}(0,T;{\bf L}^{2}(\Omega))})
						   \qquad(\mbox{using Lemma}\,\ref{lfpss})\\[1.mm]
						    &\leqslant K(B_{1},B_{2})T^{1/2},\quad(\mbox{using}\,\eqref{tsB12}\,\mbox{and}\,\eqref{eog2*}).
							\end{split}
							\end{equation}
							{Estimate of $(z\widetilde{\eta}_{tt}(\widetilde{\sigma}+\overline{\rho})\vec{e_{2}}-(\mu \Delta +(\mu+\mu{'})\nabla\mathrm{div})(z\widetilde{\eta}_{t}\vec{e}_{2}))_{t}$ in $L^{2}(0,T;{\bf L}^{2}(\Omega))$:}\\
							\begin{equation}\label{eog29}
							\begin{split}
							&\|(z\widetilde{\eta}_{tt}(\widetilde{\sigma}+\overline{\rho})\vec{e_{2}}-(\mu \Delta +(\mu+\mu{'})\nabla\mathrm{div})(z\widetilde{\eta}_{t}\vec{e}_{2}))_{t}\|_{L^{2}(0,T;{\bf L}^{2}(\Omega))}\\[1.mm]
							&\leqslant T^{1/2}K(\|\widetilde{\eta}_{tt}\|_{L^{\infty}(0,T; H^{1}(\Gamma_{s}))}\|\widetilde{\sigma}_{t}\|_{L^{\infty}(0,T; H^{1}(\Omega))})
							 +K(\|\widetilde{\eta}_{ttt}\|_{L^{2}(0,T;L^{2}(\Gamma_{s}))}\|(\widetilde{\sigma}+\overline{\rho})\|_{L^{\infty}(0,T;H^{2}(\Omega))}\\[1.mm]
							 & \qquad+\|\widetilde{\eta}_{tt}\|_{L^{2}(0,T;H^{2}(\Gamma_{s}))})
							 \qquad(\mbox{using Lemma}\,\ref{lfpss})\\[1.mm]
							 & \leqslant K(B_{2},B_{4})T^{1/2}+K(B_{1},B_{4}).
							\end{split}
							\end{equation}
							%The time derivatives of the linear terms $P'\nabla\widetilde{\sigma}$ and $A(z\widetilde{\eta}_{t}\vec{e_{2}})$ are already estimated in $(ii),$ lemma \ref{lfenl}. Let us start by estimating the following non linear term. 
							%\begin{equation}\label{3.9}
							%\begin{array}{ll}
							%\|(z\widetilde{\eta}_{tt}(\widetilde{\sigma}+\overline{\rho})\vec{e_{2}})_{t}\|^{2}_{L^{2}(0,T;{\bf L}^{2}(\Omega))}&\leqslant K_{46}(\|\widetilde{\eta}_{tt}\|^{2}_{L^{\infty}(0,T; H^{1}(\Gamma_{s}))}\|(\widetilde{\sigma}_{t})\|^{2}_{L^{\infty}(0,T; H^{1}(\Omega))})T\\[1.mm]
							%& +K_{47}(\|\widetilde{\eta}_{ttt}\|^{2}_{L^{2}(0,T;L^{2}(\Gamma_{s}))}\|(\widetilde{\sigma}+\overline{\rho})\|^{2}_{L^{\infty}(0,T;H^{2}(\Omega))})\\[1.mm]
							%& \leqslant K_{48}(B_{3},B_{4})T+K_{49}(B_{2},B_{3}).\\[1.mm]
							%\end{array}
							%\end{equation}
							%To prove the first inequality of \eqref{3.9} we have used $(ii),$ lemma \ref{lfpss} in order to estimate the product of $\widetilde{\eta}_{tt}$ and $\sigma_{t}.$ In \eqref{3.9} the product of $\widetilde{\eta}_{ttt}$ and $(\widetilde{\sigma}+\overline{\rho})$ is estimated using $(i),$ lemma \ref{lfpss}. 
							{Estimate of $({F}_{2}(\widetilde{\sigma}+\overline{\rho},(\widetilde{\bf w}+z\widetilde\eta_{t}\vec{e_{2}}),\widetilde{\eta}))_{t}$ in $L^{2}(0,T;{\bf L}^{2}(\Omega))$:}\\
						\begin{equation}\label{eog210}
							\begin{array}{ll}
							(a)\quad&\|(\widetilde{\eta}(\widetilde{\sigma}+\overline{\rho})(\widetilde{\bf w}_{t}+z\widetilde{\eta}_{tt}\vec{e_{2}}))_{t}\|_{L^{2}(0,T;{\bf L}^{2}(\Omega))}\\[1.mm]
							&\leqslant K(\|(\widetilde{\eta}_{t}(\widetilde{\sigma}+\overline{\rho})(\widetilde{\bf w}_{t}+z\widetilde{\eta}_{tt}\vec{e_{2}}))\|_{L^{2}(0,T;{\bf L}^{2}(\Omega))}
							+\|\widetilde{\eta}\widetilde{\sigma}_{t}(\widetilde{\bf w}_{t}+z\widetilde{\eta}_{tt}\vec{e_{2}})\|_{L^{2}(0,T;{\bf L}^{2}(\Omega))}\\[1.mm]
							 &\qquad+\|(\widetilde{\eta}(\widetilde{\sigma}+\overline{\rho})(\widetilde{\bf w}_{tt}+z\widetilde{\eta}_{ttt}\vec{e_{2}}))\|_{L^{2}(0,T;{\bf L}^{2}(\Omega))})\\[1.mm]
							 &\leqslant T^{1/2}K(\|\widetilde{\eta}_{t}\|_{L^{\infty}(0,T;H^{2}(\Gamma_{s}))}\|(\widetilde{\sigma}+\overline{\rho})\|_{L^{\infty}(0,T;H^{2}(\Omega))}\|(\widetilde{\bf w}_{t}+z\widetilde{\eta}_{tt}\vec{e_{2}})\|_{L^{\infty}(0,T;{\bf H}^{1}(\Omega))}\\[1.mm]
							 &\qquad+\|\widetilde{\eta}\|_{L^{\infty}(0,T;H^{2}(\Gamma_{s}))}\|\widetilde{\sigma}_{t}\|_{L^{\infty}(0,T; H^{1}(\Omega))}\|(\widetilde{\bf w}_{t}+z\widetilde{\eta}_{tt}\vec{e_{2}})\|_{L^{\infty}(0,T;{\bf H}^{1}(\Omega))})\\[1.mm]
							 &\qquad+\|\widetilde{\eta}\|_{L^{\infty}(0,T;H^{2}(\Gamma_{s}))}\|(\widetilde{\sigma}+\overline{\rho})\|_{L^{\infty}(0,T;H^{2}(\Omega))}\|(\widetilde{\bf w}_{tt}+z\widetilde{\eta}_{ttt}\vec{e_{2}})\|_{L^{2}(0,T;{\bf L}^{2}(\Omega))}\\[1.mm]
							 & \leqslant T^{1/2}K(B_{1},B_{2},B_{3},B_{4})+T^{1/2}\|\widetilde{\eta}_{t}\|_{L^{2}(0,T;H^{2}(\Gamma_{s}))}\|(\widetilde{\sigma}+\overline{\rho})\|_{L^{\infty}(0,T;H^{2}(\Omega))}\\
							 &\qquad\|(\widetilde{\bf w}_{tt}+z\widetilde{\eta}_{ttt}\vec{e_{2}})\|_{L^{2}(0,T;{\bf L}^{2}(\Omega))}
							 \qquad(\mbox{using }\,\eqref{eoLit}\,\mbox{with}\,\psi=\widetilde{\eta}\,\mbox{and the fact}\,\widetilde{\eta}(,0)=0)\\[1.mm]
							& \leqslant K(B_{1},B_{2},B_{3},B_{4})T^{1/2}.
							\end{array}
							\end{equation}
							\,\,\,\,\,\,\,$(b)$\quad Using similar estimates we can have the following
							\begin{equation}\label{eog211}
							\begin{array}{l}
							\|(z(\widetilde{\sigma}+\overline{\rho})(\widetilde{\bf w}_{z}+\widetilde{\eta}_{t}\vec{e_{2}})\eta_{t})_{t}\|_{L^{2}(0,T;{\bf L}^{2}(\Omega))}\leqslant K(B_{1},B_{2},B_{3},B_{4})T^{1/2}.
							\end{array}
							\end{equation}
							\,\,\,\,\,\,\,$(c)$\quad Now we estimate 
							$$	\left\|\left(\frac{\widetilde{\bf w}_{zz}z^{2}\widetilde{\eta}^{2}_{x}}{(1+\widetilde{\eta})}\right)_{t}\right\|_{L^{2}(0,T;{\bf L}^{2}(\Omega))}.$$
							To start with, we have the following identity of distributional derivatives
							\begin{equation}
								\label{HugeTerms}
								\left(\frac{\widetilde{\bf w}_{zz}z^{2}\widetilde{\eta}^{2}_{x}}{(1+\widetilde{\eta})}\right)_{t}=\frac{z^{2}\widetilde{\bf w}_{tzz}\widetilde{\eta}^{2}_{x}}{(1+\widetilde{\eta})}+\frac{2\widetilde{\eta}_{x}\widetilde{\eta}_{xt}\widetilde{\bf w}_{zz}}{(1+\widetilde{\eta})}-\frac{\widetilde{\bf w}_{zz}z^{2}\widetilde{\eta}^{2}_{x}\widetilde{\eta}_{t}}{(1+\widetilde{\eta})^{2}}.
							\end{equation}
We now estimate the first term of the summands. Using \eqref{eog1*} one obtains
							\begin{equation}\label{eog212}
							\begin{array}{l}
							\displaystyle
							\left\|\frac{z^{2}\widetilde{\bf w}_{tzz}\widetilde{\eta}^{2}_{x}}{(1+\widetilde{\eta})}\right\|_{L^{2}(0,T;{\bf L}^{2}(\Omega))}\leqslant K(B_{4})(\|\widetilde{\bf w}_{tzz}\|_{{L^{2}(0,T;{\bf L}^2(\Omega))}}\|\widetilde{\eta}_{x}\|_{{{L^{\infty}(\Sigma^{s}_{T})}}}^2).
							\end{array}
							\end{equation}
							Now we use inequality \eqref{eoLit} and $\widetilde{\eta}_{x}(\cdot,0)=0$ to get
							\begin{equation}\label{eog213}
							\begin{array}{l}
							\displaystyle
							\|\widetilde{\eta}_{x}\|_{{{L^{\infty}(\Sigma^{s}_{T})}}}
							\leq
							C \|\widetilde{\eta}_{x}\|_{L^{\infty}(0,T;{ H}^{2}(\Gamma_{s}))}\leqslant K(B_{3})T^{1/2}.
							\end{array}
							\end{equation}
							Hence we use \eqref{eog213} in \eqref{eog212} to obtain
							\begin{equation}\label{eog214}
							\begin{array}{l}
							\displaystyle
							\left\|\frac{z^{2}\widetilde{\bf w}_{tzz}\widetilde{\eta}^{2}_{x}}{(1+\widetilde{\eta})}\right\|_{L^{2}(0,T;{\bf L}^{2}(\Omega))}\leqslant K(B_{3},B_{4})T.
							\end{array}
							\end{equation} 
							For the second and third summands of \eqref{HugeTerms}, we similarly obtain:
							\begin{equation}
							\begin{array}{l}
							\displaystyle
							\left\|\frac{2\widetilde{\eta}_{x}\widetilde{\eta}_{xt}\widetilde{\bf w}_{zz}}{(1+\widetilde{\eta})}\right\|_{L^{2}(0,T;{\bf L}^{2}(\Omega))}\leqslant K(B_{3},B_{4})T^{1/2},  \text{ and } 
							\left\|-\frac{\widetilde{\bf w}_{zz}z^{2}\widetilde{\eta}^{2}_{x}\widetilde{\eta}_{t}}{(1+\widetilde{\eta})^{2}} \right\|_{L^{2}(0,T;{\bf L}^{2}(\Omega))}\leqslant K(B_{3},B_{4})T.
							\end{array}
							\end{equation} 

							So altogether we get
							\begin{equation}\label{eog215}
							\begin{array}{l}
							\displaystyle
							\left\|\left(\frac{\widetilde{\bf w}_{zz}z^{2}\widetilde\eta^{2}_{x}}{(1+\widetilde\eta)}\right)_{t}\right\|_{L^{2}(0,T;{\bf L}^{2}(\Omega))}\leqslant K(B_{3},B_{4})T^{1/2}.
							\end{array}
							\end{equation}
							The remaining terms in the expression of ${F}_{2}$ are relatively easier to deal with and hence we leave the details to the reader to show
							\begin{equation}\label{eog216}
							\begin{array}{l}
							\|({F}_{2}(\widetilde{\sigma}+\overline{\rho},\widetilde{\bf w}+z\widetilde\eta_{t}\vec{e_{2}},\widetilde{\eta}))_{t}\|_{L^{2}(0,T;{\bf L}^{2}(\Omega))}\leqslant K(B_{1},B_{2},B_{3},B_{4})T^{1/2}.
							\end{array}
							\end{equation}
							Hence combining the estimates \eqref{eog28}, \eqref{eog29} and \eqref{eog216} one gets \eqref{eg2}$(ii).$\\[1.mm]
							%\begin{equation}\label{eog217}
							%\begin{array}{l}
							%\|(G_{2}(\widetilde{\bf w},\widetilde{\sigma},\eta))_{t}\|^{2}_{L^{2}(0,T;{\bf L}^{2}(\Omega))}\leqslant K(B_{1},B_{2},B_{3},B_{4})T+K(B_{1},B_{4}).
							%\end{array}
							%\end{equation}\\[1.mm]
							(iii) In \eqref{eoLit} replace $\psi$ by $G_{2}(\widetilde{\bf w},\widetilde{\sigma},\widetilde{\eta})$ and use the estimate \eqref{eg2}$(ii)$ to prove \eqref{eg2}$(iii).$
					\end{proof}
				\end{lem}
				\begin{lem}\label{eog3}
							Let $B^{*}_{0}$ and $T^{*}_{0}$ are as in Lemma \ref{nonempty} and $B_{i}\geqslant B^{*}_{0}$ ($\forall\, 1\leqslant i\leqslant 4$). Then there exist $K_{6}>0$ and $K_{7}=K_{7}(B_{1},B_{2},B_{3},B_{4})>0$ such that for all $0<T\leqslant T^{*}_{0}(B_{1},B_{2},B_{3},B_{4})$ and $(\widetilde{\sigma},\widetilde{\bf w},\widetilde{\eta})\in \mathscr{C}_{T}(B_{1},B_{2},B_{3},B_{4}),$ $G_{3}(\widetilde{\sigma},\widetilde{\bf w},\widetilde{\eta})$ (defined in \eqref{1.22}) satisfies the following estimates 
						\begin{equation}\label{eg3}
						\begin{array}{ll}
						(i)& \|G_{3}(\widetilde{\sigma},\widetilde{\bf w},\widetilde{\eta})\|_{L^{\infty}(0,T;H^{1/2}(\Gamma_{s}))}\leqslant K_{6}\|(\rho_{0},{\bf u}_{0})\|_{H^{2}(\Omega)\times {\bf H}^{2}(\Omega)}+K_{7}(B_{1},B_{2},B_{3},B_{4})T^{1/2},\\[1.mm]
						(ii)& \|(G_{3}(\widetilde{\sigma},\widetilde{\bf w},\widetilde{\eta}))_{t}\|_{L^{2}(0,T;L^{2}(\Gamma_{s}))}\leqslant T^{1/6}K_{7}(B_{1},B_{2},B_{3},B_{4}).
						\end{array}
						\end{equation}
				\end{lem}
				\begin{remark}
					We emphasize that $K_{6}$ does not depend on any of the $B_{i}$ ($1\leqslant i\leqslant 4$).
				\end{remark}
				\begin{proof}
					In this proof we will consider the function $\widetilde{\bf w}$ and $(\widetilde{\sigma}+\overline{\rho})$ on $\Gamma_{s},$ i.e we take the trace of these functions and make use of well known trace theorem without mentioning it explicitly.\\[1.mm]
					(i) {Estimate of $F_{3}(\widetilde{\sigma}+\overline{\rho},\widetilde{\bf w}+\widetilde{\eta}_{t}\vec{e}_{2},\widetilde{\eta})$ (defined in \eqref{F123}) in $L^{\infty}(0,T;H^{1/2}(\Gamma_{s}))$:}\\
					(a)\,\,{First let us estimate  $(\widetilde{w}_{2,z}+\widetilde{\eta}_{t})$ in $L^{\infty}(0,T;H^{1/2}(\Gamma_{s}))$:}
					\begin{equation}\label{eog32}
					\begin{array}{ll}
					\displaystyle
					&\|(\widetilde{ w}_{2,z}+\widetilde{\eta}_{t})\|_{L^{\infty}(0,T;H^{1/2}(\Gamma_{s}))}\leqslant K(\|{\bf u}_{0}\|_{H^{2}(\Omega)}+T^{1/2}\|(\widetilde{ w}_{2,z}+\widetilde{\eta}_{t})_{t}\|_{L^{2}(0,T;H^{1/2}(\Gamma_{s}))})\\
					&\qquad\qquad\qquad\qquad\qquad\qquad\qquad\qquad\qquad\qquad\qquad\qquad\qquad\qquad(\mbox{using}\,\eqref{eoLit})\\[1.mm]
					&\qquad\qquad\qquad\qquad\,\leqslant K(\|{\bf u}_{0}\|_{H^{2}(\Omega)}+T^{1/2}K(B_{3},B_{4})),\qquad(\mbox{using}\,\eqref{twB3}\,\mbox{and}\,\eqref{teB4}).
					\end{array}
					\end{equation}
					 (b)\,\,Let us estimate $P(\widetilde{\sigma}+\overline{\rho})$ in $L^{\infty}(0,T;H^{1/2}(\Gamma_{s})).$
					\begin{equation}\label{eog31}
					\begin{array}{ll}
					&\|P(\widetilde{\sigma}+\overline{\rho})\|_{L^{\infty}(0,T;H^{1/2}(\Gamma_{s}))}\leqslant K(\|\rho_{0}\|_{H^{2}(\Omega)}+T^{1/2}\|(\widetilde{\sigma}+\overline{\rho})^{\gamma-1}\widetilde{\sigma}_{t}\|_{L^{2}(0,T;H^{1/2}(\Gamma_{s}))}ds)\\
					&\qquad\qquad\qquad\qquad\qquad\qquad\qquad\qquad\qquad\qquad\qquad\qquad\qquad\qquad\qquad\qquad (\mbox{using}\,\eqref{eoLit} )\\[1.mm]
					&\qquad\qquad\qquad\qquad\qquad\qquad\leqslant K(\|\rho_{0}\|_{H^{2}(\Omega)}+T^{1/2}K(B_{1},B_{2})),\qquad(\mbox{using}\,\eqref{tsB12}).
					\end{array}
					\end{equation}
			(c)\,\,Now we estimate $\displaystyle\frac{\widetilde{\eta}_{x}\widetilde{ w}_{1,z}}{(1+\widetilde{\eta})}$ in $L^{\infty}(0,T;H^{1/2}(\Gamma_{s})),$
					\begin{equation}\label{eog33}
					\begin{array}{ll}
					&\displaystyle\|\frac{\widetilde{\eta}_{x}\widetilde{ w}_{1,z}}{(1+\widetilde{\eta})}\|_{L^{\infty}(0,T;H^{1/2}(\Gamma_{s}))}\leqslant K(B_{4})(\|\widetilde{\eta}_{x}\|_{L^{\infty}(0,T;H^{2}(\Gamma_{s}))}\|\widetilde{ w}_{1,z}\|_{L^{\infty}(0,T;H^{1/2}(\Gamma_{s}))})\\
					&\displaystyle\qquad\qquad\qquad\qquad\qquad\qquad\qquad\qquad\qquad\qquad\qquad\qquad\qquad\qquad\qquad(\mbox{using}\,\eqref{eog1*})\\[1.mm]
					&\displaystyle \qquad\qquad\qquad\qquad\qquad\quad\leqslant T^{1/2}K(B_{3},B_{4}),\qquad(\mbox{using}\,\eqref{eoLit}\,\mbox{with}\,\psi=\widetilde{\eta}_{x}\,\mbox{and}\,\widetilde{\eta}_{x}(.,0)=0).
					\end{array}
					\end{equation}
				We use similar sort of arguments to show that 
				\begin{equation}\label{eog34}
				\begin{split}
				\|F_{3}(\widetilde{\sigma}+\overline{\rho},\widetilde{\bf w}+\widetilde{\eta}_{t}\vec{e}_{2},\widetilde{\eta})\|_{L^{\infty}(0,T;H^{1/2}(\Gamma_{s}))}\leqslant K\|\rho_{0}\|_{H^{2}(\Gamma_{s})}+K(B_{1},B_{2},B_{3},B_{4})T^{1/2}.
				\end{split}
				\end{equation}
				Combine \eqref{eog32} with \eqref{eog34} to prove \eqref{eg3}$(i).$\\[1.mm]
				(ii) {Estimate of $(F_{3}(\widetilde{\sigma}+\overline{\rho},\widetilde{\bf w}+\widetilde{\eta}_{t}\vec{e}_{2},\widetilde{\eta}))_{t}$ (defined in \eqref{F123}) in $L^{2}(0,T;L^{2}(\Gamma_{s}))$:}\\
				{First let us estimate  $(\widetilde{w}_{2,z}+\widetilde{\eta}_{t}\vec{e}_{2})_{t}$ in $L^{2}(0,T;L^{2}(\Gamma_{s}))$:}
				\begin{equation}\label{eog35}
				\begin{split}
				&\|(\widetilde{w}_{2,z}+\widetilde{\eta}_{t})_{t}\|_{L^{2}(0,T;L^{2}(\Gamma_{s}))}\\[1.mm]
				&\leqslant \|\widetilde{w}_{2,tz}\|_{L^{2}(0,T;L^{2}(\Gamma_{s}))}+\|\widetilde{\eta}_{tt}\|_{L^{2}(0,T;L^{2}(\Gamma_{s}))}\\[1.mm]
				&\leqslant T^{1/6}K(\|\widetilde{\bf w}_{t}(\cdot,0)\|_{{\bf H}^{1}_{0}(\Omega)}+\|\widetilde{\bf w}_{t}\|_{{\bf H}^{2,1}_{\Sigma_{T}}(Q_{T})})+T^{1/2}K\|\widetilde{\eta}_{tt}\|_{L^{\infty}(0,T;L^{2}(\Gamma_{s}))}\\[1.mm]
				&\qquad\qquad\qquad\qquad\qquad\qquad\qquad\qquad(\mbox{using Lemma}\,\ref{leoin}\,\mbox{with {\bf f} replaced by}\,\widetilde{\bf w}_{t})\\[1.mm]
				&\leqslant T^{1/6}K(\left\|\frac{1}{\rho_{0}}\big(G_{2}\mid_{t=0}-(-\mu \Delta -(\mu+\mu{'})\nabla\mathrm{div})({\bf u}_{0}-z\eta_{1}\vec{e_{2}})\big)\right\|_{{\bf H}^{1}_{0}(\Omega)}\\
				&\qquad\qquad+K(B_{3}))+T^{1/2}K(B_{4}).
				\qquad(\mbox{using}\,\eqref{twt0}\,\mbox{and the inequalities}\,\eqref{twB3}\,\mbox{qnd}\,\eqref{teB4})
				\end{split}
				\end{equation}
				Using similar line of arguments one can prove that the trace of $(F_{3}(\widetilde{\sigma}+\overline{\rho},\widetilde{\bf w}+\widetilde{\eta}_{t}\vec{e}_{2},\widetilde{\eta}))_{t}$ on $\Gamma_{s}$ belongs to $L^{2}(0,T;L^{2}(\Gamma_{s}))$ and the following inequality is true for $T<1,$
				\begin{equation}\label{eog36}
					\begin{split}
						\|(F_{3}(\widetilde{\sigma}+\overline{\rho},\widetilde{\bf w}+\widetilde{\eta}_{t}\vec{e}_{2},\widetilde{\eta}))_{t}\|_{L^{2}(0,T;L^{2}(\Gamma_{s}))}\leqslant T^{1/6}K(B_{1},B_{2},B_{3},B_{4}).
					\end{split}
				\end{equation}
			Combining \eqref{eog35} and \eqref{eog36}, we conclude \eqref{eg3}$(ii).$
				\end{proof}
				\begin{lem}\label{eoWs}
							Let $B^{*}_{0}$ and $T^{*}_{0}$ are as in Lemma \ref{nonempty} and $B_{i}\geqslant B^{*}_{0}$ ($\forall\, 1\leqslant i\leqslant 4$). Then there exist $K_{8}=K_{8}(B_{1},B_{2},B_{3},B_{4})>0,$ and $K_{9}=K_{9}(B_{3},B_{4})>0$ such that for all $0<T\leqslant T_{0}^{*}(B_{1},B_{2},B_{3},B_{4})$ and $(\widetilde{\sigma},\widetilde{\bf w},\widetilde{\eta})\in \mathscr{C}_{T}(B_{1},B_{2},B_{3},B_{4}),$ we have the following estimates
						(recall the notation $\widetilde{W}$ from \eqref{dotW})
						\begin{equation}\label{esW}
						\begin{array}{ll}
						(i)&\|\widetilde{W}({\widetilde{\bf w},\widetilde{\eta}})\|_{L^{1}(0,T;{\bf H}^{3}(\Omega))}\leqslant K_{8}(B_{1},B_{2},B_{3},B_{4})T^{1/2},\\[1.mm]
						(ii)&\|\widetilde{W}({\widetilde{\bf w},\widetilde{\eta}})\|_{L^{\infty}(0,T;{\bf H}^{2}(\Omega))}\leqslant K_{9}(B_{3},B_{4})+K_{8}(B_{1},B_{2},B_{3},B_{4})T^{1/2},\\[1.mm]
						(iii)&\|\widetilde{\sigma}_{t}\|_{L^{2}(0,T;L^{3}(\Omega))}+\|\nabla\widetilde{\sigma}\|_{L^{2}(0,T;L^{3}(\Omega))}\leqslant K_{8}(B_{1},B_{2},B_{3},B_{4})T^{1/2}.
						\end{array}
						\end{equation} 
				\end{lem}
				\begin{remark}
					We emphasize that $K_{9}$ does not depend on $B_{1}$ and $B_{2}.$
				\end{remark}
				\begin{proof}
					(i) One can use Lemma \ref{lfpss} to check that $\widetilde{W}({\widetilde{\bf w},\widetilde{\eta}})\in L^{2}(0,T; {\bf H}^{3}(\Omega)).$ As a consequence, \eqref{esW}$(i)$ follows.\\[1.mm]
					(ii) The following estimates follow from the regularity of $\widetilde{\bf w}.$
					\begin{equation}\label{eoWs1}
					\begin{split}
					\|\widetilde{w}_{1}\|_{L^{\infty}(0,T;H^{2}(\Omega))}\leqslant K(B_{3}),
					\end{split}
					\end{equation}
				    \begin{equation}\label{eoWs2}
					\begin{split}
					\left\|\frac{1}{(1+\widetilde{\eta})}\widetilde{w}_{2}\right\|_{L^{\infty}(0,T;H^{2}(\Omega))}\leqslant K(B_{3},B_{4}),\qquad(\mbox{using}\,\eqref{eog1*})
					\end{split}
					\end{equation}
					and
					\begin{equation}\label{eoWs3}
					\begin{split}
					&\left\|\frac{1}{(1+\widetilde{\eta})}\widetilde{w}_{1}z\widetilde{\eta}_{x}\right\|_{L^{\infty}(0,T;H^{2}(\Omega))}\\[1.mm]
					&\leqslant K(B_{4})(\|\widetilde{\bf w}\|_{L^{\infty}(0,T;{\bf H}^{2}(\Omega))}\|\widetilde{\eta}_{x}\|_{L^{\infty}(0,T;H^{2}(\Gamma_{s}))})\\[1.mm]
					&\qquad\qquad\qquad\qquad\qquad\qquad\qquad\qquad(\mbox{using}\,\eqref{eog1*}\,\mbox{and Lemma}\,\ref{lfpss})\\[1.mm]
					&\leqslant T^{1/2}K(B_{4})(\|\widetilde{\bf w}\|_{L^{\infty}(0,T;{\bf H}^{2}(\Omega))}\|\widetilde{\eta}_{xt}\|_{L^{2}(0,T;H^{2}(\Gamma_{s}))})\\[1.mm]
					&\qquad\qquad\qquad\qquad\qquad\qquad\qquad\qquad(\mbox{using}\,\eqref{eoLit}\,\mbox{with}\,\psi=\widetilde{\eta}_{x}\,\mbox{and}\,\widetilde{\eta}_{x}(.,0)=0)\\[1.mm]
					& \leqslant T^{1/2}K(B_{3},B_{4}).
					\end{split}
					\end{equation}
					Combine \eqref{eoWs1}, \eqref{eoWs2} and \eqref{eoWs3} to prove \eqref{esW}$(ii).$\\[1.mm]
					(iii) From the definition of ${\mathscr{C}}_{T}(B_{1},B_{2},B_{3},B_{4})$ we know that $\widetilde{\sigma}_{t}$ is in $L^{\infty}(0,T; H^{1}(\Omega))$ and $\nabla\widetilde{\sigma}$ belongs to $L^{\infty}(0,T; {\bf H}^{1}(\Omega)).$ Hence one uses the continuous embedding $H^{1}(\Omega)\hookrightarrow L^{3}(\Omega)$ to obtain that the embedding from $L^{\infty}(0,T;H^{1}(\Omega))\hookrightarrow L^{2}(0,T;L^{3}(\Omega))$ has a norm of size $\sqrt{T}.$ We then easily derive \eqref{esW}$(iii).$
					\end{proof}
			 	\subsubsection{Choices of $B_{1},$ $B_{2},$ $B_{3}$ and $B_{4}$}
			 		Now we will choose the constants $B_{i}\geqslant B^{*}_{0}$ ($\leqslant i\leqslant 4$) such that for a small enough time $0<T\leqslant T^{*}_{0}(B_{1},B_{2},B_{3},B_{4}),$ $L$ maps $\mathscr{C}_{T}(B_{1},B_{2},B_{3},B_{4})$ into itself.
			 	\begin{lem}\label{inclusion}
			 	Let $B^{*}_{0}$ and $T^{*}_{0}$ are as in Lemma \ref{nonempty}. There exist constants $B_{i}\geqslant B^{*}_{0}$ ($1\leqslant i\leqslant 4$) and a time $T^{*}(B_{1},B_{2},B_{3},B_{4})$ satisfying $0<T^{*}(B_{1},B_{2},B_{3},B_{4})\leqslant T^{*}_{0}(B_{1},B_{2},B_{3},B_{4})$ such that for all $0<T\leqslant T^{*}(B_{1},B_{2},B_{3},B_{4}),$ $L$ maps $\mathscr{C}_{T}(B_{1},B_{2},B_{3},B_{4})$ into itself.
			 		\end{lem} 
			 		\begin{proof}
			 			In the following we will fix $B_{i}$ ($1\leqslant i\leqslant 4$) in a hierarchical order. We use the constants $B^{*}_{0}$ (Lemma \ref{nonempty}), $c_{4}$ (Theorem \ref{t23}), $K_{6}$ (Lemma \ref{eog3}), $c_{1}$ (Theorem \ref{t21}), $K_{4}$ (Lemma \ref{eog2}), $K_{5}$ (Lemma \ref{eog2}), $c_{3}$ (Theorem \ref{t22}), $K_{9}$ (Lemma \ref{eoWs}) and $K_{2}$ (Lemma \ref{eog1}). First we set $B_{1}$ and $B_{4}$ as follows
			 			\begin{equation}\label{chooseBi1}
			 			\left\{ \begin{split}
			 			& B_{1}=\mbox{max}\{2(\|\sigma_{0}\|_{H^{2}(\Omega)}+1),B^{*}_{0}\},\\
			 			& B_{4}=\mbox{max}\{c_{4}(\|\eta_{1}\|_{H^{3}(\Gamma_{s})}+\|G_{3}\mid_{t=0}\|_{ H^{1}(\Gamma_{s})}+K_{6}\|(\rho_{0},{\bf u}_{0})\|_{H^{2}(\Omega)\times {\bf H}^{2}(\Omega)}+1),B^{*}_{0}\}.
			 			\end{split}\right.
			 			\end{equation}
			 			Now using $B_{1}$ and $B_{4}$ we choose $B_{2}$ and $B_{3}$ in the following order.
			 			\begin{equation}\label{chooseBi2}
			 			 \begin{split}
			 			& B_{3}=\mbox{max}\{c_{1}(2+K_{5}\|G_{2}\mid_{t=0}\|_{{\bf L}^{2}(\Omega)}+4(1+K_{4}(B_{1},B_{4})\\
			 			&\qquad\qquad\qquad\qquad\qquad+\left\|\frac{G_{2}\mid_{t=0}-(-\mu \Delta -(\mu+\mu{'})\nabla\mathrm{div})({\bf u}_{0}-z\eta_{1}\vec{e_{2}})}{\rho_{0}}\right\|_{{\bf H}^{1}(\Omega)})),B^{*}_{0}\},
			 			\end{split}
			 			\end{equation}
			 			and 
			 			\begin{equation}\label{chooseBi3}
			 			\begin{split}
			 			B_{2}=\mbox{max}\{c_{3}K_{9}(B_{3},B_{4})\|\sigma_{0}\|_{H^{2}(\Omega)}+K_{2}\|\rho_{0}\mbox{div}({\bf u}_{0})\|_{ H^{1}(\Omega)}+1,B^{*}_{0}\}.
			 			\end{split}
			 			\end{equation}
			    In the rest of the proof we verify that with the choices \eqref{chooseBi1}, \eqref{chooseBi2} and \eqref{chooseBi3} of $B_{i}$ ($\forall\,1\leqslant i\leqslant 4$), there exists a time $T^{*}(B_{1},B_{2},B_{3},B_{4})$ such that for all $0<T\leqslant T^{*}(B_{1},B_{2},B_{3},B_{4}),$ $L$ maps $\mathscr{C}_{T}(B_{1},B_{2},B_{3},B_{4})$ into itself.\\
			     Let $(\widetilde{\sigma},\widetilde{\bf w},\widetilde{\eta})\in \mathscr{C}_{T}(B_{1},B_{2},B_{3},B_{4})$ and $L(\widetilde{\sigma},\widetilde{\bf w},\widetilde{\eta})=(\sigma,{\bf w},\eta).$ From Theorem \ref{t21}, Theorem \ref{t22} and Theorem \ref{t23} we know that $({\sigma},{\bf w},\eta)$ satisfies the following inequalities with
			     $$(G_{1},G_{2},G_{3})=(G_{1}(\widetilde{\sigma},\widetilde{\bf w},\widetilde{\eta}),G_{2}(\widetilde{\sigma},\widetilde{\bf w},\widetilde{\eta}),G_{3}(\widetilde{\sigma},\widetilde{\bf w},\widetilde{\eta})).$$
			 	\begin{equation}\label{esswe}
			 \left\{	\begin{split}
			 		&\|{\bf w} \|_{L^{\infty}(0,T;{\bf H}^{2}(\Omega))}+\|{\bf w} \|_{L^{2}(0,T;{\bf H}^{3}(\Omega))}+\|{\bf w}_{t} \|_{L^{\infty}(0,T;{\bf H}^{1}(\Omega))}+\|{\bf w}_{t} \|_{L^{2}(0,T;{\bf H}^{2}(\Omega))}\\
			 		&+\|{\bf w}_{tt}\|_{L^{2}(0,T;{\bf L}^{2}(\Omega))}
			 		\leqslant c_{1}\{\|G_{2}\|_{L^{2}(0,T;{\bf H}^{1}(\Omega))} +\|G_{2} \|_{L^{\infty}(0,T;{\bf L}^{2}(\Omega))}+\big(\|G_{2,t} \|_{L^{2}(0,T;{\bf L}^{2}(\Omega))}\\
			 		&+\left\|\frac{G_{2}\mid_{t=0}-(-\mu \Delta -(\mu+\mu{'})\nabla\mathrm{div})({\bf u}_{0}-z\eta_{1}\vec{e_{2}})}{\rho_{0}}\right\|_{{\bf H}^{1}(\Omega)}\big)
			 		\cdot (1+\|\widetilde{\sigma}_{t}\|_{L^{2}(0,T;L^{3}(\Omega))}\\
			 		&+\|\nabla\widetilde{\sigma}\|_{L^{2}(0,T;L^{3}(\Omega))})\mathrm{exp}(c_{1}\|\widetilde{\sigma}_{t}\|^{2}_{L^{2}(0,T;L^{3}(\Omega))})\},
			 	\end{split}\right.
			 	\end{equation}
			 	\begin{equation}\label{esswe1}
			 \left\{	\begin{split}
			 	\|\sigma\|_{L^{\infty}(0,T;H^{2}(\Omega))}&\leqslant(\|\sigma_{0}\|_{H^{2}(\Omega)}+c_{2}\| G_{1}\|_{L^{1}(0,T;H^{2}(\Omega))})\mathrm{exp}(c_{2}\|\widetilde {W}\|_{L^{1}(0,T;{\bf H}^{3}(\Omega))}),\\[1.mm]
			 		\|\sigma_{t}\|_{L^{\infty}(0,T; H^{1}(\Omega))}&\leqslant c_{3}\|\widetilde{W}\|_{L^{\infty}(0,T;{\bf H}^{2}(\Omega))}[(\|\sigma_{0}\|_{H^{2}(\Omega)}+c_{2}\| G_{1}\|_{L^{1}(0,T;H^{2}(\Omega))})
			 		\vspace{1.mm}\\
			 		&\cdot\mathrm{exp}(c_{2}\|\widetilde {W}\|_{L^{1}(0,T;{\bf H}^{3}(\Omega))})]+\| G_{1}\|_{L^{\infty}(0,T; H^{1}(\Omega))},
			 	\end{split}\right.
			 	\end{equation}
			 	and
			 	\begin{equation}\label{esswe3}
			 	\left\{\begin{split}
			 		\|\eta\|_ {L^{\infty}(0,T;H^{9/2}(\Gamma_{s}))} &+\|\eta_{t}\|_{L^{2}(0,T;H^{4}(\Gamma_{s}))}+\|\eta_{t}\|_{L^{\infty}(0,T;H^{3}(\Gamma_{s}))}+\|\eta_{tt}\|_{ L^{2}(0,T;H^{2}(\Gamma_{s}))}\\[1.mm]
			 		&+\|\eta_{tt}\|_{L^{\infty}([0,T]; H^{1}(\Gamma_{s}))}+\|\eta_{ttt}\|_{L^{2}(0,T;L^{2}(\Gamma_{s}))}\leqslant c_{4}\big(\|\eta_{1}\|_{H^{3}(\Gamma_{s})}
			 		\\[1.mm]
			 		&+\|G_{3}\mid_{t=0}\|_{ H^{1}(\Gamma_{s})}+\|G_{3}\|_{L^{\infty}(0,T;H^{1/2}(\Gamma_{s}))}	+\|{G}_{3,t}\|_{ L^{2}(0,T;{L}^{2}(\Gamma_{s}))}\big).
			 	\end{split}\right.
			 	\end{equation}
			 	$(i)$ Using the estimate \eqref{eg3} on $G_{3}(\widetilde{\sigma},\widetilde{\bf w},\widetilde{\eta})$ in \eqref{esswe3} we obtain:
			 	\begin{equation}\label{esswe3*}
			 	\left\{\begin{split}
			     &\|\eta\|_ {L^{\infty}(0,T;H^{9/2}(\Gamma_{s}))} +\|\eta_{t}\|_{L^{2}(0,T;H^{4}(\Gamma_{s}))}+\|\eta_{t}\|_{L^{\infty}(0,T;H^{3}(\Gamma_{s}))}+\|\eta_{tt}\|_{ L^{2}(0,T;H^{2}(\Gamma_{s}))}\\[1.mm]
			     &+\|\eta_{tt}\|_{L^{\infty}([0,T]; H^{1}(\Gamma_{s}))}+\|\eta_{ttt}\|_{L^{2}(0,T;L^{2}(\Gamma_{s}))}\leqslant c_{4}(\|\eta_{1}\|_{H^{3}(\Gamma_{s})}
			     +\|G_{3}\mid_{t=0}\|_{ H^{1}(\Gamma_{s})}\\
			     &+K_{6}\|(\rho_{0},{\bf u}_{0})\|_{H^{2}(\Omega)\times {\bf H}^{2}(\Omega)}+K_{7}(B_{1},B_{2},B_{3},B_{4})T^{1/2}
			     +T^{1/6}K_{7}(B_{1},B_{2},B_{3},B_{4}).
			 	\end{split}\right.
			 	\end{equation}
			 	Now choose $T^{*}_{1}=T^{*}_{1}(B_{1},B_{2},B_{3},B_{4})(\leqslant T^{*}_{0}(B_{1},B_{2},B_{3},B_{4}))$ small enough positive such that 
			 	\begin{equation}\label{T1*}
			 	\begin{split}
			    K_{7}(B_{1},B_{2},B_{3},B_{4})(T^{*}_{1})^{1/2}+K_{7}(B_{1},B_{2},B_{3},B_{4})(T^{*}_{1})^{1/6}<1.
			 	\end{split}
			 	\end{equation}
			 	In view of the choice of $B_{4}$ (see \eqref{chooseBi1}) and \eqref{T1*}, for all $0<T\leqslant T^{*}_{1}$ one verifies that
			 	\begin{equation}\label{esswe3**}
			 	\left\{\begin{split}
			 	&\|\eta\|_ {L^{\infty}(0,T;H^{9/2}(\Gamma_{s}))} +\|\eta_{t}\|_{L^{2}(0,T;H^{4}(\Gamma_{s}))}+\|\eta_{t}\|_{L^{\infty}(0,T;H^{3}(\Gamma_{s}))}+\|\eta_{tt}\|_{ L^{2}(0,T;H^{2}(\Gamma_{s}))}\\[1.mm]
			 	&+\|\eta_{tt}\|_{L^{\infty}([0,T]; H^{1}(\Gamma_{s}))}+\|\eta_{ttt}\|_{L^{2}(0,T;L^{2}(\Gamma_{s}))}\leqslant B_{4}.
			 	\end{split}\right.
			 	\end{equation}
			 	\medskip
			 	$(ii)$ Using the estimates  \eqref{eg1}$(i)$ on $G_{1}(\widetilde{\sigma},\widetilde{\bf w},\widetilde{\eta})$ and \eqref{esW}$(i)$ on $\widetilde{W}({\widetilde{\bf w},\widetilde{\eta}})$ in \eqref{esswe1}$_{1}$ furnish
			 	\begin{equation}\label{B1}
			 	\begin{array}{l}
			 		\|\sigma\|_{L^{\infty}(0,T;H^{2}(\Omega))}\leqslant(\|\sigma_{0}\|_{H^{2}(\Omega)}+c_{2}K_{1}(B_{1},B_{2},B_{3},B_{4})T^{1/2})\mathrm{exp}(c_{2}K_{8}(B_{1},B_{2},B_{3},B_{4})T^{1/2}).
			 	\end{array}
			 	\end{equation}
			 	Choose $T^{*}_{2}=T^{*}_{2}(B_{1},B_{2},B_{3},B_{4})(\leqslant T^{*}_{1})$ small enough positive such that
			 	\begin{equation}\label{T2*}
			 	\begin{array}{l}
			 	c_{2}K_{1}(B_{1},B_{2},B_{3},B_{4})(T^{*}_{2})^{1/2}<1\quad\mbox{and}\quad \mathrm{exp}(c_{2}K_{8}(B_{1},B_{2},B_{3},B_{4})(T^{*}_{2})^{1/2})<2.
			 	\end{array}
			 	\end{equation}
			  In view of the choice of $B_{1}$ (see \eqref{chooseBi1}) and \eqref{T2*}, for all $0<T\leqslant T^{*}_{2}$ the following holds
			  \begin{equation}\label{B1*}
			  \begin{array}{l}
			  \|\sigma\|_{L^{\infty}(0,T;H^{2}(\Omega))}\leqslant B_{1}.
			  \end{array}
			  \end{equation}
			 	\medskip
			 	$(iii)$  Using the estimates \eqref{eg2} on $G_{2}(\widetilde{\sigma},\widetilde{\bf w},\widetilde{\eta})$ and \eqref{esW}$(iii)$ on $\widetilde{\sigma}$ in \eqref{esswe} to obtain
			 		\begin{equation}\label{esswe*}
			 		\left\{	\begin{split}
			 		&\|{\bf w} \|_{L^{\infty}(0,T;{\bf H}^{2}(\Omega))}+\|{\bf w} \|_{L^{2}(0,T;{\bf H}^{3}(\Omega))}+\|{\bf w}_{t} \|_{L^{\infty}(0,T;{\bf H}^{1}(\Omega))}+\|{\bf w}_{t} \|_{L^{2}(0,T;{\bf H}^{2}(\Omega))}+\|{\bf w}_{tt}\|_{L^{2}(0,T;{\bf L}^{2}(\Omega))}
			 		\vspace{1.mm}\\
			 		&\leqslant c_{1}\{K_{5}\|G_{2}\mid_{t=0}\|_{{\bf L}^{2}(\Omega)}+ 2K_{3}(B_{1},B_{2},B_{3},B_{4})T^{1/2}\\
			 		&+\left(K_{3}(B_{1},B_{2},B_{3},B_{4})T^{1/2}+K_{4}(B_{1},B_{4})
			 		+\left\|\frac{G_{2}\mid_{t=0}-(-\mu \Delta -(\mu+\mu{'})\nabla\mathrm{div})({\bf u}_{0}-z\eta_{1}\vec{e_{2}})}{\rho_{0}}\right\|_{{\bf H}^{1}(\Omega)}\right)\\
			 		&\cdot (1+K_{8}(B_{1},B_{2},B_{3},B_{4})T^{1/2})\mathrm{exp}(c_{1}K_{8}^{2}(B_{1},B_{2},B_{3},B_{4})T)\}.
			 		\end{split}\right.
			 		\end{equation}
			 	Choose $T^{*}_{3}=T^{*}_{3}(B_{1},B_{2},B_{3},B_{4})(\leqslant T^{*}_{2}(B_{1},B_{2},B_{3},B_{4}))$ small enough positive such that
			 	\begin{equation}\label{T3*}
			 	\begin{array}{ll}
			 	& K_{3}(B_{1},B_{2},B_{3},B_{4})(T^{*}_{3})^{1/2}<1,\\
			 	& \mbox{and}\,\,(1+K_{8}(B_{1},B_{2},B_{3},B_{4})(T^{*}_{3})^{1/2})\mathrm{exp}(c_{1}K_{8}^{2}(B_{1},B_{2},B_{3},B_{4})T^{*}_{3})<4.
			 	\end{array}
			 	\end{equation}
			 	In view of the choice of $B_{3}$ (see \eqref{chooseBi2}) and \eqref{T3*}, for all $0<T\leqslant T^{*}_{3}$ we have
			 	\begin{equation}\label{esswe**}
			 		\begin{split}
			 	\|{\bf w} \|_{L^{\infty}(0,T;{\bf H}^{2}(\Omega))}+\|{\bf w} \|_{L^{2}(0,T;{\bf H}^{3}(\Omega))}&+\|{\bf w}_{t} \|_{L^{\infty}(0,T;{\bf H}^{1}(\Omega))}\\
			 	&+\|{\bf w}_{t} \|_{L^{2}(0,T;{\bf H}^{2}(\Omega))}+\|{\bf w}_{tt}\|_{L^{2}(0,T;{\bf L}^{2}(\Omega))}<B_{3}.
			 	\end{split}
			 	\end{equation}
			 	%Using interpolation one further obtains ${\bf w}\in L^{\infty}(0,T;{\bf H}^{5/2}(\Omega)).$\\
			 	\medskip
			 	$(iv)$ Using the estimates \eqref{eg1} on $G_{1}(\widetilde{\sigma},\widetilde{\bf w},\widetilde{\eta})$ and \eqref{esW}$(i)$-\eqref{esW}$(ii)$ on $\widetilde{W}(\widetilde{\bf w},\widetilde{\eta})$ in \eqref{esswe1}$_{2}$ furnish
			 	\begin{equation}\label{B2}
			 	\begin{split}
			 		&\|\sigma_{t}\|_{L^{\infty}(0,T; H^{1}(\Omega))}\leqslant c_{3}(K_{9}(B_{3},B_{4})+K_{8}(B_{1},B_{2},B_{3},B_{4})T^{1/2})[(\|\sigma_{0}\|_{H^{2}(\Omega)}\\
			 		&\quad+c_{2}K_{1}(B_{1},B_{2},B_{3},B_{4})T^{1/2})
			 		\cdot\mathrm{exp}(c_{2}K_{8}(B_{1},B_{2},B_{3},B_{4})T^{1/2})]+K_{2}\|\rho_{0}\mathrm{div}({\bf u}_{0})\|_{ H^{1}(\Omega)}\\
			 		&\qquad+K_{1}(B_{1},B_{2},B_{3},B_{4})T^{1/2}.
			 	\end{split}
			 	\end{equation}
			 	Choose $T^{*}_{4}=T^{*}_{4}(B_{1},B_{2},B_{3},B_{4})(\leqslant T^{*}_{3})$ small enough positive such that
			 	\begin{equation}\label{T4*}
			 	\begin{split}
			 	 & c_{2}c_{3}K_{9}(B_{3},B_{4}) K_{1}(B_{1},B_{2},B_{3},B_{4})(T^{*}_{4})^{1/2}\cdot{\exp}(c_{2}K_{8}(B_{1},B_{2},B_{3},B_{4})(T^{*}_{4})^{1/2})\\
			 	 &+K_{8}(B_{1},B_{2},B_{3},B_{4})(T^{*}_{4})^{1/2}
			 	[(\|\sigma_{0}\|_{H^{2}(\Omega)}+c_{2}K_{1}(B_{1},B_{2},B_{3},B_{4})(T^{*}_{4})^{1/2})\\
			 	&\cdot\mathrm{exp}(c_{2}K_{8}(B_{1},B_{2},B_{3},B_{4})(T^{*}_{4})^{1/2})]
			 	+K_{1}(B_{1},B_{2},B_{3},B_{4})(T^{*}_{4})^{1/2}<1.
			 	\end{split}
			 	\end{equation}
			 	In view of the choice of $B_{2}$ (see \eqref{chooseBi3}) and \eqref{T4*}, we check that for all $0<T\leqslant T^{*}_{4}$ the following holds
			 	\begin{equation}\label{B2*}
			 	\begin{split}
			 	\|\sigma_{t}\|_{L^{\infty}(0,T; H^{1}(\Omega))}< B_{2}.
			 	\end{split}
			 	\end{equation}
			 	%\medskip
			 	Hence with the choices \eqref{chooseBi1}, \eqref{chooseBi2} and \eqref{chooseBi3} of the constants $B_{i}$ ($1\leqslant i\leqslant 4$), $(\sigma,{\bf w},\eta)$ satisfies the estimates \eqref{B1*}, \eqref{B2*}, \eqref{esswe**} and \eqref{esswe3**} respectively for all $0<T\leqslant T^{*}_{4}.$ We can also use similar kind of interpolation arguments as used in \eqref{3.1} to show that there exists a $T^{*}_{5}=T^{*}_{5}(B_{1},B_{2},B_{3},B_{4})$ ($\leqslant T^{*}_{4}$), positive, such that for all $0<T\leqslant T^{*}_{5},$ 
			 	\begin{equation}\label{3.26}
			 	\begin{array}{ll}
			 	1+\eta(x,t)\geqslant\delta_{0}>0,&\quad \mbox{on}\quad \Sigma^{s}_{T}\\[1.mm]
			 	\displaystyle \frac{m}{2}\leqslant \sigma(x,z,t)+\overline{\rho}\leqslant 2M,&\quad \mbox{in}\quad Q_{T}.
			 	\end{array}
			 	\end{equation}
			 	Again it follows from the equation \eqref{3.2}$_{2}$ that  ${\bf w}_{t}(0)$ satisfies the condition \eqref{twt0}. Similarly one uses \eqref{3.2}$_{6}$ to show that $\eta_{tt}(\cdot,0)$ satisfies \eqref{etatt0}. Now we set
			 	$$T^{*}=T^{*}(B_{1},B_{2},B_{3},B_{4})=T^{*}_{5}.$$
			 	Hence if $B_{i}$ ($\forall\,1\leqslant i\leqslant 4$) is chosen as in \eqref{chooseBi1}, \eqref{chooseBi2} and \eqref{chooseBi3} and $0<T\leqslant T^{*},$ $(\sigma,{\bf w},\eta)$ satisfies all the conditions \eqref{normbnd}-\eqref{inbndc}, guaranteeing $(\sigma,{\bf w},\eta)\in \mathscr{C}_{T}(B_{1},B_{2},B_{3},B_{4}).$\\ 
			    This concludes the proof of Lemma \ref{inclusion}.
			 	\end{proof}
			 	We fix the choice of $B_{i}$ ($\forall\, 1\leqslant i\leqslant 4$) and $T=T^{*}(B_{1},B_{2},B_{3},B_{4})$ as in Lemma \ref{inclusion}. Hence in the following we will simply use the notations
			 	\begin{equation}\label{TCT}
			 	\begin{array}{l}
			 	T=T^{*}\,\,\mbox{and}\,\,\mathscr{C}_{T}=\mathscr{C}_{T}(B_{1},B_{2},B_{3},B_{4}).
			 	\end{array}
			 	\end{equation}
			 	\subsection{Compactness and continuity}\label{comcon}
			      Let us observe that $\mathscr{C}_{T}$ is a convex, bounded subset of the space
			 	$$\mathcal{X}=\{(\sigma,{\bf w},\eta)\in C^{0}([0,T],{ H}^{1}(\Omega))\times C^{0}([0,T];{\bf H}^{1}(\Omega))\times C^{1}([0,T]; H^{1}(\Gamma_{s}))\cap C^{0}([0,T];H^{2}(\Gamma_{s}))\},$$
			 	endowed with the topology induced by the norm
			 	$$\|(\sigma,{\bf w},\eta)\|_{\mathcal{X}}=\sup_{t\in[0,T]}(\|\sigma(t)\|_{H^{1}(\Omega)}+\|{\bf w}(t)\|_{{\bf H}^{1}(\Omega)}+\|\eta(t)\|_{H^{2}(\Gamma_{s})}+\|\eta_{t}(t)\|_{H^{1}(\Gamma_{s})}).$$
			 	\begin{lem}\label{com}
			    Let $\mathscr{C}_{T}$ be the set as introduced in \eqref{TCT}. The set $\mathscr{C}_{T},$ when endowed with the topology of $\mathcal{X},$ is compact in $\mathcal{X}.$
			 	\end{lem}
			 	\begin{proof}
			 	We claim that the set $\mathscr{C}_{T}$ is closed in $\mathcal{X}.$ \\
			 	Assume that a sequence $(\widetilde{\sigma}_{n},\widetilde{\bf w}_{n},\widetilde\eta_{n})\in \mathscr{C}_{T}$ and that $(\widetilde{\sigma}_{n},\widetilde{\bf w}_{n},\widetilde\eta_{n})\rightarrow (\overline{\sigma},\overline{\bf w},\overline{\eta})$ in $\mathcal{X}.$ Now $\widetilde\eta_{n}\rightarrow \overline{\eta}$ in $C^{1}([0,T]; H^{1}(\Gamma_{s}))$ implies that $\widetilde\eta_{n,t}\rightarrow \overline{\eta}_{t}$ in $C^{0}([0,T]; H^{1}(\Gamma_{s})),$ $\widetilde\eta_{n,tt}\rightarrow \overline{\eta}_{tt}$ and $\widetilde\eta_{n,ttt}\rightarrow \overline{\eta}_{ttt}$  in $\mathcal{D}'(0,T;L^{2}(\Gamma_{s}))$ in particular, where $\mathcal{D}'(0,T;L^{2}(\Gamma_{s}))$ denotes the space of distributions on $(0,T)$ with values in $L^{2}(\Gamma_{s}).$ 
			 	\newline
			 	We recall the norm bounds over $\eta$ in the set $\mathscr{C}_{T}.$ Hence we have up to a subsequence (still denoted by $\widetilde\eta_{n}$) that $\widetilde\eta_{n}\rightarrow\overline{\eta}$ weak* in $L^{\infty}(0,T;H^{9/2}(\Gamma_{s})),$ $\widetilde\eta_{n,t}\rightarrow \overline{\eta}_{t}$ weakly in $L^{2}(0,T;H^{4}(\Gamma_{s}))$ and weak* in $L^{\infty}(0,T;H^{3}(\Gamma_{s})),$ $\widetilde\eta_{n,tt}\rightarrow \overline{\eta}_{tt}$ weakly in $L^{2}(0,T;H^{2}(\Gamma_{s}))$ and weak* in $L^{\infty}(0,T; H^{1}(\Gamma_{s})),$ $\widetilde\eta_{n,ttt}\rightarrow \overline{\eta}_{ttt}$ weakly in $L^{2}(0,T;L^{2}(\Gamma_{s})).$
			 	Also by the lower semi-continuity of the norms with respect to the above weak type convergences we get
			 	\begin{equation}\label{3.27}
			 	\begin{split}
			 	&\|\overline\eta\|_{L^{\infty}(0,T;H^{9/2}(\Gamma_{s}))}+	\|\overline\eta_{t}\|_{L^{\infty}(0,T;H^{3}(\Gamma_{s}))}+	\|\overline\eta_{t}\|_{L^{2}(0,T;H^{4}(\Gamma_{s}))}	+\|\overline\eta_{tt}\|_{L^{\infty}(0,T; H^{1}(\Gamma_{s}))}\\
			 	&+	\|\overline\eta_{tt}\|_{L^{2}(0,T;H^{2}(\Gamma_{s}))}
			 	+\|\overline\eta_{ttt}\|_{L^{2}(0,T;L^{2}(\Gamma_{s}))}\leqslant  B_{4}.
			 	\end{split}
			 	\end{equation}
			 As $\widetilde\eta_{n}\rightarrow \overline{\eta}$ in $C^{1}([0,T]; H^{1}(\Gamma_{s}))$ and $\widetilde\eta_{n,t}\rightarrow \overline{\eta}_{t}$ in $C^{0}([0,T]; H^{1}(\Gamma_{s})),$ hence 
			 	\begin{equation}\label{3.28}
			 	\begin{array}{l}
			 	\overline\eta(\cdot,0)=0\quad\mbox{and}\quad \overline{\eta}_{t}(0)=\eta_{1}.
			 	\end{array}
			 	\end{equation}
			 	The uniform bounds of $\|\widetilde{\eta}_{n,tt}\|_{L^{\infty}(0,T;H^{1}(\Gamma_{s}))}$ and $\|\widetilde{\eta}_{n,ttt}\|_{L^{2}(0,T;L^{2}(\Gamma_{s}))}$ and Aubin Lions lemma (\cite{aubin}) furnish that up to a subsequence (still denoted by $\widetilde{\eta}_{n}$), $\widetilde{\eta}_{n,tt}$ strongly converges to $\overline{\eta}_{tt}$ in $C^{0}([0,T];L^{2}(\Gamma_{s})).$ Hence
			 	\begin{equation}\label{inoeta}
			 	\begin{array}{l}
			 	\overline{\eta}_{tt}(\cdot,0)=\delta\eta_{1,xx}-(\mu+2\mu')(u_{0})_{2,z}+P(\rho_{0}).
			 	\end{array}
			 	\end{equation}
			 Similar arguments (used to show \eqref{3.27}) can be used to show that
			 \begin{equation}\label{3.29}
			 \begin{array}{ll}
			 \|\overline{\bf w}\|_{L^{\infty}(0,T;{\bf H}^{2}(\Omega))}+\|\overline{\bf w}\|_{{L^{2}}(0,T;{\bf H}^{3}(\Omega))}
			 +\|\overline{\bf w}_{t}\|_{L^{\infty}(0,T;{\bf H}^{1}(\Omega))}
			 &+\|\overline{\bf w}_{t}\|_{L^{2}(0,T;{\bf H}^{2}(\Omega))}\\[1.mm]
			  &+\|\overline{\bf w}_{tt}\|_{L^{2}(0,T;{\bf L}^{2}(\Omega))}\leqslant B_{3},
			 \end{array}
			 \end{equation} 	
			 \begin{equation}\label{3.30}
			 \begin{array}{l}
			 \|\overline{\sigma}\|_{L^{\infty}(0,T;H^{2}(\Omega))}\leqslant B_{1},\quad\|\overline\sigma_{t}\|_{L^{\infty}(0,T; H^{1}(\Omega))}\leqslant B_{2}.
			 \end{array}
			 \end{equation}
              Since $\widetilde{\eta}_{n}$ converges to $\overline{\eta}$ in $L^{\infty}(\Sigma^{s}_{T})$ (follows from the continuous embedding $H^{2}(\Gamma_{s})\hookrightarrow L^{\infty}(\Gamma_{s})$), one has the following (as $\widetilde{\eta}_{n}$ satisfies \eqref{1etad0})
			 	\begin{equation}\label{3.32}
			 	\begin{array}{l}
			 	1+\overline\eta(x,t)\geqslant\delta_{0}>0\quad \mbox{on}\quad \Sigma^{s}_{T}.
			 	%\frac{m}{2}\leqslant \overline\sigma(x,z,t)+\overline{\rho}\leqslant 2M&\quad \mbox{in}\quad Q_{T},
			 	\end{array}
			 	\end{equation}
			 	Observe that the weak* convergence of $\widetilde{\sigma}_{n}$ to $\overline{\sigma}$ in $L^{\infty}(0,T;L^{2}(\Omega))$ is enough to conclude that (since $\widetilde{\sigma}_{n}$ satisfies \eqref{tsgmM})
			 	\begin{equation}\label{bndsigma}
			 	\begin{array}{l}
			 	\displaystyle \frac{m}{2}\leqslant \overline\sigma(x,z,t)+\overline{\rho}\leqslant 2M\quad \mbox{in}\quad Q_{T}.
			 	\end{array}
			 	\end{equation}
			 	Using the strong convergence of $(\widetilde{\sigma}_{n},\widetilde{\bf w}_{n},\widetilde{\eta}_{n})$ to $(\overline{\sigma},\overline{\bf w},\overline{\eta})$ in $\mathcal{X}$ furnishes
			 	 \begin{equation}\label{3.31}
			 	 \begin{array}{ll}
			 	 \overline{\bf w}(\cdot,0)=({\bf u}_{0}-z\eta_{1}\vec{e_{2}})&\quad \mbox{in}\quad \Omega,\\
			 	 \overline{\sigma}(\cdot,0)=\sigma_{0}&\quad\mbox{in}\quad \Omega.
			 	 \end{array}
			 	 \end{equation}
			 	 Now we can use the uniform bounds of $\|\widetilde{\bf w}_{n,t}\|_{L^{\infty}(0,T;{\bf H}^{1}(\Omega))}$ and $\|\widetilde{\bf w}_{n,tt}\|_{L^{2}(0,T;{\bf L}^{2}(\Omega))}$ and the Aubin Lions lemma to have the convergence $\widetilde{\bf w}_{n,t}\rightarrow \overline{\bf w}_{t}$ in $C^{0}([0,T];{\bf L}^{2}(\Omega)).$ Consequently
			 	 \begin{equation}\label{icbwt}
			 	 \begin{split}
			 	 \overline{\bf w}_{t}(.,0)=\frac{1}{\rho_{0}}(G_{2}^0-(-\mu \Delta -(\mu+\mu{'})\nabla\mathrm{div})({\bf u}_{0}-z\eta_{1}\vec{e}_{2})).
			 	 \end{split}
			 	 \end{equation}
			 	So combining \eqref{3.27}-\eqref{3.28}-\eqref{inoeta}-\eqref{3.29}-\eqref{3.30}-\eqref{3.32}-\eqref{bndsigma}-\eqref{3.31}-\eqref{icbwt} we conclude that the limit point $(\overline{\sigma},\overline{\bf w},\overline{\eta})\in \mathscr{C}_{T}$ and hence $\mathscr{C}_{T}$ is closed in $\mathcal{X}$.
			 	\newline
			 	Once again using Aubin Lions lemma  we get that $\mathscr{C}_{T}$ is a compact subset of $\mathcal{X}.$ 
			    \end{proof}
			 	Now to apply Schauder's fixed point theorem one only needs to prove that $L$ is continuous on $\mathscr{C}_{T}.$ 
			 	\begin{lem}\label{cont}
			 		Let $\mathscr{C}_{T}$ be the set in \eqref{TCT}. The map $L$ is continuous from $\mathscr{C}_{T}$ into itself for the topology of $\mathcal{X}.$
			 	\end{lem}
			 	\begin{proof}
			 	 Suppose that $(\widetilde{\sigma}_{n},\widetilde{\bf w}_{n},\widetilde{\eta}_{n})\in \mathscr{C}_{T},$ converges to $(\widetilde{\sigma},\widetilde{\bf w},\widetilde{\eta})$ strongly in $\mathcal{X}.$ Then, according to Lemma \ref{com}, $(\widetilde{\sigma},\widetilde{\bf w},\widetilde{\eta})\in \mathscr{C}_{T}.$ We thus set $(\widehat{\sigma}_{n},\widehat{\bf w}_{n},\widehat{\eta}_{n})= L(\widetilde{\sigma}_{n},\widetilde{\bf w}_{n},\widetilde{\eta}_{n}),$ $(\widehat{\sigma},\widehat{\bf w},\widehat{\eta})= L(\widetilde{\sigma},\widetilde{\bf w},\widetilde{\eta}).$ Our goal is to show that $(\widehat{\sigma}_{n},\widehat{\bf w}_{n},\widehat{\eta}_{n})$ strongly converges to $(\widehat{\sigma},\widehat{\bf w},\widehat{\eta})$ in $\mathcal{X}.$ Using that $(\widehat{\sigma}_{n},\widehat{\bf w}_{n},\widehat{\eta}_{n})$ belongs to $\mathscr{C}_{T}$ (see Lemma \ref{inclusion}) we get that there exists a triplet $(\overline{\sigma},\overline{\bf w},\overline{\eta})$ such that up to a subsequence 
			 	 \begin{equation}\label{weakconv}
			 	 \begin{array}{lll}
			 	 & \widehat{\sigma}_{n}\stackrel{\ast}{\rightharpoonup}\overline{\sigma}&\,\,\mbox{in}\,\, L^{\infty}(0,T;H^{2}(\Omega))\cap W^{1,\infty}(0,T;H^{1}(\Omega)),\\
			 	 & \widehat{\bf w}_{n}\rightharpoonup \overline{\bf w}&\,\,\mbox{in}\,\, L^{2}(0,T;{\bf H}^{3}(\Omega))\cap H^{1}(0,T;{\bf H}^{2}(\Omega))\cap H^{2}(0,T;{\bf L}^{2}(\Omega)),\\
			 	 & \widehat{\bf w}_{n}\stackrel{\ast}{\rightharpoonup} \overline{\bf w}&\,\,\mbox{in}\,\, L^{\infty}(0,T;{\bf H}^{2}(\Omega))\cap W^{1,\infty}(0,T;{\bf H}^{1}(\Omega)),\\
			 	 & \widehat{\eta}_{n}\rightharpoonup \overline{\eta}&\,\,\mbox{in}\,\, H^{1}(0,T;H^{4}(\Gamma_{s}))\cap H^{2}(0,T;H^{2}(\Gamma_{s}))\cap H^{3}(0,T;L^{2}(\Gamma_{s})),\\
			 	 & \widehat{\eta}_{n}\stackrel{\ast}{\rightharpoonup} \overline{\eta}&\,\,\mbox{in}\,\, L^{\infty}(0,T;H^{9/2}(\Gamma_{s}))\cap W^{1,\infty}(0,T;H^{3}(\Gamma_{s}))\cap W^{2,\infty}(0,T;H^{1}(\Gamma_{s})).
			 	 \end{array}
			 	 \end{equation}
			 	  The compactness result proved in Lemma \ref{com} provides the strong convergence in $\mathcal{X}$ i.e, up to a subsequence, $(\widehat{\sigma}_{n},\widehat{\bf w}_{n},\widehat{\eta}_{n})$ converges strongly in $\mathcal{X}$ to $(\overline{\sigma},\overline{\bf w},\overline{\eta}).$ It is clear that in order to prove that the map $L$ is continuous it is enough to show that $(\overline{\sigma},\overline{\bf w},\overline{\eta})=(\widehat{\sigma},\widehat{\bf w},\widehat{\eta}).$ This will be verified in the following steps.\\
			 	$(i)$ We first claim that $G_{2}(\widetilde{\sigma}_{n},\widetilde{\bf w}_{n},\widetilde{\eta}_{n})$ converges weakly to $G_{2}(\widetilde{\sigma},\widetilde{\bf w},\widetilde{\eta})$ in $L^{2}(0,T;{\bf L}^{2}(\Omega)).$\\
			 	 Since $(\widetilde{\sigma}_{n},\widetilde{\bf w}_{n},\widetilde{\eta}_{n})$ belongs to $\mathscr{C}_{T}$ and we have fixed $B_{i}$ (for all $1\leqslant i\leqslant 4$) and $T,$ one can use Lemma \ref{eog2} to show that $\|G_{2}(\widetilde{\sigma}_{n},\widetilde{\bf w}_{n},\widetilde{\eta}_{n})\|_{L^{2}(0,T;L^{2}(\Omega))}$ is uniformly bounded. Hence, to prove our claim it is enough to show that $G_{2}(\widetilde{\sigma}_{n},\widetilde{\bf w}_{n},\widetilde{\eta}_{n})$ converges to $G_{2}(\widetilde{\sigma},\widetilde{\bf w},\widetilde{\eta})$ in $\mathcal{D}'(Q_{T})$ ($\mathcal{D}'(Q_{T})$ is the space of distributions on $Q_{T}$).\\
			 	  %Let us consider the term $z\widetilde{\eta}_{n,tt}(\widetilde{\sigma}_{n}+\overline{\rho})\vec{e}_{2}.$
			 	%Observe that as $(\widetilde{\sigma}_{n}+\overline{\rho})$ strongly converges to $(\widetilde{\sigma}+\overline{\rho})$ in $C^{0}([0,T]; H^{1}(\Omega))$ and $\widetilde{\eta}_{n,tt}$ weakly converges to $\widetilde{\eta}_{tt}$ in $L^{2}(0,T;H^{2}(\Gamma_{s}))$ (from the bounds over the norms of $\widetilde{\eta}_{n}$) so $z\widetilde{\eta}_{n,tt}(\widetilde{\sigma}_{n}+\overline{\rho})\vec{e}_{2}$ converges to $z\widetilde{\eta}_{tt}(\widetilde{\sigma}+\overline{\rho})\vec{e}_{2}$ weakly in $L^{2}(0,T;{\bf L}^{2}(\Omega))$ and hence in particular in $\mathcal{D}'(Q_{T}).$ 
			 	%Again as $\|\overline{\eta}_{tt}\|\in L^{\infty}(0,T;L^{\infty}(\Omega))$ and $\widetilde{\sigma}_{n}$ strongly converges to $\widetilde{\sigma}$ in $L^{2}(0,T;L^{2})$ hence applying H\"{o}lder's inequality we get that the second term in the right hand side of \eqref{con1} converges to zero.
			 	Let us consider the term $\displaystyle\frac{\widetilde{\bf w}_{n,zz}z^{2}\widetilde{\eta}^{2}_{n,x}}{(1+\widetilde{\eta}_{n})}.$
			 	%\begin{equation}\label{con2}
			 	%\int_{Q_{T}} (\frac{\widetilde{\bf w}_{n,zz}z^{2}\widetilde{\eta}^{2}_{n,x}}{(1+\widetilde{\eta}_{n})}-\frac{\widetilde{\bf w}_{zz}z^{2}\widetilde{\eta}^{2}_{x}}{(1+\widetilde{\eta})})\phi
			 	%&=\int_{Q_{T}}(\frac{z^{2}(1+\widetilde{\eta})\widetilde{\bf w}_{n,zz}(\widetilde{\eta}^{2}_{n,x}-\widetilde{\eta}^{2}_{x})}{(1+\widetilde{\eta}_{n})(1+\widetilde{\eta})})\phi+\int_{Q_{T}}(\frac{z^{2}(1+\widetilde{\eta})(\widetilde{\bf w}_{n,zz}-\widetilde{\bf w}_{zz})\widetilde{\eta}^{2}_{x})}{(1+\widetilde{\eta}_{n})(1+\widetilde{\eta})})\phi\\[1.mm]
			 %	&+\int_{Q_{T}}(\frac{z^{2}(\widetilde{\eta}-\widetilde{\eta}_{n})\widetilde{\bf w}_{zz}\widetilde{\eta}^{2}_{x}}{(1+\widetilde{\eta}_{n})(1+\widetilde{\eta})})\phi.
			 %	\end{array}
			 %	\end{equation}
			  %It is easy to check that $(\widetilde{\eta}^{2}_{n,x}-\widetilde{\eta}^{2}_{x})$ strongly converges to zero in $L^{2}(0,T;L^{2}(\Omega))$ and the rest, appearing inside the integral sign is in the space $L^{2}(0,T;{\bf L}^{2}(\Omega))$ with the norm bounded by a constant independent of $n.$ Hence an use of H\"{o}lder's inequality at once yields that the integral converges to zero. Also a similar argument gives that the third term in the right side of \eqref{con2} tends to zero. The fact that $\widetilde{\bf w}_{n,zz}$ weakly converges to $\widetilde{\bf w}_{zz}$ in $L^{2}(0,T;{\bf L}^{2}(\Omega))$ can be used to show that the second member in the right of \eqref{con2} converges to zero.
			 	From the uniform norm bound over $\|\widetilde{\bf w}_{n,zz}\|_{L^{2}(0,T;{\bf H}^{1}(\Omega))}$ we get that $\widetilde{\bf w}_{n,zz}$ converges weakly in $L^{2}(0,T;{\bf H}^{1}(\Omega))$ to $\widetilde{\bf w}_{zz}.$ Since $\widetilde{\eta}_{n}$ strongly converges to $\widetilde{\eta}$ in $C^{0}([0,T];H^{2}(\Gamma_{s}))$ and both $\widetilde{\eta}_{n}$ and $\widetilde{\eta}$ satisfy \eqref{1etad0},   $\displaystyle\frac{1}{(1+\widetilde{\eta}_{n})}$ and $\widetilde{\eta}_{n,x}$ converge strongly to $\displaystyle\frac{1}{(1+\widetilde{\eta})}$ and $\widetilde{\eta}_{x}$ respectively in the spaces $C^{0}([0,T];H^{2}(\Gamma_{s}))$ and $C^{0}([0,T];H^{1}(\Gamma_{s})).$
			 	%\begin{equation}\label{diffG2}
			 	%\begin{split}
			 	%\displaystyle
			 	%\int\limits_{Q_{T}}\left(\frac{\widetilde{\bf w}_{n,zz}z^{2}\widetilde{\eta}^{2}_{n,x}}{(1+\widetilde{\eta}_{n})}-\frac{\widetilde{\bf w}_{zz}z^{2}\widetilde{\eta}^{2}_{x}}{(1+\widetilde{\eta})}\right)
			 	%=\frac{(\widetilde{\bf w}_{n,zz}-\widetilde{\bf w}_{zz})z^{2}\widetilde{\eta}^{2}_{n,x}(1+\widetilde{\eta}))}{(1+\widetilde{\eta}_{n})(1+\widetilde{\eta})}+\frac{(\widetilde{\eta}^{2}_{n,x}-\widetilde{\eta}^{2}_{x}z^{2}\widetilde{\bf w}_{zz}(1+\widetilde{\eta}))}{}
			 	%\end{split}
			 	%\end{equation}
			 	Hence one gets in particular the strong convergence of $\widetilde{\eta}^{2}_{n,x}$ to $\widetilde{\eta}^{2}_{x}$ in the space  $C^{0}([0,T];L^{2}(\Gamma_{s})).$ This implies that $\displaystyle\frac{\widetilde{\bf w}_{n,zz}z^{2}\widetilde{\eta}^{2}_{n,x}}{(1+\widetilde{\eta}_{n})}$ converges to $\displaystyle\frac{\widetilde{\bf w}_{zz}z^{2}\widetilde{\eta}^{2}_{x}}{(1+\widetilde{\eta})}$ weakly in $L^{2}(0,T;L^{1}(\Omega))$ and hence particularly in the space $\mathcal{D}'(Q_{T}).$\\
			 	Now we consider the term $P'\widetilde{\sigma}_{n,z}z\widetilde{\eta}_{n,x}\vec{e}_{1}=(\widetilde{\sigma}_{n}+\overline{\rho})^{\gamma-1}\widetilde{\sigma}_{n,z}z\widetilde{\eta}_{n,x}\vec{e}_{1}.$ Since $\|(\widetilde{\sigma}_{n}+\overline{\rho})\|_{C^{0}(0,T;H^{2}(\Omega))}$ is uniformly bounded so is $\|(\widetilde{\sigma}_{n}+\overline{\rho})^{\gamma-1}\|_{C^{0}(0,T;H^{2}(\Omega))}$ and hence $(\widetilde{\sigma}_{n}+\overline{\rho})^{\gamma-1}$ converges weakly to $(\widetilde{\sigma}+\overline{\rho})^{\gamma-1}$ in $L^{2}(0,T;H^{2}(\Omega)).$ We also have that $\widetilde{\sigma}_{n,z}$ converges strongly to $\widetilde{\sigma}_{z}$ in $C^{0}([0,T];L^{2}(\Omega)).$ Hence $(\widetilde{\sigma}_{n}+\overline{\rho})^{\gamma-1}\widetilde{\sigma}_{n,z}$ converges weakly to $(\widetilde{\sigma}+\overline{\rho})^{\gamma-1}\widetilde{\sigma}_{z}$ in $L^{2}(0,T;L^{2}(\Omega)).$ Now the strong convergence of $\widetilde{\eta}_{n,x}$ to $\widetilde{\eta}_{x}$ in $C^{0}([0,T];H^{1}(\Gamma_{s}))$ furnish that $(\widetilde{\sigma}_{n}+\overline{\rho})^{\gamma-1}\widetilde{\sigma}_{n,z}z\widetilde{\eta}_{n,x}$ weakly converges to $(\widetilde{\sigma}+\overline{\rho})^{\gamma-1}\widetilde{\sigma}_{z}z\widetilde{\eta}_{x}$ in $L^{2}(0,T;L^{1}(\Omega)).$ Hence $(\widetilde{\sigma}_{n}+\overline{\rho})^{\gamma-1}\widetilde{\sigma}_{n,z}z\widetilde{\eta}_{n,x}\vec{e}_{1}$ converges to $(\widetilde{\sigma}+\overline{\rho})^{\gamma-1}\widetilde{\sigma}_{z}z\widetilde{\eta}_{x}\vec{e}_{1}$ in the space $\mathcal{D}'(Q_{T}).$\\
			 	We can apply similar line of arguments to prove that $G_{2}(\widetilde{\sigma}_{n},\widetilde{\bf w}_{n},\widetilde{\eta}_{n})$ converges to $G_{2}(\widetilde{\sigma},\widetilde{\bf w},\widetilde{\eta})$ in $\mathcal{D}'(Q_{T}).$ Hence we have proved that  $G_{2}(\widetilde{\sigma}_{n},\widetilde{\bf w}_{n},\widetilde{\eta}_{n})$ converges to $G_{2}(\widetilde{\sigma},\widetilde{\bf w},\widetilde{\eta})$ weakly in $L^{2}(0,T;{\bf L}^{2}(\Omega)).$
			 	\newline
			 	Also observe that $(\widetilde{\sigma}_{n}+\overline{\rho})$ converges strongly to $(\widetilde{\sigma}+\overline{\rho})$ in $C^{0}([0,T];H^{1}(\Omega))$ and $\widehat{\bf w}_{n,t},$ $(-\mu\Delta-(\mu'+\mu)\nabla(\mbox{div}))\widehat{\bf w}_{n}$ converge up to a subsequence weakly to $\displaystyle\overline{\bf w}_{t}$ and $\displaystyle(-\mu\Delta-(\mu'+\mu)\nabla(\mbox{div}))\overline{\bf w}$ respectively in the spaces $L^{2}(0,T;{\bf H}^{2}(\Omega))$ and $L^{2}(0,T;{\bf H}^{1}(\Omega)).$ Hence up to a subsequence one obtains in particular the following convergence
			 	$$(\widetilde\sigma_{n}+\overline{\rho})\widehat{\bf w}_{n,t}-\mu\Delta\widehat{\bf w}_{n}-(\mu'+\mu)\nabla(\mbox{div}\widehat {\bf w}_{n})\rightharpoonup (\widetilde\sigma+\overline{\rho})\overline{\bf w}_{t}-\mu\Delta\overline{\bf w}-(\mu'+\mu)\nabla(\mbox{div}\overline{\bf w})\quad\mbox{in}\quad L^{2}(0,T;{\bf L}^{2}(\Omega)).$$ 
			 	Now consider \eqref{3.2}$_{2}$ with $(\widetilde{\sigma},\widetilde{\bf w},\widetilde{\eta})$ and ${\bf w}$ replaced respectively by $(\widetilde{\sigma}_{n},\widetilde{\bf w}_{n},\widetilde{\eta}_{n})$ and $\widehat{\bf w}_{n}.$ The weak convergences discussed so far allow to pass to the limits in both sides of this equation. So using the uniqueness of weak solution for the linear problem \eqref{2.1.1} we conclude that $\overline{\bf w}=\widehat{\bf w}.$\\
			 	$(ii)$ Now we claim that $G_{1}(\widetilde{\sigma}_{n},\widetilde{\bf w}_{n},\widetilde{\eta}_{n})$ converges weakly to $G_{1}(\widetilde{\sigma},\widetilde{\bf w},\widetilde{\eta})$ in $L^{2}(0,T;L^{2}(\Omega)).$\\ 
			 	Let us consider the term $\displaystyle\frac{1}{(1+\widetilde{\eta}_{n})}(\widetilde{ w}_{n})_{1,z}z\widetilde{\eta}_{n,x}(\widetilde{\sigma}_{n}+\overline{\rho}).$ We already know that $\displaystyle\frac{1}{(1+\widetilde{\eta}_{n})}$ and $\widetilde{\eta}_{n,x}$ converge strongly to $\displaystyle\frac{1}{(1+\widetilde{\eta})}$ and $\widetilde{\eta}_{x}$ respectively in the spaces $C^{0}([0,T];H^{2}(\Gamma_{s})$ and $C^{0}([0,T];H^{1}(\Gamma_{s})).$ One also observes that $(\widetilde{w}_{n})_{1,z}$ weakly converges to $\widetilde{w}_{1,z}$ in $L^{2}(0,T;{H}^{2}(\Omega))$ (since $\widetilde{\bf w}_{n}\rightharpoonup \widetilde{\bf w}$ in $L^{2}(0,T;{\bf H}^{3}(\Omega))$). Finally the strong convergence of $(\widetilde{\sigma}_{n}+\overline{\rho})$ to $(\widetilde{\sigma}+\overline{\rho})$ in $C^{0}([0,T];H^{1}(\Omega))$ furnish the weak convergence of $\displaystyle\frac{1}{(1+\widetilde{\eta}_{n})}(\widetilde{ w}_{n})_{1,z}z\widetilde{\eta}_{n,x}(\widetilde{\sigma}_{n}+\overline{\rho})$ to $\displaystyle\frac{1}{(1+\widetilde{\eta})}(\widetilde{ w})_{1,z}z\widetilde{\eta}_{x}(\widetilde{\sigma}+\overline{\rho})$ in $L^{2}(0,T;L^{2}(\Omega)).$ We can apply similar arguments for other terms in the expression of $G_{1}(\widetilde{\sigma},\widetilde{\bf w},\widetilde{\eta})$ in order to prove the weak convergence of $G_{1}(\widetilde{\sigma}_{n},\widetilde{\bf w}_{n},\widetilde{\eta}_{n})$ to $G_{1}(\widetilde{\sigma},\widetilde{\bf w},\widetilde{\eta})$ in $L^{2}(0,T;L^{2}(\Omega)).$\\
			 	 We further observe that $\nabla\widehat{\sigma}_{n}$ strongly converges to $\nabla\overline{\sigma}$ in $C^{0}([0,T];L^{2}(\Omega)).$ Since $(\widetilde{ w}_{n})_{1}$ weakly converges to $\widetilde{w}_{1}$ in $L^{2}(0,T;H^{3}(\Omega)),$ $(\widetilde{\eta}_{n})_{x}$ strongly converges to $\widetilde{\eta}_{x}$ in $L^{\infty}(\Sigma^{s}_{T})$ (because $(\widetilde{\eta}_{n})_{x}$ strongly converges to $\widetilde{\eta}_{x}$ in $C^{0}([0,T];H^{1}(\Gamma_{s}))$ and the embedding $H^{1}(\Gamma_{s})\hookrightarrow L^{\infty}(\Gamma_{s})$ is continuous) and $\frac{1}{(1+\widetilde{\eta}_{n})}$ strongly converges to $\frac{1}{(1+\widetilde{\eta})}$ in $C^{0}([0,T];H^{2}(\Gamma_{s}))$, $\frac{1}{(1+\widetilde{\eta}_{n})}(\widetilde{w}_{n})_{1}z(\widetilde{\eta}_{n})_{x}(\widehat{\sigma}_{n})_{z}$ weakly converges to $\frac{1}{(1+\widetilde{\eta})}\widetilde{w}_{1}z\widetilde{\eta}_{x}\widehat{\sigma}_{z}$ in $L^{2}(0,T;L^{2}(\Omega)).$ Besides, up to a subsequence $(\widehat{\sigma}_{n})_{t}$ weakly converges to $\overline{\sigma}_{t}$ in $L^{2}(0,T;L^{2}(\Omega)).$ Hence up to a subsequence we have 
			    $$(\widehat{\sigma}_{n})_{t}+\begin{bmatrix}
			    (\widetilde{w}_{n})_{1}\\
			    \frac{1}{(1+\widetilde{\eta}_{n})}((\widetilde{w}_{n})_{2}-(\widetilde{ w}_{n})_{1}z(\widetilde{\eta}_{n})_{x})
			    \end{bmatrix}\cdot\nabla\widehat{\sigma}_{n}\rightharpoonup {\overline{\sigma}_{t}}+\begin{bmatrix}
			    \widetilde{ w}_{1}\\
			    \frac{1}{(1+\widetilde{\eta})}(\widetilde{w}_{2}-\widetilde{ w}_{1}z\widetilde\eta_{x})
			    \end{bmatrix}
			    \cdot \nabla\overline{\sigma}\quad\mbox{in}\quad L^{2}(0,T;L^{2}(\Omega)).$$
			 	Now consider \eqref{3.2}$_{1}$ with $(\widetilde{\sigma},\widetilde{\bf w},\widetilde{\eta})$ and ${\sigma}$ replaced respectively by $(\widetilde{\sigma}_{n},\widetilde{\bf w}_{n},\widetilde{\eta}_{n})$ and $\widehat{\sigma}_{n}.$ The weak type convergences discussed so far allow to pass to the limits in both sides of this equation. Hence from uniqueness of weak solution of the linear problem \eqref{2.2.1} we conclude that $\overline{\sigma}=\widehat{\sigma}.$\\
			 	$(iii)$ One can use similar line of arguments as used so far to show that $G_{3}{(\widetilde{\sigma}_{n},\widetilde{\bf w}_{n},\widetilde{\eta}_{n})}$ converges weakly to $G_{3}{(\widetilde{\sigma},\widetilde{\bf w},\widetilde{\eta})}$ in $L^{2}(0,T;L^{2}(\Gamma_{s})).$ Using the norm bounds of $\widehat{\eta}_{n}$ (since $(\widehat{\sigma}_{n},\widehat{\bf w}_{n},\widehat{\eta}_{n})\in \mathscr{C}_{T}$) we can prove that up to a subsequence the left hand side of \eqref{3.2}$_{6}$ with $\eta$ replaced by $\widehat{\eta}_{n}$ converges weakly to $$\displaystyle\overline{\eta}_{tt}-\beta \overline{\eta}_{xx}-  \delta\overline{\eta}_{txx}+\alpha\overline{\eta}_{xxxx}$$
			 	in $L^{2}(0,T;L^{2}(\Gamma_{s})).$ Now the uniqueness of weak solution to the problem \eqref{2.3.1} furnishes $\overline{\eta}=\widehat{\eta}.$
			 	Hence the proof of Lemma \ref{cont} is complete.
			 	\end{proof}
			 	\subsection{Conclusion}\label{conc}
			    The following properties hold\\
			 	(i) The convex set $\mathscr{C}_{T}$ is non-empty (Lemma \ref{nonempty}) and is a compact subset of $\mathcal{X}$ (Lemma \ref{com}).\\
			 	(ii) The map $L$, defined in \eqref{welposL}, is continuous on $\mathscr{C}_{T}$ in the topology of $\mathcal{X}$ (Lemma \ref{cont}).
				\\
				(iii) The map $L$ maps $\mathscr{C}_{T}$ to itself (Lemma \ref{inclusion}).\\ 
			 	Thus, all the assumptions of Schauder fixed point theorem are satisfied by the map $L$ on $\mathscr{C}_{T},$ endowed with the topology of $\mathcal{X}.$ Therefore, Schauder fixed point theorem yields a fixed point $(\sigma_{f},{\bf w}_{f},\eta_{f})$ of the map $L$ in $\mathscr{C}_{T}.$ From the definition of the map $L$, one has $(\sigma_{f},{\bf w}_{f},\eta_{f})\in Z^{T}_{1}\times Y^{T}_{2}\times Z^{T}_{3}.$ Hence we have the following time continuities (since still now one only has the regularities \eqref{eg1} of $G_{1}(\sigma_{f},{\bf w}_{f},\eta_{f}),$ \eqref{eg2} of $G_{2}(\sigma_{f},{\bf w}_{f},\eta_{f})$ and \eqref{eg3} of $G_{3}(\sigma_{f},{\bf w}_{f},\eta_{f})$)
			 	\begin{equation}\label{timecont}
			 	\begin{array}{l}
			 	{\sigma}_{f}\in C^{0}([0,T];H^{2}(\Omega)),\\
			 	{\bf w}_{f}\in C^{0}([0,T];{\bf H}^{5/2}(\Omega))\cap C^{1}([0,T];{\bf H}^{1}(\Omega)),\\
			 	\eta_{f} \in C^{0}([0,T];H^{4}(\Gamma_{s}))\cap C^{1}([0,T];H^{3}(\Gamma_{s}))\cap C^{2}([0,T];H^{1}(\Gamma_{s})).
			 	\end{array}
			 	\end{equation}
			 	The regularities \eqref{timecont} can be used to further check that $G_{1}(\sigma_{f},{\bf w}_{f},\eta_{f})\in C^{0}([0,T];H^{1}(\Omega))$ and $G_{3}(\sigma_{f},{\bf w}_{f},\eta_{f})\in C^{0}([0,T];H^{1/2}(\Gamma_{s})).$ Hence we use Corollary \ref{dencor} and the Corollary \ref{timebeam} to obtain the following
			 	\begin{equation}\nonumber
			 	\begin{array}{l}
			 	(\sigma_{f})_{t}\in C^{0}([0,T];H^{1}(\Omega))\,\,\mbox{and}\,\, \eta_{f}\in C^{0}([0,T];H^{9/2}(\Gamma_{s})).
			 	\end{array}
			 	\end{equation} 
			 	 Hence, $(\sigma_{f},{\bf w}_{f},\eta_{f})\in Y^{T}_{1}\times Y^{T}_{2}\times Y^{T}_{3}.$ The trajectory $(\sigma_{f},{\bf w}_{f},\eta_{f})$ solves the nonlinear problem \eqref{chdb} in $Y^{T}_{1}\times Y^{T}_{2}\times Y^{T}_{3}.$ Consequently the system \eqref{1.21} admits a solution. This further implies that the original system \eqref{1.1}-\eqref{1.2}-\eqref{1.3} admits a strong solution in sense of the Definition \ref{doss}. Finally the proof of Theorem \ref{main} is complete.
			 	\end{proof}

\bibliographystyle{plain}
%\bibliography{bibliography}

%The uniform bounds of $\|\widetilde{\eta}_{n,tt}\|_{L^{\infty}(0,T;H^{1}(\Gamma_{s}))}$ and $\|\widetilde{\eta}_{n,ttt}\|_{L^{2}(0,T;L^{2}(\Gamma_{s}))}$ and Aubin Lions lemma (\cite{aubin}) furnish that up to a subsequence (still denoted by $\widetilde{\eta}_{n}$) $\widetilde{\eta}_{n}$ strongly converges to $\overline{\eta}$ in $C^{0}([0,T];L^{2}(\Gamma_{s})).$ Hence
%\begin{equation}\label{inoeta}
%\begin{array}{l}
%\overline{\eta}_{tt}(\cdot,0)=\delta\eta_{1,xx}-(\mu+2\mu')(u_{0})_{2,z}+P(\rho_{0}).
%\end{array}
%\end{equation}
\end{document}